\documentclass[a4paper,11pt]{amsart}
\title{Gopakumar-Vafa invariants via vanishing cycles}
\date{}
\author{Davesh Maulik and Yukinobu Toda}

\usepackage{fancybox}
\usepackage{amscd}
\usepackage{amsmath}
\usepackage{amssymb}
\usepackage{amsthm}
\usepackage{float}
\usepackage[dvips]{graphicx}
\usepackage{tikz}
\usepackage[all,ps,dvips]{xy}

\usepackage{color}

\usepackage{array}
\usepackage{amscd}

\DeclareFontFamily{U}{rsfs}{%
\skewchar\font127}
\DeclareFontShape{U}{rsfs}{m}{n}{%
<-6>rsfs5<6-8.5>rsfs7<8.5->rsfs10}{}
\DeclareSymbolFont{rsfs}{U}{rsfs}{m}{n}
\DeclareSymbolFontAlphabet
{\mathrsfs}{rsfs}
\DeclareRobustCommand*\rsfs{%
\@fontswitch\relax\mathrsfs}

\theoremstyle{plain}
\newtheorem{thm}{Theorem}[section]
\newtheorem{prop}[thm]{Proposition}
\newtheorem{lem}[thm]{Lemma}

\newtheorem{defi}[thm]{Definition}
\newtheorem{rmk}[thm]{Remark}
\newtheorem{cor}[thm]{Corollary}

\newtheorem{prop-defi}[thm]{Proposition-Definition}
\newtheorem{thm-defi}[thm]{Theorem-Definition}
\newtheorem{lem-defi}[thm]{Lemma-Definition}

\newtheorem{conj}[thm]{Conjecture}

\newdimen\argwidth
\def\db[#1\db]{
 \setbox0=\hbox{$#1$}\argwidth=\wd0
 \setbox0=\hbox{$\left[\box0\right]$}
  \advance\argwidth by -\wd0
 \left[\kern.3\argwidth\box0 \kern.3\argwidth\right]}

\newcommand{\cC}{\mathcal{C}}

\newcommand{\eE}{\mathcal{E}}
\newcommand{\fF}{\mathcal{F}}
\newcommand{\gG}{\mathcal{G}}
\newcommand{\hH}{\mathcal{H}}
\newcommand{\iI}{\mathcal{I}}

\newcommand{\lL}{\mathcal{L}}
\newcommand{\mM}{\mathcal{M}}

\newcommand{\oO}{\mathcal{O}}
\newcommand{\pP}{\mathcal{P}}

\newcommand{\sS}{\mathcal{S}}

\newcommand{\uU}{\mathcal{U}}
\newcommand{\vV}{\mathcal{V}}

\newcommand{\zZ}{\mathcal{Z}}

\newcommand{\bL}{\mathbb{L}}
\newcommand{\wH}{\widehat{H}}

\newcommand{\Hom}{\mathop{\rm Hom}\nolimits}

\newcommand{\dotimes}{\stackrel{\textbf{L}}{\otimes}}
\newcommand{\dR}{\mathbf{R}}

\newcommand{\Hilb}{\mathop{\rm Hilb}\nolimits}

\newcommand{\Pic}{\mathop{\rm Pic}\nolimits}

\newcommand{\Chow}{\mathop{\rm Chow}\nolimits}
\newcommand{\GV}{\mathop{\rm GV}\nolimits}
\newcommand{\GVloc}{\mathop{\rm GV}^{\rm{loc}}\nolimits}

\newcommand{\Sh}{\mathop{\rm Sh}\nolimits}

\newcommand{\Obs}{\mathop{\rm Obs}\nolimits}

\newcommand{\id}{\textrm{id}}

\newcommand{\Ext}{\mathop{\rm Ext}\nolimits}
\newcommand{\Spec}{\mathop{\rm Spec}\nolimits}

\newcommand{\Coh}{\mathop{\rm Coh}\nolimits}

\newcommand{\cneq}{\mathrel{\raise.095ex\hbox{:}\mkern-4.2mu=}}
\newcommand{\eqcn}{\mathrel{=\mkern-4.5mu\raise.095ex\hbox{:}}}

\newcommand{\gr}{\mathop{\rm gr}\nolimits}

\newcommand{\Perv}{\mathop{\rm Perv}\nolimits}

\newcommand{\Ex}{\mathop{\rm Ex}\nolimits}

\newcommand{\oPPer}{\mathop{\rm ^{0}Per}\nolimits}

\newcommand{\iPPer}{\mathop{\rm ^{-1}Per}\nolimits}

\newcommand{\pH}{\mathop{^{p}\mathcal{H}}\nolimits}

\newcommand{\pF}{\mathop{^{{p}}\mathcal{F}}\nolimits}
\newcommand{\oF}{\mathop{^{{0}}\mathcal{F}}\nolimits}
\newcommand{\iF}{\mathop{^{{-1}}\mathcal{F}}\nolimits}
\newcommand{\pT}{\mathop{^{{p}}\mathcal{T}}\nolimits}

\newcommand{\pPPer}{\mathop{\rm ^{\mathit{p}}Per}\nolimits}

\newcommand{\PT}{\mathop{\rm PT}\nolimits}
\newcommand{\IC}{\mathop{\rm IC}\nolimits}

\newcommand{\Sym}{\mathop{\rm Sym}\nolimits}

\newcommand{\Imm}{\mathop{\rm Im}\nolimits}

\newcommand{\Ker}{\mathop{\rm Ker}\nolimits}

\newcommand{\RHom}{\mathop{\dR\mathrm{Hom}}\nolimits}

\makeatletter
 
  \@addtoreset{equation}{section}
\makeatother

\begin{document}
\maketitle

\begin{abstract}
In this paper, we propose an ansatz for defining Gopakumar-Vafa invariants of Calabi-Yau threefolds, using perverse sheaves of vanishing cycles.  
Our proposal is a modification 
of a recent approach of Kiem-Li, 
which is itself based on earlier ideas of Hosono-Saito-Takahashi.   
We conjecture that
 these invariants are equivalent
to other curve-counting theories such as 
Gromov-Witten theory and Pandharipande-Thomas theory.  

Our main theorem is that, for local surfaces, our invariants agree with PT invariants for irreducible one-cycles. 
We also give a counter-example to the Kiem-Li conjectures, where our invariants match the predicted answer.  Finally, we give examples where
our invariant matches the expected answer in cases where the cycle is non-reduced, non-planar, or non-primitive.
\end{abstract}

\setcounter{tocdepth}{1}
\tableofcontents

\section{Introduction}
\subsection{Background}
Let $X$ be a smooth projective Calabi-Yau 
threefold over $\mathbb{C}$. 
For $g\ge 0$ and $\beta \in H_2(X, \mathbb{Z})$, 
the corresponding \textit{Gromov-Witten invariant} 
\begin{align*}
\mathrm{GW}_{g, \beta}
=\int_{[\overline{M}_g(X, \beta)]^{\rm{vir}}} 1 
 \in \mathbb{Q}
\end{align*}
enumerates 
stable maps $f \colon C \to X$
from connected, nodal curves 
$C$ of 
arithmetic genus $g$ such that $f_{\ast}[C]=\beta$. 
In general, these invariants are given by
an infinite sequence of rational numbers; nevertheless, 
for fixed $\beta$,
they are expected to be controlled by a finite collection of
integer invariants.  Indeed,
based on the string duality between type IIA 
and M theory, 
 Gopakumar-Vafa~\cite{GV} 
conjectured the existence of
integer-valued invariants
\begin{align}\label{intro:ng}
n_{g, \beta} \in \mathbb{Z}, \ 
g\ge 0, \ \beta \in H_2(X, \mathbb{Z})
\end{align}
which vanish for sufficiently large $g$
and which determine the Gromov-Witten series by the identity
\begin{align}\label{intro:GW/GV}
\sum_{\beta>0, g\ge 0}
\mathrm{GW}_{g, \beta}
\lambda^{2g-2}t^{\beta}
=\sum_{\beta>0, g\ge 0, k\ge 1}
\frac{n_{g, \beta}}{k}
\left(2\sin\left( \frac{k\lambda}{2} \right)\right)^{2g-2}
t^{k\beta}. 
\end{align}
The invariants (\ref{intro:ng})
are called \textit{Gopakumar-Vafa (GV for short) invariants}. 

In order to define these invariants directly\footnote{
One can take equation \eqref{intro:GW/GV} as an indirect definition of GV invariants, in which case integrality and vanishing 
become conjectures.  A symplectic approach to proving the integrality is pursued in \cite{Ionel-Parker}.}, 
the original
approach of Gopakumar-Vafa was to use the
$sl_2 \times sl_2$-action 
on the cohomology of a certain moduli space
of D-branes, which should be given 
by a moduli space of one-dimensional sheaves.
The goal of this paper is to propose
an ansatz for making this mathematically precise.  As we will review shortly,
there have been earlier efforts in this direction, most notably by Hosono-Saito-Takahashi~\cite{HST} and Kiem-Li~\cite{KL}.
Our approach is a modification of the recent Kiem-Li proposal via perverse sheaves of vanishing cycles, where 
we use the perverse filtration for the Hilbert-Chow map instead of the action of $sl_2 \times sl_2$.
We then show that our definition of GV invariants 
matches with stable pair invariants introduced by 
Pandharipande-Thomas~\cite{PT}
in several cases, in particular for irreducible one-cycles on local surfaces.

\subsection{Proposed definition}
Let
$\Sh_{\beta}(X)$ denote the moduli space of one-dimensional 
stable sheaves $E$ on $X$ 
satisfying\footnote{The case for other choices of $\chi(E)$
 will be considered in Subsection~\ref{subsec:chi}.}
\begin{align*}
[E]=\beta \in H_2(X, \mathbb{Z}), \ \chi(E)=1.
\end{align*}
Let $\Chow_{\beta}(X)$ denote the Chow 
variety parameterizing effective one-cycles on $X$ with 
homology class $\beta$. 
There is a Hilbert-Chow map
\begin{align}\label{intro:HC}
\pi \colon \Sh_{\beta}^{\rm{red}}(X) \to \Chow_{\beta}(X)
\end{align}
sending 
a stable sheaf to its fundamental one-cycle.
In~\cite{BDJS, KL}, a certain 
perverse sheaf $\phi_{\sS h}$ on $\Sh_{\beta}(X)$
is constructed whose Euler characteristic recovers the usual invariants
of \textit{Donaldson-Thomas (DT for short) theory}~\cite{Thom}. 
Roughly speaking, the moduli space $\Sh_{\beta}(X)$ is 
locally written as a critical locus of 
some function on a smooth scheme, 
and $\phi_{\sS h}$ is obtained by gluing together the locally-defined
sheaves of vanishing cycles.  An important subtlety in this construction is that the gluing is not uniquely 
determined and depends on a choice of \textit{orientation data}, 
that is a square root of the virtual canonical 
line bundle on $\Sh_{\beta}(X)$. 

We propose the following definition of 
GV invariants: 
\begin{defi}\emph{(Definition~\ref{def:GV})}
\label{intro:defi}
We define the GV invariants 
$n_{g, \beta}$ 
by the identity
\begin{align}\label{intro:defGV}
\sum_{i \in \mathbb{Z}} 
\chi(\pH^i(\dR \pi_{\ast}\phi_{\sS h}))y^i=
\sum_{g\ge 0}n_{g, \beta}(y^{\frac{1}{2}}+y^{-\frac{1}{2}})^{2g}. 
\end{align}
\end{defi}
Here 
$\pH^i(-)$ is the $i$-th cohomology functor
with respect to the perverse t-structure.
By the self-duality of $\phi_{\sS h}$ and the Verdier duality, 
the LHS of (\ref{intro:defGV}) is uniquely written as the 
form of the RHS.  
The GV invariants in Definition~\ref{intro:defi}
are obviously integers, which vanish for sufficiently large $g$. 

As currently formulated
our GV invariants in (\ref{intro:defGV})
 depend on a choice of an orientation, 
and a canonical choice is not known. 
We impose an additional restriction
that the orientation data is trivial along the fibers
of  (\ref{intro:HC}).  These are denoted \textit{Calabi-Yau}
orientations, and we conjecture that such choices always exist.
 As we will see, the invariants $n_{g, \beta}$ defined by (\ref{intro:defGV})
are independent of the choice of CY orientation data.

Our definition of GV invariants 
in Definition~\ref{intro:defi}
is related to the character formula 
of the $sl_2 \times sl_2$-actions 
in earlier approaches~\cite{HST, KL}, 
see Subsection~\ref{subsec:previous}.  As far as we know,
the observation that the character formula can be reformulated via perverse cohomology
goes back to work of Chuang-Diaconescu-Pan~\cite{CDP2} on the Hitchin fibration and the survey article
of Pandharipande-Thomas~\cite{PT13} in their discussion of \cite{HST}.

One advantage of this reformulation
is that it naturally 
lifts to a definition of \textit{local GV invariants}, i.e. 
we can define a constructible function (see Definition~\ref{def:locgv})
\begin{align}\label{intro:locgv}
n_{g, -}^{\rm{loc}} \colon \Chow_{\beta}(X) \to \mathbb{Z}
\end{align}
whose integral
over the Chow variety gives $n_{g, \beta}$. 
This makes it possible to compare 
our GV invariants with 
stable pair invariants 
for a fixed one-cycle. 

\subsection{PT/GV correspondence for irreducible cycles}

Given $X$ as before, a
\textit{stable pair} on $X$ consists of 
a pair 
\begin{align*}
(F, s), \ 
s \colon \oO_X \to F
\end{align*}
such that $F$ is a pure one-dimensional sheaf on $X$, 
and $s$ is a morphism whose 
cokernel is at most zero-dimensional. 
Virtual invariants for stable pair spaces were
introduced by Pandharipande-Thomas (PT for short) in~\cite{PT}. 
 In~\cite{PT3}, \textit{local PT invariants} are defined by taking the 
constructible function
\begin{align}\label{intro:locpt}
P_{n, -}^{\rm{loc}} \colon \Chow_{\beta}(X) \to \mathbb{Z}
\end{align}
obtained by integrating the Behrend function~\cite{Beh}
over the locus of the moduli 
space of stable pairs with fixed fundamental cycle. 
The usual \textit{PT invariant} $P_{n, \beta} \in \mathbb{Z}$
is determined from the local invariants by integrating over the Chow variety; when
$X$ is projective, this agrees with the original definition
by virtual classes~\cite{Beh}.

For any threefold, there is a conjectural
equivalence~\cite{MNOP,PT} between
GW and PT invariants, which is proved in many cases~\cite{MNOP, BrPa, PP}. 
In particular, formula (\ref{intro:GW/GV})
implies a conjectural equivalence between PT invariants 
and our GV invariants.
Furthermore, since both PT invariants and our GV invariants
can be refined to local invariants (\ref{intro:locgv}), (\ref{intro:locpt}), 
we can conjecture a local PT/GV correspondence
(see Figure~1). For an irreducible one-cycle, 
it is given as follows:
\begin{conj}\label{intro:conj}
For an irreducible one-cycle
$\gamma \in \Chow_{\beta}(X)$, we have the identity
\begin{align*}
\sum_{n\in \mathbb{Z}}P_{n, \gamma}^{\rm{loc}}q^n=
\sum_{g\ge 0} n_{g, \gamma}^{\rm{loc}}
(q^{\frac{1}{2}}+q^{-\frac{1}{2}})^{2g-2}. 
\end{align*}
\end{conj}
We will formulate the conjecture for general one-cycles in~Conjecture~\ref{GV:conj2};
it requires contributions coming from effective summands of the cycle $\gamma$.

The first main result of this paper is to 
prove Conjecture~\ref{intro:conj} for local surfaces. 
\begin{thm}\emph{(Theorem~\ref{thm:locsur})}
\label{intro:thm1}
Let $S$ be a smooth projective surface 
with $H^1(\oO_S)=0$, and 
\begin{align*}
p \colon X=\mathrm{Tot}(K_S) \to S
\end{align*}
the non-compact CY 3-fold. 
Then Conjecture~\ref{intro:conj} is true for 
any one-cycle $\gamma$ on $X$ 
such that $p_{\ast}\gamma$ is irreducible. 
\end{thm}

We also prove Conjecture~\ref{intro:conj}
for smooth one-cycles on an arbitrary Calabi-Yau threefold $X$.

\begin{thm}\emph{(Theorem~\ref{thm:smooth})}
\label{intro:thm2}
For a smooth curve $C \subset X$, 
Conjecture~\ref{intro:conj} is true 
for the one cycle $\gamma=[C]$. 
\end{thm}

The situations in Theorem~\ref{intro:thm1}, \ref{intro:thm2}
include many cases where both 
$\Sh_{\beta}(X)$ and $\Chow_{\beta}(X)$ are singular. 
In particular, in these cases, the sheaf $\phi_{\sS h}$ will typically not be pure and the Behrend function will not be constant.
The main idea to prove the above theorems is to use
the vanishing cycles functor to reduce to the generalized Macdonald formula 
for versal deformations of locally planar curves, proved in~\cite{MY2, MS}.
The basic outline of this reduction is explained in Section~\ref{sec:PT/GV},
which we recommend for readers who just want to see the essential concept without
the (many) technical details.

\subsection{Examples for non-reduced cycles}

One limitation of the theorems above is they only apply in cases where the one-cycle is integral.
In Section~\ref{sec:non-red}, 
we produce an infinite family of examples where the local PT/GV 
correspondence holds for 
one-cycles that are non-reduced or non-planar.

The idea for the construction is
to study how our invariants behave for certain 3-fold flops and
combine this analysis with Theorem~\ref{intro:thm1}, \ref{intro:thm2}.
We will show the following result:
\begin{thm}\emph{(Corollary~\ref{cor:ng2})}
Let $\phi \colon X \dashrightarrow X^{\dag}$ be a flop between CY 3-folds 
and $C \subset X$ an irreducible curve which is not contained in $\Ex(\phi)$. 
Suppose that Conjecture~\ref{intro:conj}
holds for $\gamma=[C]$. Then the local PT/GV correspondence holds for 
$\phi_{\ast}\gamma$.
\end{thm}

In typical examples, the flopped cycle $\phi_{\ast}\gamma$ can be arranged to be non-reduced and non-planar.
In such cases,
the correspondence requires the more general formulation of
Conjecture~\ref{GV:conj2}.  
In the examples obtained above, the contributions from effective summands of the reducible one-cycle will typically be nonzero.  Furthermore, we can iterate the above theorem to obtain
more complicated one-cycles where our conjecture holds.

In Section~\ref{sec:non-prim}, we also give examples of the local PT/GV correspondence where the one-cycle is non-primitive.
One interesting family of such examples is due to 
Chuang-Diaconescu-Pan~\cite{CDP2}, when $X$ is the total space of $K_C \oplus \mathcal{O}_C$ over a curve $C$ of genus $g$.
In this case, the conjecture for $\beta = r[C]$ is a consequence of the $P=W$ conjecture for Higgs bundles of rank $r$ on the curve $C$; since
the $P=W$ conjecture is proven in rank $2$, this gives an infinite family of non-primitive examples.

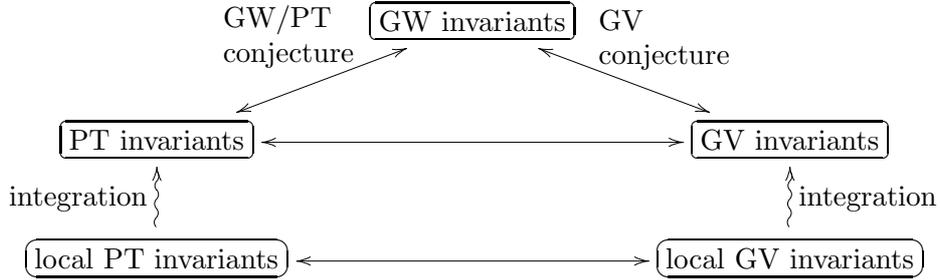
\begin{figure}\label{fig:intro}
\begin{align*}
\xymatrix{
& \ar@{<->}[ld]_-{\begin{array}{l}\mbox{GW/PT} \\
\mbox{conjecture}
\end{array}}
  \ovalbox{\mbox{GW invariants}}  \ar@{<->}[rd]^{\begin{array}{l}
\mbox{GV} \\
\mbox{conjecture}
\end{array}} & \\
\ovalbox{\mbox{PT invariants}} \ar@{<->}[rr] &  &  \ovalbox{\mbox{GV invariants}} \\
\ovalbox{\mbox{local PT invariants}} \ar@{~>}[u]^-{\mbox{integration}}
 \ar@{<->}[rr]^{}&  & 
 \ovalbox{\mbox{local GV invariants}}\ar@{~>}[u]_-{\mbox{integration}}
}
\end{align*}
\caption{GW/PT/GV correspondence}
\end{figure}

\subsection{Previous history}\label{subsec:previous}

A mathematical approach for
GV invariants was first proposed by 
Hosono-Saito-Takahashi (HST for short)~\cite{HST}. 
They considered the intersection complex
$\IC(\Sh_{\beta}(X))$ and applied
the decomposition theorem~\cite{BBD} 
with respect to the map (\ref{intro:HC})
to define the action of $sl_2 \times sl_2$ on the intersection 
cohomology
$IH^{\ast}(\Sh_{\beta}(X))$.
They
rearrange this as 
\begin{align}\notag
IH^{\ast}(\Sh_{\beta}(X))=\bigoplus_{g\ge 0}
I_g \otimes R_g
\end{align}
where $I_g$ is the 
cohomology of a $g$-dimensional complex torus 
with its natural left $sl_2$-action, and $R_g$ is a certain (virtual)
right
$sl_2$-representation. 
The
HST definition is 
given by the Euler characteristic of $R_g$. 
Since the IC-sheaf is not sensitive to the virtual structure of
$\Sh_{\beta}(X)$,
one can find examples where this approach does not match
the expected answer, even in genus $0$.  

In genus $0$,
Katz~\cite{Katz} proposed as a definition the virtual integral
\begin{align}\notag
n_{0, \beta}=\int_{[\Sh_{\beta}(X)]^{\rm{vir}}} 1 \in \mathbb{Z}. 
\end{align}
Since this can be written as a Behrend-weighted Euler characteristic, one can
use a wall-crossing argument to prove this is compatible with the GV invariants defined 
via stable pairs spaces.

More recently, Kiem-Li proposed a combination of these two approaches, using
the perverse sheaf $\phi_{\sS h}$ mentioned earlier.  One ambiguity in their proposal 
is they do not specify
how to choose an orientation on $\Sh_{\beta}(X)$ to define $\phi_{\sS h}$, and different choices
lead to different prescriptions. Once chosen, they consider
a natural lift of $\phi_{\sS h}$ to the category of mixed Hodge modules.  In order to apply the
decomposition theorem, they require a pure Hodge module, so take the associated graded 
$\gr_W^{\bullet}(\phi_{\sS h})$
with respect to the weight filtration, and
apply the formalism of \cite{HST} described earlier.
 In genus $0$, independent of orientation, this recovers Katz's definition.

\subsection{Counter-example to Kiem-Li conjecture}
We will give an example 
that the earlier definitions~\cite{HST, KL}
are not deformation-invariant and, in particular, do not
match the invariants as calculated via PT theory.  
Let $S$ be an Enriques surface and 
$E$ an elliptic curve. 
Let 
$\sigma \colon \widetilde{S} \to S$ be
its K3 cover. 
We take a CY 3-fold
\begin{align*}
X=(\widetilde{S} \times E)/\langle \tau \rangle.
\end{align*}
Here 
$\tau$ is 
an involution of $\widetilde{S} \times E$
which acts on $\widetilde{S}$ by 
the covering involution of $\sigma$
and acts on $E$ by 
$x \mapsto -x$. 
An Enriques surface $S$ always 
admits an elliptic fibration $S \to \mathbb{P}^1$, 
with a double fiber $2C$. 
We take a curve class $\beta$ in $X$ by
\begin{align*}
\beta=([C], 0) \in H_2(S, \mathbb{Z}) \oplus \mathbb{Z}[E]
=H_2(X, \mathbb{Z}). 
\end{align*}

\begin{prop}\emph{(Proposition~\ref{prop:counter2})}\label{prop:counter}
Suppose that $C$ is of type $I_n$ for $n\ge 2$, i.e. 
$C$ is a circle of $\mathbb{P}^1$ with $n$-irreducible components.  
Then the HST, KL, our definitions, 
and the expected 
answers (from GW or PT theory)
are given in the following table: 
\begin{align*}
\begin{tabular}{|c|c|c|c|c|c|c|} \hline
 & \rm{HST} & \rm{KL} & \rm{ours} & \rm{expected}  \\ \hline
$n_{0, \beta}$  & $-8n$  &  $0$ &  $0$ & $0$   \\ \hline
$n_{1, \beta}$  & $4n$  &  $4n$ &  $4$ & $4$  \\ \hline
$n_{\ge 2, \beta}$  & $0$  &  $0$ & $0$ & $0$ \\ \hline
\end{tabular}
\end{align*}
\end{prop}

The result of Proposition~\ref{prop:counter}
is a summary of the computations in Subsection~\ref{subsec:counter}. 
Since all Enriques surfaces
are deformation equivalent, 
the resulting invariants $n_{g, \beta}$ should 
be independent of the type of $C$. 
The result of $g=1$ for KL definition 
is true for any choice of orientation data, 
which does not match with the predicted answer. 
Therefore Proposition~\ref{prop:counter} gives a counter-example 
to the Kiem-Li conjecture~\cite[Conjecture~7.4]{KL}. 

One interesting feature of this example is that 
our invariants are preserved under deformations despite the
fact that the Chow variety itself jumps in dimension. 
Although our invariant in this example is deformation-invariant, we do not see a mechanism for this in general families of CY threefolds, due to our poor understanding of the Chow variety.  For this reason, our definition of GV invariants still may not be the final one and a better understanding of the Hilbert-Chow map and deformation invariance will be needed.

One can also 
ask what happens if we study the motivic vanishing cycles associated to 
$\Sh_{\beta}(X)$ in the sense of Bussi-Joyce-Meinhardt~\cite{BDM}, 
and define the motivic GV invariants as in~\cite{ToBPS}.  It turns 
out that the motivic GV invariants are different from Kiem-Li's invariants
due to some rearrangement of weights, but 
in any case they also do not give a correct answer.
An example already occurs in the case of the nodal rational curve in 
Section~\ref{ex:enriques}, where 
the genus one motivic invariant becomes zero.

\subsection{Outline of the paper}
In Section~\ref{sec:dcrit},
we review Joyce's notion of $d$-critical structures and 
introduce GV type invariants for $d$-critical schemes. 
In Section~\ref{sec:GVCY}, we 
define GV invariants on Calabi-Yau 3-folds and formulate
conjectures relating them with PT and GW invariants. 
In Section~\ref{sec:PT/GV}, we explain the idea proving 
PT/GV correspondence using versal deformations of curves. 
In Section~\ref{sec:local}, we prove Theorem~\ref{intro:thm1}.
In Section~\ref{sec:smooth}, we prove Theorem~\ref{intro:thm2}. 
In Section~\ref{sec:compare}, we prove Proposition~\ref{prop:counter}. 
In Section~\ref{sec:non-red}, we produce examples 
for non-reduced cycles using 3-fold flops. 
In Section~\ref{sec:non-prim}, we discuss some examples for non-primitive 
cycles. 
In Appendix~\ref{sec:append}, we discuss 
Calabi-Yau orientation data for $d$-critical schemes.

\subsection{Acknowledgments}
We are grateful to 
Tomoyuki Abe, 
Jim Bryan, Duiliu-Emanuel Diaconescu, 
Igor Dolgachev, 
Jesse Kass, 
Young-Hoon Kiem, 
Jun Li, 
Georg Oberdieck, 
Rahul Pandharipande, 
Christian Schnell and 
Richard Thomas
for many discussions and valuable comments. 
We are also grateful to the referees for many suggestions and comments. 
This article was written while 
Y.~T.~
was visiting Massachusetts Institute of Technology
from 2015 October to 2016 September
by the 
JSPS Program for Advancing Strategic International Networks to Accelerate the 
Circulation of Talented Researchers.  
D.~M.~is supported by 
NSF grants DMS-1645082 and DMS-1564458.
Y.~T.~is supported by World Premier 
International Research Center Initiative
(WPI initiative), MEXT, Japan, 
and Grant-in Aid
for Scientific Research grant (No.~26287002)
from MEXT, Japan.

\subsection{Notation and convention}
In this paper, all varieties and schemes 
are defined over $\mathbb{C}$. 
For a scheme $M$, we
will only consider constructible sheaves with $\mathbb{Q}$-coefficients. 
We denote by $\Perv(M)$ the category 
of perverse sheaves on $M$, which is the heart of 
a t-structure on the derived category of 
constructible sheaves on $M$ (see~\cite{BBD}). 
Let $\iota \colon M^{\rm{red}} \hookrightarrow M$ be the reduced part
of $M$. 
Since $\iota$ is a homeomorphism, 
we always identity $\Perv(M)$ with $\Perv(M^{\rm{red}})$
in a natural way. 

For a bounded 
complex $E$ of constructible sheaves on $M$, 
we denote by $\pH^i(E)$ the $i$-th cohomology with respect 
to the perverse t-structure, and $\chi(E)$ is the 
the Euler characteristic of $\dR \Gamma(M, E)$. 
For a constructible function $\nu$ on a scheme $M$, 
the weighted Euler characteristic is denoted by 
\begin{align*}
\int_M \nu \ de \cneq \sum_{m \in \mathbb{Z}}
m \cdot e(\nu^{-1}(m)). 
\end{align*}
Here $e(-)$ is the topological Euler characteristic. 
We will use the fact that, 
for a complex $E$ of constructible sheaves on a 
finite type scheme $M$, we have 
\begin{align*}
\chi(E)=\int_{M} \nu_E \ de
\end{align*}
where $\nu_E$ is the constructible function 
given by $p \mapsto \chi(E|_{p})$.

\section{GV type invariants for $d$-critical schemes}\label{sec:dcrit}

In this section, we review Joyce's work on (algebraic) 
$d$-critical structures and define GV-type
invariants in this setting\footnote{In~\cite{JoyceD}, Joyce also introduces an 
analytic version of $d$-critical structures; this is equivalent to the notion of virtual 
critical structures in~\cite{KL}.
Although we work with algebraic $d$-critical structures, 
the arguments of this section also apply for 
analytic $d$-critical structures.}.  We then introduce the notion of Calabi-Yau $d$-critical structures, which 
allows us to fix the ambiguity due to choice of orientation.

\subsection{$d$-critical schemes}
We recall the notion of $d$-critical structures
introduced by Joyce~\cite{JoyceD}. 
For any complex scheme $M$, 
Joyce~\cite{JoyceD} shows that 
there exists a canonical sheaf of 
$\mathbb{C}$-vector spaces 
$\sS_{M}$ on $M$
satisfying the following
property:
for any Zariski open subset $R \subset M$
and a closed embedding 
$i \colon R \hookrightarrow V$
into a smooth scheme $V$, 
there is an exact sequence
\begin{align}\label{S:property}
0 \longrightarrow \sS_{M}|_{R}
 \longrightarrow \oO_V/I^2 \stackrel{d_{\rm{DR}}}{\longrightarrow}
\Omega_V/I \cdot \Omega_V. 
\end{align}
Here $I \subset \oO_V$ is the ideal sheaf
which defines $R$
and $d_{\rm{DR}}$ is the de-Rham differential. 
Moreover there is a natural decomposition 
\begin{align*}
\sS_M=\sS_M^0 \oplus \mathbb{C}_M
\end{align*}
where $\mathbb{C}_M$ is the constant sheaf
on $M$. 
The sheaf $\sS_M^{0}$ restricted to $R$ is 
the kernel of the 
composition 
\begin{align*}
\sS_M|_{R} \hookrightarrow \oO_V/I^2 \twoheadrightarrow \oO_{R^{\rm{red}}}. 
\end{align*}
For example, 
suppose that  
$f \colon V \to \mathbb{A}^1$ is
a regular function such that 
\begin{align}\label{R=df}
R=\{df=0\}, \ 
f|_{R^{\rm{red}}}=0.
\end{align}
Then 
$f+I^2$ is an element of
$\Gamma(R, \sS_{M}^0|_{R})$. 

\begin{defi}\emph{(\cite{JoyceD})}
A pair $(M, s)$
for a complex scheme $M$
and $s \in \Gamma(M, \sS_M^0)$
is called a $d$-critical 
scheme 
if for any 
$x \in M$, there is an open 
neighborhood $x \in R \subset M$,  
a closed embedding $i \colon R \hookrightarrow V$
into a smooth scheme $V$, 
a regular function $f \colon V \to \mathbb{A}^1$
satisfying (\ref{R=df})
such that 
$s|_{V}=f+I^2$
holds. 
In this case, 
the data
\begin{align}\label{crit:chart}
\xi=
(R, V, f, i)
\end{align}
is called a $d$-critical chart.
The section $s$ is called a $d$-critical 
structure of $M$.  
\end{defi} 
Roughly speaking, a 
$d$-critical
scheme $(M, s)$ is locally written 
as a critical locus of some function $f$ on a 
smooth scheme, and 
the section $s$ remembers the function $f$.  
Given a $d$-critical scheme $(M, s)$, 
there exists a line bundle 
$K_{M, s}$ on $M^{\rm{red}}$, called the 
\textit{virtual canonical line bundle}, 
such that for any $d$-critical chart (\ref{crit:chart})
there is a natural isomorphism
\begin{align}\label{nat:K}
K_{M, s}|_{R^{\rm{red}}} \stackrel{\cong}{\to}
K_V^{\otimes 2}|_{R^{\rm{red}}}. 
\end{align}
\begin{defi}\emph{(\cite{JoyceD})}
An orientation of a $d$-critical scheme $(M, s)$
is a
 choice of a square root line bundle 
$K_{M, s}^{1/2}$ for $K_{M, s}$ on $M^{\rm{red}}$
and an isomorphism
\begin{align}\label{isom:orient}
(K_{M, s}^{1/2})^{\otimes 2} \stackrel{\cong}{\to}
K_{M, s}. 
\end{align}
A $d$-critical scheme with an orientation is called
an oriented $d$-critical scheme. 
\end{defi}

\subsection{Sheaves of vanishing cycles}
Let $f \colon V \to \mathbb{A}^1$ be 
a regular function on a smooth scheme $V$, 
and set 
$R=\{df=0\}$.
Suppose that $f|_{R^{\rm{red}}} =0$
and set $V_0=f^{-1}(0)$.  
We have the associated vanishing cycle functor
(see~\cite[Theorem~5.2.21]{Dimbook})
\begin{align*}
\phi_f \colon \Perv(V) \to \Perv(V_0). 
\end{align*}
Let $\IC(V) \in \Perv(V)$ be the intersection complex 
on $V$, which 
coincides with $\mathbb{Q}_V[\dim V]$ 
since $V$ is smooth. 
We have the perverse sheaf 
of vanishing cycles
supported on $R^{\rm{red}} \subset V_0$
\begin{align}\label{vsheaf}
\phi_f(\IC(V)) \in \Perv(R) \subset \Perv(V_0). 
\end{align}
Let $(M, s)$ be a $d$-critical scheme. 
For a $d$-critical chart 
$(R, V, f, i)$ as in 
(\ref{crit:chart}), 
we have the sheaf of vanishing cycles 
(\ref{vsheaf}) on $R$. 
In~\cite{BDJS} it is proved that 
if $(M, s)$ is oriented, then the 
sheaves of vanishing cycles (\ref{vsheaf}) glue 
to give a global perverse sheaf on $M$. 
Let
\begin{align}\label{isom:KK}
(K_{M, s}^{1/2}|_{R^{\rm{red}}})^{\otimes 2}
\cong K_V^{\otimes 2}|_{R^{\rm{red}}}
\end{align}
be the isomorphism given by
the composition of (\ref{nat:K})
and (\ref{isom:orient}). 
Then there is a 
$\mathbb{Z}/2\mathbb{Z}$-principal bundle 
\begin{align}\notag
\tau_R \colon 
\widetilde{R}^{\rm{red}} \to R^{\rm{red}}
\end{align}
which parametrizes local square roots 
of the isomorphism (\ref{isom:KK}). 
We have the decomposition
\begin{align*}
\tau_{R\ast} \mathbb{Q}_{\widetilde{R}^{\rm{red}}}
=\mathbb{Q}_{R^{\rm{red}}} \oplus \lL_{\xi}
\end{align*}
for a rank one local system $\lL_{\xi}$ on $R^{\rm{red}}$. 
The following result is proved in~\cite{BDJS}
(also see~\cite{KL} for the same result 
in the framework of virtual critical structures): 
\begin{thm}\emph{(\cite[Theorem~6.9]{BDJS})}\label{thm:BDJS}
For an oriented 
$d$-critical scheme $\mM=(M, s, K_{M, s}^{1/2})$, 
there exists a natural 
perverse sheaf 
$\phi_{\mM}$ on $M$ such that for any 
$d$-critical chart (\ref{crit:chart})
there is a natural isomorphism
\begin{align}\label{isom:IC}
\phi_{\mM}|_{R} \stackrel{\cong}{\to}
\phi_f(\IC(V)) \otimes \lL_{\xi}. 
\end{align}
Moreover there exists a 
natural isomorphism
$\mathbb{D}_M(\phi_{\mM}) \cong \phi_{\mM}$, 
where $\mathbb{D}_M$ is the Verdier dualizing functor. 
\end{thm}

\subsection{GV type invariants}\label{subsec:GVtype}
For an oriented $d$-critical scheme
$\mM=(M, s, K_{M, s}^{1/2})$, let 
\begin{align*}
\pi \colon M^{\rm{red}} \to T
\end{align*}
be 
a projective morphism
to a finite type complex scheme $T$. 
We will use the 
perverse sheaf $\phi_{\mM}$ 
in Theorem~\ref{thm:BDJS} through the following
lemma:
\begin{lem}\label{lem:ng}
There exist unique $\GV_{g, \mM/T} \in \mathbb{Z}$ for 
$g \in \mathbb{Z}_{\ge 0}$, 
and $\GV_{g, \mM/T}=0$ for $g\gg 0$, 
such that 
\begin{align}\label{def:ng}
\sum_{i \in \mathbb{Z}}
\chi(
\pH^i(
\dR \pi_{\ast} \phi_{\mM}))y^i
=\sum_{g\ge 0} \GV_{g, \mM/T}(y^{\frac{1}{2}}+y^{-\frac{1}{2}})^{2g}. 
\end{align}
\end{lem}
\begin{proof}
By the isomorphism 
$\mathbb{D}_M(\phi_{\mM}) \cong \phi_{\mM}$
and the Verdier duality, we have the isomorphism
\begin{align*}
\mathbb{D}_T(\dR \pi_{\ast}\phi_{\mM})
\cong \dR \pi_{\ast}\phi_{\mM}. 
\end{align*}
Since $\mathbb{D}_T$ preserves the perverse
t-structure, 
it follows that 
\begin{align*}
\mathbb{D}_T(\pH^{-i}(\dR \pi_{\ast}\phi_{\mM}))
\cong \pH^i(\dR \pi_{\ast}\phi_{\mM}).
\end{align*}
Therefore the LHS of (\ref{def:ng}) is
a polynomial of $y^{\pm 1}$ which is invariant under
$y \mapsto y^{-1}$. 
By the induction of the degree of $y$, it is 
easy to see that any polynomial of $y^{\pm 1}$ invariant 
under $y \mapsto y^{-1}$ is uniquely written 
as the form of the RHS of (\ref{def:ng}). 
\end{proof}
We also have the following local version
of Lemma~\ref{lem:ng}, 
whose proof is identical to Lemma~\ref{lem:ng}. 
\begin{lem}\label{lem:ngloc}
In the situation of Lemma~\ref{lem:ng}, 
for $t \in T$
there exist unique $\GVloc_{g, \mM/T, t}\in \mathbb{Z}$
for $g \in \mathbb{Z}_{\ge 0}$, 
and $\GVloc_{g, \mM/T, t}=0$ for $g\gg 0$,  
such that 
\begin{align}\notag
\sum_{i \in \mathbb{Z}}
\chi(
\pH^i(
\dR \pi_{\ast} \phi_{\mM})|_{t})y^i
=\sum_{g\ge 0} \GVloc_{g, \mM/T, t}(y^{\frac{1}{2}}+y^{-\frac{1}{2}})^{2g}. 
\end{align}
\end{lem}
For each $g \in \mathbb{Z}_{\ge 0}$, 
we have the constructible function on $T$
\begin{align*}
\GVloc_{g, \mM/T, -} \colon T \to \mathbb{Z}, \ 
t \mapsto \GVloc_{g,\mM/T, t}.
\end{align*}
Then the integer $\GV_{g, \mM/T}$ in (\ref{def:ng}) is written as 
\begin{align}\label{int:id}
\GV_{g, \mM/T}=\int_{T} \GVloc_{g, \mM/T, -} \ de. 
\end{align}

In genus zero, our GV type invariants are described in terms of 
Behrend's constructible function. 
For the perverse sheaf $\phi_{\mM}$ in Theorem~\ref{thm:BDJS}, 
the 
\textit{Behrend constructible function}~\cite{Beh}
on $M$ is defined 
by 
\begin{align}\label{Bfunc}
\nu_{M} \colon M \to \mathbb{Z}, \ 
p \mapsto \chi(\phi_{\mM}|_{p}). 
\end{align}
The constructible function 
$\nu_{M}$ is independent of the choice of orientation 
data. Indeed, it is proved in~\cite{Beh} that 
$\nu_M$ only depends on the scheme structure of $M$. 
\begin{lem}\label{lem:g=0}
We have the identities
\begin{align*}
\GV_{0, \mM/T}=\int_{M} \nu_{M} \ de, \ 
\GVloc_{0, \mM/T, t}=\int_{\pi^{-1}(t)} \nu_{M} \ de. 
\end{align*}
In particular, $\GV_{0, \mM/T}$, $\GVloc_{0, \mM/T, t}$ are independent 
of the choice of orientation. 
\end{lem}
\begin{proof}
By substituting $y=-1$ to (\ref{def:ng}), 
we obtain 
$\GV_{0, \mM/T}=\chi(\phi_{\mM})$. 
Therefore the first identity holds. 
The second identity also holds in the similar way. 
\end{proof}

\subsection{Strictly Calabi-Yau conditions}
In this subsection, we introduce a certain CY condition
for $d$-critical schemes.  While not needed for 
our definitions, it is convenient for proofs later on.

\begin{defi}\label{defi:slcy}
(i)
A $d$-critical scheme $(M, s)$ is called 
strictly CY if 
there is a global 
$d$-critical chart 
\begin{align}\label{cychart}
(M, J, f, i), 
\ i \colon M \hookrightarrow J, \ 
f \colon J \to \mathbb{A}^1
\end{align}
of $(M, s)$ such that $K_J=\oO_J$. 
Here 
$i$ is a closed embedding and $f$ is a regular function. 
A $d$-critical chart (\ref{cychart}) 
is called a CY $d$-critical chart. 

(ii)
A $d$-critical scheme $(M, s)$ with
a projective morphism $\pi \colon M^{\rm{red}} \to T$
is called strictly CY at $t \in T$
if there is an open neighborhood $t \in U \subset T$
such that $(M_U, s_U)$ is strictly CY.
Here $M_U \cneq \iota(\pi^{-1}(U))$, 
$\iota \colon M^{\rm{red}} \hookrightarrow M$ is the 
natural closed embedding, and $s_U=s|_{M_U}$. 
\end{defi} 
\begin{rmk}
If a $d$-critical scheme $(M, s)$ is strictly CY, 
the line bundle $i^{\ast}K_J=\oO_{M}$
together with the identity map 
$\mathrm{id} \colon \oO_M^{\otimes 2} \to \oO_M$
gives a CY orientation
(see Definition~\ref{loc:cy}) of $(M, s)$. 
\end{rmk}

In the situation of 
Definition~\ref{defi:slcy}~(ii), 
let us take 
a CY $d$-critical chart $(M_U, J, f, i)$. 
Moreover, let us make the additional assumption
that the ring $H^0(\oO_J)$ is finitely generated.
The function $f$ factors through the 
affinization
\begin{align*}
f \colon J \stackrel{\pi_J}{\longrightarrow}
 T\cneq \Spec H^0(\oO_J) \stackrel{g}{\longrightarrow} \mathbb{A}^1. 
\end{align*} 
Suppose that $U \subset T$ is affine. 
The Stein factorization 
of $\pi \colon M_U^{\rm{red}} \to U$ 
is given by
\begin{align*}
\pi \colon 
M_U^{\rm{red}} \stackrel{\pi_1}{\longrightarrow}
\overline{U} \cneq \Spec H^0(\oO_{M_U^{\rm{red}}}) \stackrel{\pi_2}{\longrightarrow}
 U. 
\end{align*}
By the property of the
affinization, we have the commutative diagram
\begin{align}\label{dia:sCYchart}
\xymatrix{
M_U^{\rm{red}} 
\ar@<-0.3ex>@{^{(}->}[r]^-{\iota} 
 \ar[d]_-{\pi_1} 
& M_U 
 \ar@<-0.3ex>@{^{(}->}[r]^-{i} 
& J \ar[rd]^-{f} \ar[d]_-{\pi_J} & \\
\overline{U} \ar[rr]^-{i'} & & T \ar[r]^-{g} & \mathbb{A}^1. 
}
\end{align}
In the proof of Theorem~\ref{intro:thm1}, \ref{intro:thm2}, 
we will show the strictly CY conditions 
for moduli of one dimensional sheaves 
and use the diagrams as above.  

\begin{rmk}
Let 
$\mM_U=(M_U, s_U, i^{\ast}K_J)$
be the oriented $d$-critical scheme 
given by the CY $d$-critical chart $(M_U, J, i, f)$, 
and $\phi_{\mM_U}$ the
perverse sheaf given 
in Theorem~\ref{thm:BDJS}. 
Although $\phi_{\mM_U}$ is not necessary 
pure, 
one can use the BBD decomposition theorem~\cite{BBD}
for $\dR \pi_{J\ast} \IC(J)$, 
the compatibility of $\phi_g$
 with 
proper push forwards~\cite[Proposition~4.2.11]{Dimbook}, 
to show the decomposition
\begin{align}\label{phi:decompose}
\dR \pi_{\ast} \phi_{\mM_U}
\cong \bigoplus_{j \in \mathbb{Z}} \pH^{j}(\dR \pi_{\ast}\phi_{\mM_U})[-j]. 
\end{align}
By (\ref{phi:decompose}), one can interpret 
$\GVloc_{g, (M, s)/T, t}$ in terms of 
the character of $sl_2$-action on
the RHS of (\ref{phi:decompose}). 
As this fact will not be used in this paper, we omit the details. 
\end{rmk}

\section{GV invariants on Calabi-Yau 3-folds}\label{sec:GVCY}
Let $X$ be a smooth 
quasi-projective CY 3-fold, i.e. there is an isomorphism
\begin{align}\label{isom:omega}
\oO_X \stackrel{\cong}{\to} K_X.
\end{align}
Below, we fix the isomorphism (\ref{isom:omega})
In this section, we define GV invariants on $X$
and formulate the conjectures relating them with PT and GW
invariants.

\subsection{Definition of GV invariants}\label{subsec:defgv}
Let 
\begin{align}\notag
\Coh_{\le 1}(X) \subset \Coh(X)
\end{align}
be the subcategory 
consisting of 
sheaves whose supports are compact 
and have dimensions less than or equal to one. 
For an ample divisor $\omega$ on $X$ and 
$0\neq E \in \Coh_{\le 1}(X)$, the $\omega$-slope 
$\mu_{\omega}(E)$ is defined by
\begin{align*}
\mu_{\omega}(E) \cneq \frac{\chi(E)}{\omega \cdot [E]} \in 
\mathbb{Q} \cup \{\infty\}. 
\end{align*}
\begin{defi}
An object $E \in \Coh_{\le 1}(X)$ 
is called $\omega$-(semi)stable 
if for any subobject $0\neq E'\subsetneq E$, we have 
the inequality
$\mu_{\omega}(E')<(\le) \mu_{\omega}(E)$. 
\end{defi}
\begin{rmk}\label{rmk:omega}
For $E \in \Coh_{\le 1}(X)$ with $\chi(E)=1$, 
its $\omega$-stability 
is equivalent to that $\chi(E') \le 0$ for any 
subsheaf $0\subsetneq E' \subsetneq E$. 
In particular, it is independent of a
choice of $h$. 
\end{rmk}
For $\beta \in H_2(X, \mathbb{Z})$, 
we denote by 
\begin{align}\label{moduli:M}
\Sh_{\beta}(X)
\end{align}
the moduli space of $\omega$-stable
$E \in \Coh_{\le 1}(X)$ with $[E]=\beta$, $\chi(E)=1$. 
Note that $\Sh_{\beta}(X)$ 
is independent of $h$ by Remark~\ref{rmk:omega}. 
Also the condition $\chi(E)=1$ implies that $\Sh_{\beta}(X)$ is fine, i.e. 
there is no strictly $\omega$-semistable $E \in \Coh_{\le 1}(X)$
with $[E]=\beta$, $\chi(E)=1$, and 
there is a universal sheaf  
\begin{align*}
\eE \in \Coh(X \times \Sh_{\beta}(X)). 
\end{align*}
Then it is known that $\Sh_{\beta}(X)$ admits a canonical $d$-critical 
structure:  
\begin{thm}\emph{(\cite{BBBJ})}\label{thm:CYM}
A choice of the trivialization (\ref{isom:omega}) 
gives a canonical 
$d$-critical structure $s_{\Sh}$ on
$\Sh_{\beta}(X)$
such that its virtual canonical bundle
is given by 
\begin{align}\label{K:vir}
K_{\Sh}^{\rm{vir}}
 \cneq  \det \left(\dR p_{\Sh\ast} 
\dR \hH om_{X \times \Sh_{\beta}(X)}(\eE, \eE)
\right)|_{\Sh_{\beta}^{\rm{red}}(X)}. 
\end{align}
Here $p_{\Sh} \colon X \times \Sh_{\beta}(X) \to \Sh_{\beta}(X)$
 is the projection. 
\end{thm}

\begin{rmk}\label{rmk:shifted}
The canonical $d$-critical structure in 
Theorem~\ref{thm:CYM}
is induced from derived deformation theory.  
Let $\widehat{\Sh}_{\beta}(X)$ be the derived moduli space of
$h$-stable sheaves $E$ with $[E]=\beta$, $\chi(E)=1$. 
If $X$ is projective, by~\cite{PTVV},
a choice of (\ref{isom:omega})
gives a canonical $(-1)$-shifted symplectic structure on 
$\widehat{\Sh}_{\beta}(X)$, 
which is written as a Darboux form by~\cite{BBBJ}, 
and induce the canonical $d$-critical structure
on its truncation $\Sh_{\beta}(X) \subset \widehat{\Sh}_{\beta}(X)$ 
in Theorem~\ref{thm:CYM}. 
\end{rmk}
\begin{rmk}
Since we only assume that $X$ is quasi-projective, 
we need a generalization of the result in~\cite{PTVV}
to the quasi-projective case. 
If $X$ is written as $(u \neq 0)$ for a smooth projective 
variety $Y$ and $0\neq u \in H^0(-K_Y)$, 
Bussi~\cite[Theorem~5.2]{Bussi} 
shows that $\widehat{\Sh}_{\beta}(X)$ has a canonical $(-1)$-shifted 
symplectic structure. 
In general, this follows from 
the work of Preygel~\cite[Theorem~3.0.6]{Pre}
combined with the original argument of~\cite{PTVV}. 
\end{rmk}
Let 
\begin{align}\label{HCmap}
\pi \colon \Sh_{\beta}^{\rm{red}}(X) \to \Chow_{\beta}(X)
\end{align}
be the Hilbert-Chow morphism. 
Here $\Chow_{\beta}(X)$
is the Chow variety which parametrizes compactly supported
effective one-cycles on $X$ with homology class $\beta$
(see~\cite{Ko}), 
and the map 
(\ref{HCmap}) sends 
a one dimensional sheaf $E$ 
to 
the associated fundamental cycle of $E$. 
Because 
$\Sh_{\beta}(X)$ is a fine moduli space, 
the morphism (\ref{HCmap}) is a 
projective morphism. 
\begin{rmk}
Here we use the classical definition of the Chow variety, which is 
a reduced scheme and denoted as 
$\Chow'(X)$ in~\cite{Ko}.
The existence of the map (\ref{HCmap}) follows from, for example, 
by the argument of~\cite[Corollary~7.15]{Ryd}.  
\end{rmk}

The $d$-critical scheme
in Theorem~\ref{thm:CYM} always
admits a (non-canonical) orientation~\cite{NNAO}. 
Let us take one of them and consider 
an oriented $d$-critical scheme
\begin{align*}
\sS h_{\beta}(X)=(\Sh_{\beta}(X), s_{\Sh}, (K_{\Sh}^{\rm{vir}})^{1/2}).
\end{align*}
In the following definition, 
we take a special type of orientation data $(K_{\Sh}^{\rm{vir}})^{1/2}$, 
called \textit{CY orientation data} (see~Definition~\ref{loc:cy}).
We expect that such an orientation data always exist
(or at least locally on $\Chow_{\beta}(X)$, 
 see Conjecture~\ref{conj:GV1}). 

\begin{defi}\label{def:GV}
We define the 
invariants $n_{g, \beta} \in \mathbb{Z}$ by
(see Lemma~\ref{lem:ng})
\begin{align*}
n_{g, \beta} \cneq 
\GV_{g, \sS h_{\beta}(X)/\Chow_{\beta}(X)}. 
\end{align*}
The invariant $n_{g, \beta}$ is called the
genus $g$ Gopakumar-Vafa invariant 
with curve class $\beta$. 
\end{defi}

\begin{rmk}
Once we take a CY orientation data, the 
resulting invariant $n_{g, \beta}$ is independent of 
an orientation data as long as it is CY
(see Lemma~\ref{lem:gvcompare}). 
\end{rmk}

The local version is defined in the similar way: 

\begin{defi}\label{def:locgv}
We define the invariants 
$n_{g, \gamma}^{\rm{loc}} \in \mathbb{Z}$ for $g\in \mathbb{Z}_{\ge 0}$
by (see Lemma~\ref{lem:ngloc})
\begin{align*}
n_{g, \gamma}^{\rm{loc}} \cneq 
\GVloc_{g, \sS h_{\beta}(X)/\Chow_{\beta}(X), \gamma}. 
\end{align*}
The invariant $n_{g, \gamma}^{\rm{loc}}$ is called the
local genus $g$ Gopakumar-Vafa invariant at $\gamma$. 
\end{defi}

\begin{rmk}\label{def:global}
Similarly to the global case, 
the local GV invariant 
$n_{g, \gamma}^{\rm{loc}}$ is defined 
using a CY orientation data. 
However for the local case, we only need 
such an orientation data locally 
on $\Chow_{\beta}(X)$ near $\gamma$
(i.e. only need to assume
that $\Sh_{\beta}(X)$ is CY at $\gamma$
in Definition~\ref{loc:cy}~(iii)).
If the latter condition holds for any $\gamma \in \Chow_{\beta}(X)$, 
(i.e. (\ref{HCmap}) is a CY fibration in Definition~\ref{loc:cy}~(iv)), 
we can define the global 
GV invariants by the identity
(see Remark~\ref{rmk:loc/global})
\begin{align*}
n_{g, \beta}=\int_{\Chow_{\beta}(X)} n_{g, -}^{\rm{loc}} \ de.
\end{align*}
If $\Sh_{\beta}(X)$ has a global CY orientation data, the above 
definition coincides with the one in Definition~\ref{def:GV}. 
\end{rmk}
If $X$ is projective, 
our genus zero GV invariant agrees with 
Katz's definition~\cite{Katz}:
\begin{lem}\label{lem:Katz}
If $X$ is projective, we 
have the identity
\begin{align*}
n_{0, \beta}=\int_{[\Sh_{\beta}(X)]^{\rm{vir}}}1. 
\end{align*}
\end{lem}
\begin{proof}
The assumption implies that $\Sh_{\beta}(X)$ is also projective and the RHS
makes sense. 
Then the lemma follows from Lemma~\ref{lem:g=0} together with  
the result of Behrend~\cite{Beh} that the degree of the 
virtual class of $\Sh_{\beta}(X)$ coincides with the
weighted Euler number of its Behrend constructible function. 
\end{proof}

\subsection{Conjectures on the relation to PT/GW invariants}
\label{sec:PTGW}
Let $X$ be 
a smooth quasi-projective CY 3-fold
as before. 
By definition, a \textit{stable pair}
introduced by Pandharipande-Thomas~\cite{PT}
 consists of a pair 
\begin{align*}
(F, s), \ F \in \Coh_{\le 1}(X), \ s \colon \oO_X \to F
\end{align*}
where $F$ is
pure one-dimensional 
and $s$ is surjective in dimension one.
For $\beta \in H_2(X, \mathbb{Z})$ and $n\in \mathbb{Z}$, 
let
\begin{align*}
P_n(X, \beta)
\end{align*}
be the moduli space of 
stable pairs $(F, s)$ on $X$ with 
$[F]=\beta$ and $\chi(F)=n$. 
By~\cite{BBBJ}, 
the moduli space of stable pairs $P_n(X, \beta)$ admits 
a canonical $d$-critical 
structure $s_P$
whose virtual canonical line bundle 
$K_P^{\rm{vir}}$ is given similarly to (\ref{K:vir})
for universal stable pairs. 
By~\cite{NNAO}, 
it always has a 
(non-canonical) orientation $(K_{P}^{\rm{vir}})^{1/2}$.
An oriented $d$-critical scheme
\begin{align*}
\pP_n(X, \beta)=(P_n(X, \beta), s_P, (K_{P}^{\rm{vir}})^{1/2})
\end{align*}
determines
the sheaf of vanishing cycles $\phi_{\pP}$ on $P_n(X, \beta)$
by Theorem~\ref{thm:BDJS}. 
Similarly to (\ref{Bfunc}), let 
\begin{align*}
\nu_{P} \colon P_n(X, \beta) \to \mathbb{Z}, \ 
p \mapsto \chi(\phi_{\pP}|_{p})
\end{align*}
be the Behrend
constructible function
on $P_n(X, \beta)$. 
The PT invariant is defined by 
\begin{align*}
P_{n, \beta} \cneq \int_{P_n(X, \beta)} \nu_{P} \ de.
\end{align*}
If $X$ is projective, 
it coincides with 
the integration of the zero-dimensional 
virtual class of $P_n(X, \beta)$~\cite{Beh}. 
\begin{rmk}\label{P:dcritical}
We don't need the $d$-critical structure on $P_n(X, \beta)$ 
to define $P_{n, \beta}$, as we only need the 
Behrend constructible function $\nu_P$
to define it. 
The $d$-critical structure on $P_n(X, \beta)$
and the associated vanishing cycle sheaf will be used 
later, e.g. in Subsection~\ref{subsec:GVCY3}. 
\end{rmk}

By~\cite[Lemma~3.1]{PT}, 
there exist unique integers
\begin{align*}
n_{g, \beta}^{P} \in \mathbb{Z}, \ \beta>0, \ 
g \in \mathbb{Z}
\end{align*}
such that 
the logarithm of the generating series of 
stable pairs 
is written as 
\begin{align}\label{PT/GV:form}
&\log \left(
1+
\sum_{\beta>0, 
 n\in \mathbb{Z}}
P_{n, \beta}(-q)^n t^{\beta} \right) \\
&= \notag
\sum_{\beta>0}
\sum_{g\in \mathbb{Z}, k\ge 1}
\frac{n_{g, \beta}^{P}}{k}
(-1)^{g-1}
(q^{\frac{k}{2}}-q^{-\frac{k}{2}})^{2g-2} t^{k\beta}. 
\end{align}
Here $\beta>0$ means that $\beta$ is a homology class of an effective 
one cycle on $X$.
\begin{conj}\label{GV:conj1}
Under the situation of Definition~\ref{def:GV}, we have the identity
\begin{align}\label{id:GV/PT}
n_{g, \beta}^P=n_{g, \beta}, \ \beta>0, 
\ g \in \mathbb{Z}. 
\end{align}
\end{conj}

The local PT invariants are also defined in a similar way.  
For $\gamma \in \Chow_{\beta}(X)$, 
let $P_n^{\rm{loc}}(X, \gamma) \subset P_n(X, \beta)$ be the 
closed subset 
corresponding to stable pairs $(F, s)$ 
such that the fundamental cycle of $F$ 
coincides with $\gamma$.  
We define 
\begin{align*}
P_{n, \gamma}^{\rm{loc}} \cneq \int_{P_n^{\rm{loc}}(X, \gamma)} \nu_{P} \ de. 
\end{align*}
We have the constructible function
\begin{align*}
P_{n, -}^{\rm{loc}} \colon 
\Chow_{\beta}(X) \to \mathbb{Z}, \ 
\gamma \mapsto P_{n, \gamma}^{\rm{loc}}.
\end{align*}
Similarly to (\ref{PT/GV:form}), there 
exist locally constructible functions
\begin{align*}
n_{g, -}^{P, \rm{loc}} \colon 
\Chow(X) \cneq \coprod_{\beta>0} \Chow_{\beta}(X) \to \mathbb{Z}, \ 
g \in \mathbb{Z}
\end{align*}
such that we have the identity
\begin{align}\label{PT/GV:local}
&\log\left(
1+
\sum_{
 n\in \mathbb{Z}}
P_{n, -}^{\rm{loc}}q^n \right) \\
&= \notag
\sum_{
g\in \mathbb{Z}, k\ge 1
}
(k)_{\ast}
\frac{n_{g, -}^{P, \rm{loc}}}{k}
(-1)^{g-1}
((-q)^{\frac{k}{2}}-(-q)^{-\frac{k}{2}})^{2g-2}. 
\end{align}
The above identity is interpreted
as the identity of $\mathbb{Q}(\!(q)\!)$-valued
 functions
on $\Chow(X)$. Here for 
functions $n_1$, $n_2$ 
on $\Chow(X)$, the product is defined by 
\begin{align}\label{const:prod}
n_1 \cdot n_2 \cneq 
(+)_{\ast}(n_1 \boxtimes n_2)
\end{align}
where 
$+$ is an obvious addition map of one cycles
\begin{align*}
+ \colon \Chow(X) \times \Chow(X) \to \Chow(X).  
\end{align*}
The logarithm of the LHS of (\ref{PT/GV:local})
is taken with respect to the product (\ref{const:prod}). 
Also $(k)_{\ast}$ in the RHS of (\ref{PT/GV:local}) is the 
push-forward of the map 
$\gamma \mapsto k\gamma$ on $\Chow(X)$. 

\begin{conj}\label{GV:conj2}
Under the situation of Definition~\ref{def:locgv},
we have the identity
\begin{align}\label{id:GV/PT2}
n_{g, \gamma}^{P, \rm{loc}}=n_{g, \gamma}^{\rm{loc}}, 
\ \gamma \in \Chow_{\beta}(X), 
\ g \in \mathbb{Z}. 
\end{align} 
\end{conj}

\begin{rmk}
By integrating over the Chow variety, it is 
easy to see that Conjecture~\ref{GV:conj2} implies 
Conjecture~\ref{GV:conj1}. 
\end{rmk}

\begin{rmk}
For an irreducible one cycle $\gamma$, 
Conjecture~\ref{GV:conj2} is equivalent to 
Conjecture~\ref{intro:conj}. 
\end{rmk}

\begin{rmk}
The identity (\ref{id:GV/PT}) 
in particular implies $n_{g, \beta}^{P}=0$ for $g<0$, 
which is the strong rationality conjecture 
in~\cite[Conjecture~3.14]{PT}.
By~\cite[Theorem~6.4]{Tsurvey}, the
above vanishing is equivalent to the multiple 
cover conjecture of generalized DT invariants of 
one dimensional semistable sheaves
(see~\cite[Conjecture~6.3]{Tsurvey}), and in this 
case the identity (\ref{id:GV/PT}) holds for $g=0$. 
Similarly in the local case, the vanishing $n_{g, \gamma}^{P, \rm{loc}}=0$ 
for $g<0$ 
is equivalent to the multiple cover conjecture of local 
generalized DT invariants (see~\cite[Conjecture~4.13]{Todpara}) 
and in this case
(\ref{id:GV/PT2}) holds for $g=0$. 
\end{rmk}

If $X$ is projective, 
the comparison with Gromov-Witten invariants is 
also formulated in a similar way. 
Let 
\begin{align*}
\mathrm{GW}_{g, \beta} \in \mathbb{Q}
\end{align*}
be the 
genus $g$ Gromov-Witten invariant 
on $X$ with curve class $\beta$. 
Then there exist (a priori rational numbers)
\begin{align*}
n_{g, \beta}^{GW} \in \mathbb{Q}, \ 
g\in \mathbb{Z}_{\ge 0}, \ \beta>0
\end{align*}
satisfying the identity: 
\begin{align*}
\sum_{\beta>0, g\ge 0}
\mathrm{GW}_{g, \beta}
\lambda^{2g-2}t^{\beta}
=\sum_{\beta>0, g\ge 0, k\ge 1}
\frac{n_{g, \beta}^{GW}}{k}
\left(2\sin\left( \frac{k\lambda}{2} \right)\right)^{2g-2}
t^{k\beta}. 
\end{align*}
\begin{conj}\label{conj:GV/GW}
Suppose that $X$ is a smooth projective CY 3-fold. 
Then under the situation of Definition~\ref{def:GV},
we have the identity
\begin{align*}
n_{g, \beta}^{GW}=n_{g, \beta}, \ \beta>0, \ g \in \mathbb{Z}_{\ge 0}. 
\end{align*}
\end{conj}
\begin{rmk}
Suppose that we know either 
$n_{g, \beta}^P=0$ for $g<0$ or 
$n_{g, \beta}^{GW}=0$ for $g\gg 0$. 
Then the conjectural GW/DT(PT)
correspondence~\cite{MNOP, PT}
implies $n_{g, \beta}^{P}=n_{g, \beta}^{GW}$. 
\end{rmk}

\subsection{Independence of Euler characteristic}\label{subsec:chi}
For $k \in \mathbb{Z}$, let 
\begin{align*}
\Sh_{\beta, k}(X)
\end{align*}
 be the moduli space of 
$\omega$-stable $E \in \Coh_{\le 1}(X)$
satisfying $[E]=\beta$, $\chi(E)=k$.
We note that, for $k\neq 1$, the moduli space $\Sh_{\beta, k}(X)$
may depend on a choice of $\omega$. 
Similarly to the $k=1$ case, 
we have the Hilbert-Chow map
\begin{align*}
\pi_{k} \colon \Sh_{\beta, k}^{\rm{red}}(X) \to \Chow_{\beta}(X).
\end{align*}
In Subsection~\ref{subsec:defgv}, we used the moduli space 
for $k=1$ and the map 
$\pi=\pi_1$ to define GV invariants. 
For other value of $k$, assuming that
$\Sh_{\beta, k}(X)$ is fine and $\pi_k$ is a CY fibration
(see Definition~\ref{loc:cy}), we
expect that we
 can also use $\Sh_{\beta, k}(X)$ to define GV invariants. 
Namely by replacing $\Sh_{\beta}(X)$ by $\Sh_{\beta, k}(X)$
in Definition~\ref{def:GV} and Definition~\ref{def:locgv},  
we have the GV type invariants
\begin{align}\label{GV:invk}
n_{g, \beta, k, \omega} \in \mathbb{Z}, \quad
n_{g, \gamma, k, \omega}^{\rm{loc}} \in \mathbb{Z}
\end{align}
for $\beta>0$ and $\gamma \in \Chow_{\beta}(X)$
respectively. 
A priori, the invariants (\ref{GV:invk}) may also 
depend on $\omega$. 
\begin{conj}\label{conj:ngk}
Assuming that 
$\Sh_{\beta, k}(X)$ is fine, 
we have the following: 
\begin{enumerate}
\item 
$\pi_k$ is a CY fibration. In particular, the invariants (\ref{GV:invk})
are defined. 
\item The invariants (\ref{GV:invk}) are independent of $\omega$. 
\item The invariants (\ref{GV:invk}) are independent of $k$. 
\end{enumerate}
\end{conj}
In genus zero, Conjecture~\ref{conj:ngk} (ii) 
is known to be true by a wall-crossing argument
(see~\cite[Theorem~6.6]{JS}), 
and Conjecture~\ref{conj:ngk} (iii) is a special case of the
multiple cover conjecture for generalized DT invariants
(see~\cite[Conjecture~4.13]{Todpara}). 
By another wall-crossing argument, Conjecture~\ref{conj:ngk} (iii) 
for a primitive one cycle $\gamma$ 
is proved when $g=0$ (see~\cite[Lemma~2.12]{Todmu}), 
and we expect that a similar argument may be applied 
for $g>0$. 
For an irreducible one cycle $\gamma$, 
Conjecture~\ref{conj:ngk} (iii)
should follow along with the same argument of~\cite[Proposition~2.1]{PT3}
by tensoring a local line bundle
with degree one on the support of $\gamma$.

\section{PT/GV correspondence for locally planar curves}\label{sec:PT/GV}
In 
this section, we explain an approach
for proving Conjecture~\ref{GV:conj2}
for integral planar curves. 
More precisely, we show
that the existence 
of strictly CY $d$-critical 
chart (see Definition~\ref{defi:slcy}) 
together with 
the results of~\cite{MY2, MS}
implies the conjecture.

\subsection{GV formula for locally versal deformations}
Let $C$ be an integral projective curve with 
at worst planar singularities. 
Let $g$ be the 
arithmetic genus of $C$, and 
 $\{c_1, \ldots, c_k\}$
the singular set of $C$. 
Let
\begin{align}\label{v:deform}
\pi_{T} \colon 
\cC \to T
\end{align}
be a flat family of curves with 
smooth base $T$ 
such that
$C=\pi_{T}^{-1}(0)$ for $0 \in T$. 
We recall the notion of locally versal family: 
\begin{defi}
A family (\ref{v:deform}) is called 
\textit{locally versal} at $0 \in T$ if the 
natural map
\begin{align}\label{Tan:map}
\mathrm{Tan}_0(T) \to \prod_{i=1}^{k} \mathrm{Tan}_0
(\mathrm{Def}(C, c_i))
\end{align}
is surjective.
Here $\mathrm{Def}(C, c_i)$ is the miniversal
deformation space of the singularity $c_i \in C$, 
and $\mathrm{Tan}_0(\ast)$ is the tangent space at $0$. 
If a family (\ref{v:deform})
is locally versal at any $t \in T$, 
it is called a locally versal family. 
\end{defi}
We will use the following lemma: 
\begin{lem}\label{lem:versal}
For any projective 
flat family of curves $\pi_H \colon \cC_{H} \to H$ 
with at worst planar singularities, 
there is a locally 
versal family of curves $\cC \to T$
and a 
closed immersion $H \hookrightarrow T$ such that 
$\cC_H=\cC \times_T H$. 
\end{lem}
\begin{proof}
Since $\pi_H$ is projective, 
there is a closed embedding
$\cC_H \subset \mathbb{P}^n \times H$ 
as $H$-schemes. 
Since $\pi_H$ is flat, it induces the morphism of schemes
$h \colon H \to \Hilb(\mathbb{P}^n)$, 
where $\Hilb(\mathbb{P}^n)$ is the Hilbert scheme of 
one dimensional subschemes in $\mathbb{P}^n$. 
Let 
\begin{align}\label{pi:Hilb}
\pi_{\Hilb} \colon \cC_{\Hilb} \to \Hilb(\mathbb{P}^n)
\end{align}
 be
the universal curve. 
We check that $\pi_{\Hilb}$ is locally 
versal at any point in the image of $h$. 
For $t \in H$, 
let us write $C=\pi_H^{-1}(t) \subset \mathbb{P}^n$. 
Let $I \subset \oO_{\mathbb{P}^n}$ be the ideal 
sheaf of $C$. 
By the exact sequence 
\begin{align*}
0 \to I/I^2 \to \Omega_{\mathbb{P}^n}|_{C} \to \Omega_C \to 0
\end{align*}
we have the exact sequence
\begin{align*}
\Hom(I/I^2, \oO_C) &\to 
\Ext_C^1(\Omega_C, \oO_C) 
\to H^1(C, T_{\mathbb{P}^n}|_{C}) \\
&\to 
\Ext_C^1(I/I^2, \oO_C) \to 
\Ext_C^2(\Omega_C, \oO_C) \to 0. 
\end{align*}
By replacing the embedding
$\cC_H \subset \mathbb{P}^n \times H$
 if necessary, we may 
assume that $H^1(C, \oO_{C}(1))=0$ holds
for any choice of $t \in H$. 
Then the vanishing  
$H^1(C, T_{\mathbb{P}^n}|_{C})=0$
also holds. 
Since the singularities of $C$ are planar, 
we also have $\Ext_C^2(\Omega_C, \oO_C)=0$. 
By the above exact sequence, 
we have $\Ext_C^1(I/I^2, \oO_C)=0$, 
therefore $\Hilb(\mathbb{P}^n)$ is smooth 
at $h(t)$. 
Moreover we have the surjections
\begin{align*}
\Hom(I/I^2, \oO_C) \twoheadrightarrow 
\Ext_C^1(\Omega_C, \oO_C) \twoheadrightarrow 
H^0(C, \eE xt^1(\Omega_C, \oO_C)).
\end{align*}
The above map is identified with the map
(\ref{Tan:map}) for the family
 (\ref{pi:Hilb}), 
therefore 
 $\pi_{\Hilb}$ is locally versal at $t$. 

Let $h(H) \subset U \subset \Hilb(\mathbb{P}^n)$ be an
open subset on which 
the corresponding one dimensional subscheme have 
at worst planar singularities, and 
$\pi_{\Hilb}$ is locally versal. 
We set $T=H \times U$, 
and $\cC=H \times \pi_{\Hilb}^{-1}(U)$. 
Then $(\id \times \pi_{\Hilb}) \colon \cC \to T$
is a locally versal family. 
By taking the embedding 
$(\id, h) \colon H \hookrightarrow T$, 
the lemma follows. 
\end{proof}
For a locally versal family (\ref{v:deform}),
let  
\begin{align*}
\pi^{[n]} \colon 
\cC^{[n]} \to T, \ 
\pi_J \colon \overline{J} \to T
\end{align*}
be 
the $\pi_{T}$-relative Hilbert scheme of $n$-points, 
$\pi_{T}$-relative
rank one torsion free sheaves with Euler characteristic one, 
respectively. 
The following is the main result of~\cite{MY2, MS}. 
\begin{thm}\emph{(\cite{MY2, MS})}\label{thm:versal}
Both of $\cC^{[n]}$ and $\overline{J}$ are smooth
for any $n\ge 1$. 
After replacing $T$ by its \'etale cover if necessary, 
we have the identity
in $K(\Perv(T))(\!(q)\!)$: 
\begin{align}\label{id:versal}
\sum_{n\ge 0}
\dR \pi_{\ast}^{[n]} \IC(\cC^{[n]}) q^{n+1-g}
=\frac{q}{(1+q)^2}
\sum_{i\in \mathbb{Z}} \pH^i(\dR \pi_{J\ast}\IC(\overline{J}))q^{i}. 
\end{align}
\end{thm}
Let $g \colon T \to \mathbb{A}^1$
be a regular function.  
We define $f^{[n]}$ and $f_{J}$ by the 
commutative diagram 
\begin{align*}
\xymatrix{
\overline{J} \ar[rd]_-{\pi_J}  \ar[rrd]^-{f_{J}} & &  \\
  &  T \ar[r]_-{g} & \mathbb{A}^1  \\
\cC^{[n]} \ar[ru]^-{\pi^{[n]}} \ar[rru]_-{f^{[n]}} &  &
}
\end{align*}
Let us consider the associated vanishing cycle sheaves:
\begin{align}\label{vanish:phi}
&\phi^{[n]} \cneq \phi_{f^{[n]}}(\IC(\cC^{[n]}))
\in \Perv(\cC^{[n]}), \\
\notag
&\phi_{J} \cneq \phi_{f_{J}}(\IC(\overline{J})) \in 
\Perv(\overline{J}).
\end{align}
Applying Theorem~\ref{thm:versal}, we have the following lemma: 
\begin{lem}
After replacing $T$ by its \'etale cover if necessary, 
we have the identity in $K(\Perv(T))(\!(q)\!)$:
\begin{align}\label{id:phi}
\sum_{n\ge 0}
\dR \pi_{\ast}^{[n]}\phi^{[n]} q^{n+1-g}
=\frac{q}{(1+q)^2}
\sum_{i\in \mathbb{Z}} \pH^{i} (\dR \pi_{J\ast}\phi_{J}) q^{i}. 
\end{align}
\end{lem}
\begin{proof}
By the compatibility of vanishing cycles with 
proper push forwards (see~\cite[Proposition~4.2.11]{Dimbook}), 
we have
\begin{align*}
\phi_g(\dR \pi_{\ast}^{[n]}\IC(\cC^{[n]}))
=\dR \pi_{\ast}^{[n]}\phi^{[n]}, \ 
\phi_g(\dR \pi_{J\ast} \IC(\overline{J}))=
\dR \pi_{J\ast}\phi_J. 
\end{align*}
Since $\phi_g$ preserves the perverse t-structure, 
it commutes with the 
perverse cohomology functor $\pH^i(-)$. Therefore we have 
\begin{align*}
\phi_g(\pH^{i} (\dR \pi_{J\ast}\IC(\overline{J})))
=\pH^{i} (\dR \pi_{J\ast}\phi_J). 
\end{align*} 
Therefore the lemma follows by 
applying the vanishing cycle functor $\phi_g$ to both sides of 
(\ref{id:versal}). 
\end{proof}

\subsection{GV formula for Calabi-Yau 3-folds}\label{subsec:GVCY3}
Let $X$ be a smooth quasi-projective CY 3-fold, 
and $C \subset X$
an integral projective curve with at worst planar singularities
with arithmetic genus $g$. 
For $\beta=[C] \in H_2(X, \mathbb{Z})$
and $n\in \mathbb{Z}$, let
\begin{align*}
(P=P_{n+1-g}(X, \beta), s_{P}), \ 
(\Sh=\Sh_{\beta}(X), s_{\Sh}) 
\end{align*}
be the $d$-critical schemes considered
in Section~\ref{sec:PTGW} and Section~\ref{subsec:defgv}
respectively. 
We have the Hilbert-Chow morphisms
\begin{align*}
\xymatrix{
P^{\rm{red}}
\ar[dr]_{\pi_P}  &   &  \Sh^{\rm{red}} \ar[dl]^{\pi_{\Sh}} \\
& \Chow_{\beta}(X).
}
\end{align*}
For an open neighborhood
$[C] \in U \subset \Chow_{\beta}(X)$, 
let $\Sh_U \subset \Sh$, $P_U \subset P$ 
be the open subschemes 
whose underlying spaces are $\pi_{\Sh}^{-1}(U)$, $\pi_P^{-1}(U)$
respectively. 
Suppose that there exist
a 
locally versal deformation 
$\pi_T \colon \cC \to T$
of $C$  
and closed embeddings 
\begin{align*}
i_T \colon U \hookrightarrow T, \ 
i_P \colon P_U \hookrightarrow \cC^{[n]}, \ 
i_{\Sh} \colon {\Sh}_U \hookrightarrow \overline{J}
\end{align*}
such that we have the commutative diagrams
\begin{align*}
\xymatrix{
{\Sh}_U^{\rm{red}} 
 \ar@<-0.3ex>@{^{(}->}[r]^-{\iota} 
 \ar[d]_-{\pi_{\Sh}} 
& {\Sh}_U 
 \ar@<-0.3ex>@{^{(}->}[r]^-{i_{\Sh}} 
& \overline{J} \ar[rd]^-{f_{J}} \ar[d]_-{\pi_{J}} & \\
U  \ar@<-0.3ex>@{^{(}->}[rr]^-{i_T} & & T \ar[r]^-{g} & \mathbb{A}^1,
}
\xymatrix{
P_U^{\rm{red}} \ar@<-0.3ex>@{^{(}->}[r]^-{\iota}  
 \ar[d]_-{\pi_P} 
& P_U 
 \ar@<-0.3ex>@{^{(}->}[r]^-{i_P} 
& \cC^{[n]} \ar[rd]^-{f^{[n]}} \ar[d]_-{\pi^{[n]}} & \\
U  \ar@<-0.3ex>@{^{(}->}[rr]^-{i_T}
 & & T \ar[r]^-{g} & \mathbb{A}^1
}
\end{align*}
giving $d$-critical charts 
\begin{align*}
({\Sh}_U, \overline{J}, f_J, i_{\Sh}), \ 
(P_U, \cC^{[n]}, f^{[n]}, i_P)
\end{align*}
of 
$({\Sh}_U, s_{\Sh}|_{{\Sh}_U})$, $(P_U, s_P|_{P_U})$
respectively.  
Since the $\pi_J$-relative 
canonical line bundle of $\overline{J}$ is trivial
(see~\cite[Theorem~A]{MRV}),  
$({\Sh}, s_{\Sh})$ 
has an orientation data
on ${\Sh}_U^{\rm{red}}$ which is trivial 
as a line bundle, i.e. 
CY orientation in Definition~\ref{loc:cy}. 
Indeed 
the above diagram for ${\Sh}$ 
is a strictly 
CY $d$-critical 
chart 
(see Definition~\ref{defi:slcy} and the diagram (\ref{dia:sCYchart})).
Therefore the local 
GV invariant $n_{g, \gamma}^{\rm{loc}} \in \mathbb{Z}$ is defined
as in Definition~\ref{def:locgv}, using 
a CY orientation data. 
\begin{prop}\label{prop:irred2}
Suppose that the above assumption
holds for all $n\ge 0$. Then 
Conjecture~\ref{GV:conj2} holds for 
$\gamma=[C] \in \Chow_{\beta}(X)$. 
\end{prop}
\begin{proof}
We set $\phi^{[n]}, \phi_J$ as in (\ref{vanish:phi}). 
By the definition of $P_{n, \gamma}^{\rm{loc}}$, we have 
\begin{align*}
\sum_{n\ge 0}
\chi(\dR \pi^{[n]}_{\ast} \phi^{[n]}|_{\gamma})q^{n+1-g} 
&=\sum_{n\ge 0} \chi(\phi^{[n]}|_{P_n^{\rm{loc}}(X, \gamma)}) q^{n+1-g} \\
&=\sum_{n \in \mathbb{Z}}P_{n, \gamma}^{\rm{loc}}q^n.
\end{align*}
Also by the definition of $n_{g, \gamma}^{\rm{loc}}$, 
we have 
\begin{align*}
\sum_{i \in \mathbb{Z}}
\chi(\pH^{i} (\dR \pi_{J\ast}\phi_{J})|_{\gamma}) q^{i}
=\sum_{g\ge 0} 
n_{g, \gamma}^{\rm{loc}} (q^{\frac{1}{2}}+q^{-\frac{1}{2}})^{2g}. 
\end{align*}
The proposition follows by taking the Euler 
characteristics of (\ref{id:phi})
at $i_T(\gamma) \in T$. 
\end{proof}

\section{GV formula for local surfaces}\label{sec:local}
Let $S$ be a smooth projective surface with 
$H^1(\oO_S)=0$.
In this section, we prove Theorem~\ref{intro:thm1}
for the non-compact CY 3-fold $X$
\begin{align}\label{X:total}
p \colon X \cneq \mathrm{Tot}(K_S) \to S. 
\end{align}
This will require a lengthy discussion on the 
natural obstruction theory for sheaves on $S$ and the $d$-critical structure
for sheaves on $X$.  Once these technical details are in place,
we will use them to apply the general approach of the last section.

\subsection{Overview of the proof}
Here we give an outline of the proof of Theorem~\ref{intro:thm1}. 
Let ${\Sh}_X$, $P_X$ be the moduli spaces of one dimensional 
stable sheaves, stable pairs on $X$ respectively. 
Our strategy is to find  
$\Chow(X)$-local $d$-critical charts
of ${\Sh}_X$, $P_X$ via versal deformations of 
irreducible curves with at worst planar singularities
as in the last section. 
We construct such $d$-critical charts 
by relating these moduli spaces with 
similar moduli spaces ${\Sh}_S$, $P_S$ on 
the \textit{surface} $S$. 

Namely we have the natural projections
\begin{align*}
p_{\ast} \colon {\Sh}_X \to {\Sh}_S, \ 
p_{\ast} \colon P_X \to P_S.
\end{align*}
The fiber of the morphism $p_{\ast} \colon {\Sh}_X \to {\Sh}_S$
 at each point $[F] \in {\Sh}_S$ is 
the vector space 
\begin{align*}
\Hom(F, F \otimes K_S)=\Ext_S^2(F, F)^{\vee}.
\end{align*}
If ${\Sh}_S$ is smooth so that $\Ext_S^2(F, F)$ is constant 
for $[F] \in {\Sh}_S$, then 
${\Sh}_X$ is just a vector bundle on ${\Sh}_S$. 
Indeed as $\Ext_S^2(F, F)$ is the obstruction space for the 
deformations of the sheaf $F$ inside $S$, 
the vector bundle ${\Sh}_X \to {\Sh}_S$ is  
nothing but the dual of the obstruction bundle on ${\Sh}_S$. 
In general, ${\Sh}_X$ is the 
\textit{dual obstruction cone}~\cite{JiaTho}
over ${\Sh}_S$, 
determined by a perfect obstruction theory of ${\Sh}_S$. 
The property of dual 
obstruction cones tells us how to 
write ${\Sh}_X$ as a critical locus. 
A similar argument also applies to 
$p_{\ast} \colon P_X \to P_S$. 

Let $H$ be the Chow variety on $S$
(which is smooth by our assumption $H^1(\oO_S)=0$)
with HC maps
\begin{align}\label{mor:PHM}
P_S \to H \leftarrow {\Sh}_S.
\end{align}
For each $[C] \in H$, we will 
take an embedding $H \subset T$
(locally on $C$), where $T$ is the 
base of a locally versal deformation 
$\cC \to T$ of $C$
(see Lemma~\ref{lem:versal}). 
Then we will see that 
the diagram
\begin{align}\label{over:CTJ}
\cC^{[n]} \to T \leftarrow \overline{J}
\end{align}
in the last section restricts to the diagram (\ref{mor:PHM})
on the closed subscheme $H \subset T$. 

As both of $H$, $T$ are smooth, 
locally near $[C] \in H$ we can 
realize $H$ as a zero section of a regular 
section $s$ of a vector bundle $E$. 
Then we can form the following diagram
\begin{align}\label{dia:overview:C}
\xymatrix{
\cC^{[n]} \times_T E^{\vee} \ar[r] \ar[rd]_-{f^{[n]}} & E^{\vee} \ar[d]_-{g} &
\ar[l] \ar[ld]^-{f_J} \overline{J} \times_T E^{\vee} \\
& \mathbb{A}^1
}
\end{align}
Here $g$ is the function on $E^{\vee}$
defined by
\begin{align*}
g(a, e)=(s(a), e), \ a \in T, e \in E^{\vee}|_{a}.
\end{align*}
We will show that, locally 
near $[C] \in H$ we have smooth surjections
\begin{align}\label{PM:surj}
P_X \twoheadrightarrow \{df^{[n]}=0\}, \ 
{\Sh}_X \twoheadrightarrow \{df_J=0\}, 
\ \Chow(X) \twoheadrightarrow Q \subset E^{\vee}
\end{align}
for some closed subset $Q \subset E^{\vee}$, 
compatible with the maps in (\ref{dia:overview:C})
and HC maps on $P_X$, ${\Sh}_X$.
The fibers of the maps in (\ref{PM:surj}) are 
$H^1(\oO_C(C))^{\vee}$, 
so they are isomorphisms if $H^1(\oO_C(C))=0$.
Then the same argument of Proposition~\ref{prop:irred2}
shows the desired result. 

The existence of smooth 
surjections (\ref{PM:surj})
is the most technical part in the proof. 
For this purpose, we will 
carefully compare the obstruction theories on ${\Sh}_S$, $P_S$
induced by their deformation theories
with those induced by the embedding of (\ref{mor:PHM})
into (\ref{over:CTJ}), and 
investigate the Chow variety on $X$.

\subsection{Chow variety on local surfaces}
Let us take an algebraic curve class 
$\beta \in H_2(S, \mathbb{Z})$. 
By the assumption $H^1(\oO_S)=0$,  
there is a unique $L_{\beta} \in \Pic(S)$
such that
$c_1(L_{\beta})=\beta$. 
The Chow variety on $S$ is just the linear system
\begin{align*}
\Chow_{\beta}(S)=\lvert L_{\beta} \rvert. 
\end{align*} 
In this case, $\Chow_{\beta}(S)$ is also 
identified with the Hilbert scheme of pure 
one-dimensional subschemes in $S$ with homology class $\beta$. 
We define the open subscheme
\begin{align*}
H_{\beta}
\subset \Chow_{\beta}(S)
\end{align*} 
consisting of 
irreducible one-cycles. 
Note that any curve in $H_{\beta}$ has 
arithmetic genus $g$ given by
\begin{align}\label{genus}
g=1+\frac{1}{2}\beta(K_S+\beta). 
\end{align}
Let $\cC_{\beta}$ be the universal curve
\begin{align*}
\pi_{H} \colon 
\cC_{\beta} \subset S \times H_{\beta} \to H_{\beta}. 
\end{align*} 
By~\cite[Appendix]{KoTho}, 
there is a perfect obstruction theory
on $H_{\beta}$
\begin{align*}
\uU_H^{\bullet} \cneq 
(\dR \pi_{H\ast} \oO_{\cC_{\beta}}(\cC_{\beta}))^{\vee}
\to \mathbb{L}_{H_{\beta}}. 
\end{align*}
As $H_{\beta}$ is smooth, 
we have the distinguished triangle
\begin{align*}
R^1 \pi_{H \ast} \oO_{\cC_{\beta}}(\cC_{\beta})^{\vee}[1] \to 
\uU_H^{\bullet} \to \mathbb{L}_{H_{\beta}}. 
\end{align*}

For the non-compact CY 3-fold (\ref{X:total}), 
we have the push-forward map
\begin{align}\label{map:p}
p_{\ast} \colon \Chow_{\beta}(X) \to \Chow_{\beta}(S). 
\end{align}
\begin{lem}\label{lem:chow}
For $[C] \in H_{\beta}$, 
the set of closed points 
of $(p_{\ast})^{-1}([C])$ is identified with 
$H^0(\nu_{\ast}\oO_{\widetilde{C}} \otimes K_S)$. 
Here $\nu \colon \widetilde{C} \to C$ is the normalization of $C$. 
\end{lem}
\begin{proof}
Let $C' \subset X$ be an irreducible
curve with $p(C')=C$. 
By taking the diagonal map 
$C' \hookrightarrow X \times_S X=\mathrm{Tot}(p^{\ast}K_S)$, 
we obtain the 
section of $H^0(p^{\ast}K_S|_{C'})=
H^0(p_{\ast}\oO_{C'} \otimes K_S)$. 
Since the normalization $\nu$ factors through 
$p|_{C'}$
\begin{align*}
\nu \colon 
\widetilde{C} \to C' \stackrel{p|_{C'}}{\to} C
\end{align*} we
obtain the section of $\nu_{\ast}\oO_{\widetilde{C}} \otimes K_S$. 
Conversely a section of $\nu_{\ast}\oO_{\widetilde{C}} \otimes K_S$
gives a morphism $\widetilde{C} \to \mathrm{Tot}(\nu^{\ast}K_S)$.
By composing it with the projection 
$\mathrm{Tot}(\nu^{\ast}K_S) \to X$, 
we obtain the morphism $\widetilde{C} \to X$
whose image gives a point of $(p_{\ast})^{-1}([C])$. 
\end{proof}

\subsection{Moduli spaces of one-dimensional stable sheaves on surfaces}
Let $\Sh_{\beta}(S)$ be the moduli space of 
one-dimensional stable sheaves $F$ on $S$
satisfying $[F] =\beta$
and $\chi(F)=1$. 
We set 
\begin{align*}
{\Sh}_{S}
 \subset \Sh_{\beta}(S)
\end{align*}
 to be the open 
subscheme consisting of sheaves $F$
whose fundamental cycles are irreducible. 
Then we have the Hilbert-Chow morphism
\begin{align}\label{MnS}
\rho_{\Sh} \colon 
{\Sh}_{S} \to H_{\beta}
\end{align}
sending $F$ to its fundamental cycle. 
\begin{rmk}
Since $S$ is a surface the morphism (\ref{MnS})
is defined without taking the 
reduced part of ${\Sh}_S$. 
Indeed let $\fF \in \Coh(S \times T)$ be a flat family of 
objects in ${\Sh}_S$. 
Then there is an exact sequence
\begin{align*}
0 \to \fF^{-1} \to \fF^{0} \to \fF \to 0
\end{align*}
for vector bundles $\fF^i$. 
By taking the determinants, we obtain the global 
section of 
$\det(\fF^0) \otimes \det(\fF^{-1})^{-1}=L_{\beta} \boxtimes \oO_T$, giving 
the $T$-valued point of 
$H_{\beta}$. 
\end{rmk}
By Lemma~\ref{lem:versal},
 there is a locally versal deformation $\pi_T \colon \cC_{T} \to T$
and a closed embedding $i \colon H_{\beta} \hookrightarrow T$
such that we have the
Cartesian square
\begin{align}\label{dia:C}
\xymatrix{
\cC_{\beta} 
\ar@<-0.3ex>@{^{(}->}[r]
\ar[d]_{\pi_{H}}\ar@{}[dr]|\square &
 \cC_T \ar[d]^{\pi_{T}} \\
H_{\beta} \ar@<-0.3ex>@{^{(}->}[r]_{i} & T. 
}
\end{align}
Let 
\begin{align*}
\pi_{J} \colon 
\overline{J} \to T
\end{align*}
be the $\pi_T$-relative moduli space of 
rank one torsion free sheaves 
with Euler characteristic one, which 
is non-singular by Theorem~\ref{thm:versal}. 
We have the Cartesian square 
\begin{align}\label{dia:M}
\xymatrix{
{\Sh}_{S}
\ar@<-0.3ex>@{^{(}->}[r]
 \ar[d]_{\rho_{\Sh}}\ar@{}[dr]|\square & \overline{J}
\ar[d]^{\pi_{J}} \\
H_{\beta} \ar@<-0.3ex>@{^{(}->}[r]_{i} & T. 
}
\end{align}
Since $\pi_J$ is flat (see~\cite[Theorem~C (ii)]{MRV}),
the above diagram in particular
implies that 
${\Sh}_{S}$ has at worst
locally complete intersection singularities. 

We next discuss obstruction theories on ${\Sh}_S$. 
Let
$\cC_{\beta}'$
be the universal curve over ${\Sh}_S$
\begin{align*}
\cC_{\beta}'
 \cneq \cC_{\beta}\times_{H_{\beta}} {\Sh}_S
\subset S \times {\Sh}_S
\end{align*}
and $\mathbb{F}$ the universal sheaf
\begin{align*}
\mathbb{F} \in \Coh(\cC_{\beta}') \subset \Coh(S \times {\Sh}_S). 
\end{align*}
By a standard deformation theory of sheaves, we 
 have the perfect obstruction theory on ${\Sh}_S$
\begin{align}\label{obs:U}
\uU_{\Sh}^{\bullet} \cneq 
(\tau_{\ge 1}\dR \hH om_{p_{\Sh}}(\mathbb{F}, \mathbb{F}))^{\vee}[-1] \to 
\mathbb{L}_{{\Sh}_S}
\end{align}
where
$p_{\Sh} \colon {\Sh}_S \times S \to {\Sh}_S$ is the projection. 
\begin{prop}\label{prop:natU}
There is a natural distinguished triangle 
\begin{align}\label{nat:U}
\rho_{\Sh}^{\ast}
R^1 \pi_{H\ast} \oO_{\cC_{\beta}}(\cC_{\beta})^{\vee}[1] \to 
\uU_{\Sh}^{\bullet} \to \mathbb{L}_{{\Sh}_S}.
\end{align}
\end{prop}
\begin{proof}
Let us consider the natural morphism in $D^b(S \times {\Sh}_S)$
\begin{align*}
\oO_{\cC_{\beta}'} \otimes p_S^{\ast}K_S
\to \dR \hH om_{S \times {\Sh}_S}(\mathbb{F}, \mathbb{F})
 \otimes p_S^{\ast}K_S. 
\end{align*}
By pushing forward to ${\Sh}_S$ and using the
Grothendieck duality, we obtain the morphism
in $D^b({\Sh}_S)$
\begin{align}\label{homF}
\dR \hH om_{p_{\Sh}}(\mathbb{F}, \mathbb{F})[1] \to 
\dR \hH om_{p_{\Sh}}(\oO_{\cC_{\beta}'}, \oO_{{\Sh}_S \times S})[1].
\end{align}
Note that the RHS is identified with
\begin{align*}
\dR p_{{\Sh}\ast} \oO_{\cC_{\beta}'}(\cC_{\beta}')
=
\rho_{\Sh}^{\ast}\dR \pi_{H\ast}\oO_{\cC_{\beta}}(\cC_{\beta}). 
\end{align*}
By taking the truncations $\tau_{\ge 1}$ of (\ref{homF}) and 
dualizing, we obtain the morphism 
\begin{align*}
\rho_{\Sh}^{\ast}\uU^{\bullet}_H \to \uU^{\bullet}_{\Sh}. 
\end{align*}
Let $\gG^{\bullet}_{\Sh}$ be the cone of the above morphism. 
Then there is a morphism 
$\gG_{\Sh}^{\bullet} \to \mathbb{L}_{{\Sh}_S/H_{\beta}}$ which fits into 
the morphism of distinguished triangles:
\begin{align}
\label{tri:perf}
\xymatrix{
\rho_{\Sh}^{\ast}\uU_H^{\bullet} \ar[r]\ar[d] & \uU_{\Sh}^{\bullet} \ar[r] \ar[d] &
\gG_{\Sh}^{\bullet} \ar[d] \\
\rho_{\Sh}^{\ast}\mathbb{L}_{H_{\beta}} \ar[r] & \mathbb{L}_{{\Sh}_S} \ar[r] &
\mathbb{L}_{{\Sh}_S/H_{\beta}}. 
}
\end{align}

We claim that 
the morphism
$\gG_{\Sh}^{\bullet} \to \mathbb{L}_{{\Sh}_S/H_{\beta}}$ is a quasi-isomorphism. 
First we check that $\gG_{\Sh}^{\bullet}$
is concentrated on $[-1, 0]$. 
For a sheaf $F$ giving a closed point $p\in {\Sh}_S$, 
let $C \subset S$ be the
support of $F$.  
By the top distinguished triangle in (\ref{tri:perf}), 
we have 
the exact sequence
\begin{align*}
0 \to \hH^{-2}(\gG_{\Sh}^{\bullet}|_{p}) \to 
H^1(\oO_C(C))^{\vee} \to \Ext_S^1(F, F)^{\vee}. 
\end{align*}
By the Serre duality, the right morphism is 
identified with the natural morphism 
\begin{align}\notag
H^0(K_S|_{C}) \to \Hom(F, F \otimes K_S)
\end{align} which is 
clearly injective as $F$ is 
a torsion free sheaf on $C$. Therefore 
$\hH^{-2}(\gG_{\Sh}^{\bullet}|_{p})=0$ and 
$\gG_{{\Sh}}^{\bullet}$ is
concentrated on $[-1, 0]$. 
Also note 
that the left and the middle morphisms in the diagram (\ref{tri:perf})
satisfy that $\hH^0(\ast)$ are isomorphisms and $\hH^{-1}(\ast)$ are
 surjective. 
Then an easy diagram chasing  
shows that $\gG_{\Sh}^{\bullet} \to \mathbb{L}_{{\Sh}_S/H_{\beta}}$ is a 
$\rho_{\Sh}$-relative 
perfect obstruction theory. 
Its virtual dimension is 
\begin{align*}
1-\chi(F, F)-\chi(\oO_C(C))=g
\end{align*} 
by the Riemann-Roch theorem, where 
$g$ is the arithmetic genus of $C$ 
given by (\ref{genus}). 

On the other hand, since $T$ and $\overline{J}$ 
in the diagram (\ref{dia:M})
are smooth, it follows that 
\begin{align*}
\mathbb{L}_{{\Sh}_S/H_{\beta}}=
(\pi_J^{\ast}\Omega_T \to \Omega_{\overline{J}})|_{{\Sh}_S}. 
\end{align*}
Therefore $\id \colon \mathbb{L}_{{\Sh}_S/H_{\beta}} 
\to \mathbb{L}_{{\Sh}_S/H_{\beta}}$ is also a $\rho_{\Sh}$-relative
perfect obstruction theory, 
whose virtual dimension is also $g$ as 
$\pi_J$ has relative dimension $g$.  
Therefore $\gG_{\Sh}^{\bullet} \to \mathbb{L}_{{\Sh}_S/H_{\beta}}$ must 
be a quasi-isomorphism.

By taking the cones of 
the diagram (\ref{tri:perf}), 
we obtain the distinguished triangle (\ref{nat:U}). 
\end{proof}

\subsection{Dual obstruction cone}\label{subsec:cone}
In general, let $M$ be a complex scheme with a 
perfect obstruction theory 
$\uU^{\bullet} \to \mathbb{L}_M$. The 
\textit{dual obstruction cone}
is a cone over $M$ defined by
\begin{align}\label{obs:cone}
\mathrm{Obs}^{\ast}(\uU^{\bullet}) 
\cneq \Spec_{\oO_{M}} \left(
\bigoplus_{i\ge 0}\Sym^i(\hH^1(\uU^{\bullet \vee})) \right).
\end{align}
By~\cite{JiaTho},
the dual obstruction cone (\ref{obs:cone}) is locally 
written as a critical locus
of some function on a smooth scheme. 
Suppose that $M$ 
is cut out by a section $s$ of a vector 
bundle $E \to A$ on a smooth scheme $A$, and 
the obstruction theory $\uU^{\bullet}$ is given by
\begin{align}\label{UAE}
\xymatrix{
\uU^{\bullet}= 
(E^{\vee}|_{M} \ar[r]^-{ds}\ar@<3.0ex>@{->>}[d]_{s}
 & \Omega_A|_{M} \ar@{=}[d]) \\
\mathbb{L}_M =  (I/I^2 \ar[r] & \Omega_A|_{M}).
}
\end{align}
Here $I \subset \oO_A$ is the ideal sheaf of $M$ in $A$. 
Then the dual obstruction cone (\ref{obs:cone}) is 
written as the critical locus of the function
\begin{align}\label{crit:f}
f \colon \mathrm{Tot}(E^{\vee}) \to \mathbb{A}^1, \ 
f(a, e)=(s(a), e)
\end{align}
where $a \in A$ and $e \in E^{\vee}|_{a}$
(see~\cite[Section~2]{JiaTho} for details). 
Since any perfect obstruction theory is 
locally of the form (\ref{UAE}), 
the dual obstruction cone (\ref{obs:cone}) is locally written as a 
critical locus. 

Suppose that $M$ has at worst locally complete
intersection singularities. 
Then $\id \colon \mathbb{L}_M \to \mathbb{L}_M$ is 
a perfect obstruction theory, 
and $\mathrm{Obs}^{\ast}(\mathbb{L}_M)$ 
carries a natural $d$-critical structure which locally
is given by the above construction.
Indeed, the dual obstruction cone is the underlying scheme 
of the $(-1)$-shifted cotangent derived scheme
(see~\cite[Definition~1.20]{PTVV})
$$\Omega_{M}^{\bullet}[-1] \to M$$
and the $d$-critical structure 
is induced by the natural $(-1)$-shifted symplectic structure.

In the above situation, suppose that 
$\uU^{\bullet} \to \mathbb{L}_M$ is a perfect obstruction theory. 
Then we have the distinguished triangle
\begin{align*}
\vV[1] \to \uU^{\bullet} \to \mathbb{L}_M
\end{align*}
for a vector bundle $\vV$ on $M$. 
By taking the dual and the long exact sequence of 
cohomology, we have the exact sequence
\begin{align*}
0 \to \hH^1(\mathbb{L}_M^{\vee}) \to \hH^1(\uU^{\bullet \vee}) \to 
\vV^{\vee} \to 0. 
\end{align*}
Therefore we have the smooth surjection
\begin{align}\notag
\mathrm{Obs}^{\ast}(\uU^{\bullet}) \twoheadrightarrow 
\mathrm{Obs}^{\ast}(\mathbb{L}_M). 
\end{align}

Assume the perfect obstruction theory on $M$ has a local model as given at the beginning of this section.
One can then show that the pullback of
the canonical $d$-critical structure
on $\mathrm{Obs}^{\ast}(\mathbb{L}_M)$ agrees
with the $d$-critical structure
\begin{align}\label{dcrit:U}
(\mathrm{Obs}^{\ast}(\uU^{\bullet}), s_{\uU}), \ 
s_{\uU} \in \Gamma(\sS_{\mathrm{Obs}^{\ast}(\uU^{\bullet})}^{0}). 
\end{align}
Indeed, since the underlying scheme structure on $M$ is lci, one can show the dg-scheme structure
is locally split on $M$, which implies the claim.
Locally, the $d$-critical structure can be defined by critical charts, as described in equation~\eqref{crit:f}.

\subsection{Moduli spaces of stable sheaves on local surfaces}

We apply the above dual obstruction cone construction to
the perfect obstruction theory 
$\uU_{\Sh}^{\bullet} \to \mathbb{L}_{{\Sh}_S}$
given in (\ref{obs:U}), and compare 
it with the moduli space of one-dimensional 
stable sheaves on $X=\mathrm{Tot}(K_S)$. 
Let $\Sh_{\beta}(X)$ be the moduli space of 
compactly supported 
one-dimensional stable sheaves
$E$ on $X$ with $[p_{\ast}E]=\beta$
and $\chi(E)=1$. 
Let
\begin{align*}
{\Sh}_{X} \subset \Sh_{\beta}(X)
\end{align*}
be the open 
subset corresponding to sheaves $E$ such that
the fundamental cycle of $p_{\ast}E$ is irreducible. 
\begin{lem}\label{lem:Mobs}
We have the following commutative diagram: 
\begin{align}\label{dia:MObs}
\xymatrix{
{\Sh}_X \ar[rr]^{\cong} \ar[dr]_{p_{\ast}}
  &  & \mathrm{Obs}^{\ast}(\uU_{\Sh}^{\bullet}) \ar[ld] \\
& {\Sh}_S. &
}
\end{align}
Here the top morphism is a
canonical isomorphism, 
$p_{\ast}$ is induced by the projection 
$p \colon X \to S$ and 
the right morphism is the natural morphism 
defined from the cone structure.
\end{lem}
\begin{proof}
For a closed point $[E] \in {\Sh}_X$, 
the sheaf $p_{\ast}E$ on $S$ 
is stable as it is pure and
has irreducible support. 
Therefore we have the morphism 
$p_{\ast} \colon {\Sh}_X \to {\Sh}_S$. 
The fiber of this morphism at $[F] \in {\Sh}_S$ is 
given by 
the $\oO_X$-module structures on $F$, 
that is 
$\Hom(F, F \otimes K_S)$. 
On the other hand, the fiber of 
the right morphism of (\ref{dia:MObs}) 
at $[F] \in {\Sh}_S$ is given by 
$\Ext_S^2(F, F)^{\vee}$ which is isomorphic 
to $\Hom(F, F \otimes K_S)$.
Therefore the fibers of left and right morphisms in (\ref{dia:MObs})
are canonically identified.  
It is straightforward to generalize this argument for 
flat families of sheaves in ${\Sh}_S$, which shows the 
existence of a canonical 
isomorphism in the diagram (\ref{dia:MObs}). 
\end{proof}
By the diagram (\ref{dia:M}), 
${\Sh}_S$ has at worst locally complete intersection 
singularities.
Therefore 
the construction in (\ref{dcrit:U})
yields the $d$-critical scheme
\begin{align}\label{MX:dcrit}
({\Sh}_X, s_{\Sh}), \ s_{\Sh} \in \Gamma(\sS_{{\Sh}_X}^{0}).
\end{align}
Let us take the dual and the long exact sequence of 
cohomologies of (\ref{nat:U}).
Then we obtain the exact 
sequence of sheaves
\begin{align}\label{ex:U}
0 \to \hH^{1}(\mathbb{L}_{{\Sh}_S}^{\vee}) \to \hH^1(\uU_{\Sh}^{\bullet \vee}) \to 
\rho_{\Sh}^{\ast}R^1 \pi_{H\ast} \oO_{\cC_{\beta}}(\cC_{\beta}) \to 0. 
\end{align}
We define $O_H$ to be the total space of the vector bundle 
\begin{align}\label{def:O}
O_{H} \cneq 
R^1 \pi_{H\ast} \oO_{\cC_{\beta}}(\cC_{\beta})^{\vee}
\to H_{\beta}
\end{align}
on $H_{\beta}$. 
By the exact sequence (\ref{ex:U})
and Lemma~\ref{lem:Mobs}, 
the $H_{\beta}$-group scheme
$O_H$ 
acts on ${\Sh}_X$ fiberwise over $H_{\beta}$ without fixed points. 
The quotient space is
\begin{align}\label{quot:M}
\eta_{\Sh} \colon 
{\Sh}_X \twoheadrightarrow 
{\Sh}_X/O_H=\Obs^{\ast}(\mathbb{L}_{{\Sh}_S}). 
\end{align} 
By the Riemann-Roch computation, the 
relative dimension $d$ of $\eta_{\Sh}$ is given by
\begin{align}\label{rel:dim}
d=\dim H_{\beta}-\frac{1}{2}\beta^2 +\frac{1}{2}K_S \cdot \beta. 
\end{align}
By the construction of (\ref{MX:dcrit}), 
the $d$-critical structure $s_{\Sh}$ is pulled 
back from the smooth surjection $\eta_{\Sh}$.

\begin{rmk}\label{rmk:compared}
One subtlety in this discussion
is the compatibility of 
the $d$-critical structure $s_{\Sh}$ in (\ref{MX:dcrit}) 
 with the $d$-critical structure $s_{\Sh}^{\rm{der}}$
induced by the derived deformation theory as in 
Remark~\ref{rmk:shifted}. 
In fact, we expect a stronger matching of $(-1)$-shifted symplectic structures
is known to experts; however,
since a reference is unavailable, we sketch a proof of the weaker
statement as follows.  

First, both $s_{\Sh}$ and $s_{\Sh}^{\rm{der}}$ 
are homogeneous with weight $1$ with
respect to the natural action of $\mathbb{C}^{*}$ on ${\Sh}_X$.
Using this, it suffices to show their formal completions
agree at sheaves $[E] \in {\Sh}_X$ which are pushed forward from $S$, i.e. of the 
form $E = i_{*}F$ for a sheaf $F$ on $S$
and the inclusion $i \colon S \hookrightarrow X$ by the zero section. 

At any such point, the $d$-critical structure $s_{\Sh}^{\rm{der}}$ 
is determined formal-locally by
the cyclic pairing on the $L_{\infty}$-algebra $L_{X,E} = \RHom_{X}(E,E)$ induced by
Serre duality.  In this case, since $[E]$ is $\mathbb{C}^*$-fixed, $L_{X,E}$ inherits 
an extra grading compatible with 
the higher operations and the cyclic structure; using this grading, we can identify
$$L_{X,E} = L_{S,F} \oplus L_{S,F}^{\vee}[1]$$
where $L_{S,F} = \RHom_{S}(F,F)$ and the right-hand side has a cyclic dgla
structure from the coadjoint action and the natural pairing.
If we use this decomposition to compute $s_{\Sh}^{\rm{der}}$ and compare with 
equation~\eqref{crit:f},  
we see that $s_{\Sh} = s_{\Sh}^{\rm{der}}$. 
\end{rmk}

\subsection{Hilbert-Chow map on ${\Sh}_X$}
We set
\begin{align*}
\wH_{\beta} \cneq p_{\ast}^{-1}(H_{\beta}) \subset \Chow_{\beta}(X)
\end{align*}
where $p_{\ast}$ is the morphism (\ref{map:p}). 
We have the Hilbert-Chow map 
$
\pi_{\Sh} \colon {\Sh}_X^{\rm{red}} \to \wH_{\beta}$
for $X$ and the commutative diagram
\begin{align}\label{HC:M}
\xymatrix{
{\Sh}_X^{\rm{red}} \ar[r]^-{p_{\ast}} \ar[d]_-{\pi_{\Sh}} \ar[rd] 
& {\Sh}_S \ar[d]^-{\rho_{\Sh}} \\
\wH_{\beta} \ar[r]_-{p_{\ast}} & H_{\beta}.
}
\end{align}
Below we fix a point
\begin{align}\label{fix:c}
c =[C] \in H_{\beta}
\end{align}
for an irreducible curve
$C \subset S$. 
 Let ${\Sh}_{X, c}^{\rm{red}}$ be the reduced 
fiber at $c$ of the 
morphism ${\Sh}_X^{\rm{red}} \to H_{\beta}$ in the 
diagram (\ref{HC:M}). 
Let $\nu \colon \widetilde{C} \to C$ be the normalization. 
By Lemma~\ref{lem:chow}, we obtain the morphism
\begin{align}\label{mor:reduced}
\pi_{{\Sh}, c} \colon 
{\Sh}_{X, c}^{\rm{red}} \to H^0(\nu_{\ast}\oO_{\widetilde{C}} \otimes K_S).
\end{align} 
Note that the fiber of (\ref{def:O}) at $c$
is 
\begin{align*}
H^1(\oO_C(C))^{\vee}=H^0(K_S|_{C}).
\end{align*}
The $O_H$-action on ${\Sh}_X$ induces the 
$H^0(K_S|_{C})$-action on ${\Sh}_{X, c}^{\rm{red}}$. 
On the other hand, 
$H^0(K_S|_{C})$ is contained in 
$H^0(\nu_{\ast}\oO_{\widetilde{C}} \otimes K_S)$, 
hence acts on it by the translation. 
\begin{lem}\label{lem:equiv}
The morphism (\ref{mor:reduced})
is equivariant with respect to the $H^0(K_S|_{C})$-action.  
\end{lem}
\begin{proof}
A closed point of ${\Sh}_{X, c}^{\rm{red}}$ consists of 
a pair $(F, s)$, where $F \in \Coh(S)$
gives a closed point of ${\Sh}_S$ 
with fundamental cycle $c$, 
and $s$ is a morphism 
$s \colon F \to F \otimes K_S$
(see the proof of Lemma~\ref{lem:Mobs}). 
Let $\widetilde{F}=\nu^{\ast}F/T$
where $T \subset \nu^{\ast}F$ is the torsion part. 
Since $\widetilde{C}$ is smooth, 
$\widetilde{F}$ is a line bundle on $\widetilde{C}$ and 
we have the embedding 
\begin{align}\label{inc:nu}
\eE nd(F) \subset \nu_{\ast} \eE nd(\widetilde{F})=\nu_{\ast}
 \oO_{\widetilde{C}}. 
\end{align}
Combined with the inclusion $\oO_{C} \subset \eE nd(F)$, we have 
\begin{align}\label{incs}
H^0(K_S|_{C}) \subset \Hom(F, F \otimes K_S)
\subset H^0(\nu_{\ast}\oO_{\widetilde{C}} \otimes K_S). 
\end{align}
The $H^0(K_S|_{C})$-action on 
${\Sh}_{X, c}^{\rm{red}}$ is given by 
the translation with respect to the
first embedding of (\ref{incs}). 
Also the morphism (\ref{mor:reduced}) is given by 
the second embedding of (\ref{incs}). 
Therefore (\ref{mor:reduced}) is $H^0(K_S|_{C})$-equivariant.  
\end{proof}
For the normalization $\nu \colon \widetilde{C} \to C$, we 
set $Q_c$ to be
\begin{align*}
Q_c \cneq \nu_{\ast}\oO_{\widetilde{C}}/\oO_C.
\end{align*}
\begin{lem}\label{lem:quot}
We have 
\begin{align*}
H^0(\nu_{\ast}\oO_{\widetilde{C}} \otimes K_S)/
H^0(K_S|_{C})=\overline{Q}_c \cneq \Ker\left(Q_c
\stackrel{u}{\to} 
\Omega_{H_{\beta}}|_{c} \right)
\end{align*}
for some morphism $u$. 
\end{lem}
\begin{proof}
By the exact sequence
\begin{align*}
0 \to K_S|_{C} \to \nu_{\ast}\oO_{\widetilde{C}} \otimes
K_S \to Q_c \to 0
\end{align*}
we obtain the long exact sequence
\begin{align*}
0 \to H^0(K_S|_{C}) \to H^0(\nu_{\ast}\oO_{\widetilde{C}} \otimes K_S)\to Q_c \to H^1(K_S|_{C}).
\end{align*}
Then the lemma follows from the identification 
\begin{align*}
H^1(K_S|_{C})=H^0(\oO_C(C))^{\vee}=\Omega_{H_{\beta}}|_{c}.
\end{align*}
\end{proof}
Let $\Obs^{\ast}(\bL_{{\Sh}_S})_c^{\rm{red}}$ be the 
reduced fiber of 
the composition at $c=[C] \in H_{\beta}$:
\begin{align*}
\Obs^{\ast}(\bL_{{\Sh}_S})^{\rm{red}} \to {\Sh}_S \stackrel{\rho_{\Sh}}{\to} H_{\beta}
\end{align*}
where the first morphism is the projection. 
By Lemma~\ref{lem:equiv}
and Lemma~\ref{lem:quot}, 
we obtain the Cartesian square
\begin{align}\label{car:LM}
\xymatrix{
{\Sh}_{X, c}^{\rm{red}} \ar[r] \ar[d]_{\pi_{{\Sh}, c}}
\ar@{}[dr]|\square
 &
\Obs^{\ast}(\bL_{{\Sh}_S})_c^{\rm{red}} \ar[d] \\
H^0(\nu_{\ast}\oO_{\widetilde{C}} \otimes K_S) \ar[r] &
\overline{Q}_c
}
\end{align} 
such that horizontal morphisms are smooth 
morphisms with fiber $H^0(K_S|_{C})$. 
\subsection{CY condition for ${\Sh}_X$}
We consider the 
embedding $i \colon H_{\beta} \hookrightarrow T$
in the 
 diagram (\ref{dia:C}). 
We take  
an open neighborhood $U$ in $T$
of the point (\ref{fix:c}), 
and a bundle $E$ on it with a section $s$
\begin{align}\label{etale:U}
c \in U \subset T, \  
E \to U, \ s \in \Gamma(U, E)
\end{align}
such that 
$s$ is a regular section 
which cuts out $H_{\beta}$, i.e. 
\begin{align*}
H_{U} \cneq H_{\beta} \cap U =(s=0) 
\subset U.
\end{align*} 
By taking the fiber 
products of the diagrams (\ref{dia:M}), (\ref{HC:M})
with $U$ over $T$, we obtain the diagrams
\begin{align}\label{dia:XS}
\xymatrix{
{\Sh}_{S, U}
\ar@<-0.3ex>@{^{(}->}[r]
 \ar[d]_-{\rho_{\Sh}}\ar@{}[dr]|\square & \overline{J}_U
\ar[d]^-{\pi_{J}} \\
H_{U} \ar@<-0.3ex>@{^{(}->}[r]_-{i} & U, 
}
\
\xymatrix{
{\Sh}_{X, U}^{\rm{red}} \ar[r]^-{p_{\ast}} \ar[d]_-{\pi_{\Sh}} \ar[rd] 
& {\Sh}_{S, U} \ar[d]^-{\rho_{\Sh}} \\
\wH_{U} \ar[r]_-{p_{\ast}} & H_{U}.
}
\end{align}
Since $s$ is a regular 
section of $E$, we
have the isomorphism
\begin{align}\notag
(\pi_J^{\ast}E^{\vee}|_{{\Sh}_{S, U}}
\stackrel{d(\pi_J^{\ast} s)}{\longrightarrow}
\Omega_{\overline{J}_U}|_{{\Sh}_{S, U}}) \stackrel{\cong}{\to} 
\mathbb{L}_{{\Sh}_{S, U}}. 
\end{align}
By the argument in Subsection~\ref{subsec:cone},  
we have the commutative diagram
\begin{align}\label{dia:MR2}
\xymatrix{
\mathrm{Obs}^{\ast}(\mathbb{L}_{{\Sh}_{S, U}}) \stackrel{\cong}{\to}
\{df_{J}=0\}
 \ar@<-0.3ex>@{^{(}->}[r]
& \overline{J}_U \times_U E^{\vee}
 \ar[rd]^-{f_J} \ar[d]_-{\pi_J} & \\
 & E^{\vee} \ar[r]^-{g} & \mathbb{A}^1
}
\end{align}
giving a $d$-critical chart of 
$\mathrm{Obs}^{\ast}(\mathbb{L}_{{\Sh}_{S, U}})$. 
Here $g$ is the function 
defined as in (\ref{crit:f}), i.e.
\begin{align*}
g(a, e) =(s(a), e), \ 
a \in U, \ 
e \in E_a^{\vee} \cneq E^{\vee}|_{a}.
\end{align*}
By shrinking $U$ if necessary, 
the canonical line bundle of 
 $\overline{J}_U$ is trivial
(see~\cite[Theorem~A]{MRV}).
Then
the above $d$-critical 
structure is strictly CY
and (\ref{dia:MR2}) 
is a  
CY $d$-critical chart (see Definition~\ref{defi:slcy}).


\begin{lem}
Let 
$\overline{J}_C$ be the 
fiber of $\pi_J \colon \overline{J} \to T$ at 
the point (\ref{fix:c}), and 
$E_c$ the fiber of $E \to U$ at 
$c \in H_{U} \subset U$. 
There is a closed embedding 
$\overline{Q}_c \subset E_c^{\vee}$ 
such that the following diagram commutes: 
\begin{align}\label{dia:TE}
\xymatrix{
\Obs^{\ast}(\bL_{{\Sh}_S})_c^{\rm{red}} 
\ar@<-0.3ex>@{^{(}->}[r]
\ar[d]
& \overline{J}_C \times E_c^{\vee} \ar[d] \\
\overline{Q}_c \ar@<-0.3ex>@{^{(}->}[r] & E_c^{\vee}. 
}
\end{align}
Here the top morphism is 
induced by the embedding (\ref{dia:MR2}), and
the left morphism is given in (\ref{car:LM}).
\end{lem}
\begin{proof}
By the description (\ref{quot:M}), 
$\Obs^{\ast}(\bL_{{\Sh}_S})_c^{\rm{red}}$ is a cone
over $\overline{J}_C$ whose fiber at a closed point $[F] \in \overline{J}_C$ is the quotient of $\Hom(F, F\otimes K_S)$ by the 
action of $H^0(K_S|_{C})$.
Similarly to Lemma~\ref{lem:quot}, we see that 
\begin{align}\label{endF}
\Hom(F, F\otimes K_S)/H^0(K_S|_{C})=
\Ker\left( \eE nd(F)/\oO_C \to \Omega_{H_{\beta}}|_{c} \right)
\end{align}
which is a closed subspace of $E_c^{\vee}$
by the diagram (\ref{dia:MR2}). 
We have the embedding (\ref{inc:nu}), such that the 
equality holds if $F=\nu_{\ast}L$ for some line bundle $L$ on $\widetilde{C}$. 
Applying (\ref{endF}) for such $F$, we obtain the 
embedding $\overline{Q}_c \subset E_c^{\vee}$ 
which makes the diagram (\ref{dia:TE})
commutative. 
\end{proof}
We write $O_{H, U} =O_H \times_{H_{\beta}}H_U$, where 
$O_H$ is given by (\ref{def:O}). 
By the above lemma, 
we have the commutative diagram
\begin{align}\label{Car:closed}
\xymatrix{
{\Sh}_{X, U}^{\rm{red}} \ar[r]^(.4){\eta_{\Sh}} \ar[d]_{\pi_{\Sh}}
\ar@{}[dr]|\square & \Obs^{\ast}(\mathbb{L}_{{\Sh}_{S, U}})^{\rm{red}} \ar[d]
\ar@<-0.3ex>@{^{(}->}[r]
 & \overline{J}_U \times_U E^{\vee} \ar[d]^{\pi_J}  \\
\wH_{U} \ar[r]^(.4){\eta}
 & \wH_{U}/O_{H, U} \ar[r] & E^{\vee}.
} 
\end{align}
where the
left horizontal morphisms are 
quotient maps with respect to the free 
$O_{H, U}$-actions
and the right horizontal
morphisms are bijections onto their images. 
Combined with the diagram (\ref{dia:MR2}), we obtain the following: 
\begin{cor}\label{thm:CYloc}
The $d$-critical scheme
$({\Sh}_X, s_{\Sh})$ in 
 (\ref{MX:dcrit}) 
together with the morphism 
${\Sh}_X^{\rm{red}} \to \wH_{\beta}$ in (\ref{HC:M}) is a
CY fibration in the sense of Definition~\ref{loc:cy}. 
\end{cor}

\subsection{GV invariants on local surfaces}
We keep the notation in the 
previous subsection. 
Let
\begin{align*}
\phi_{\sS h} \in \Perv({\Sh}_{X, U})
\end{align*}
be the perverse sheaf of vanishing cycles on ${\Sh}_{X, U}$
determined by its canonical $d$-critical structure
with a CY orientation data (see Corollary~\ref{thm:CYloc}). 
Let 
\begin{align*}
\phi_{J} \cneq 
\phi_{f_J}(\IC(\overline{J}_U \times_{U} E^{\vee}))
  \in \Perv(\mathrm{Obs}^{\ast}(\mathbb{L}_{{\Sh}_{S, U}}))
\end{align*}
be the perverse sheaf of vanishing cycles 
associated to the CY $d$-critical chart (\ref{dia:MR2}). 
\begin{lem}\label{lem:defphiM}
We can choose a CY orientation data on ${\Sh}_{X, U}$ such that 
 the following identity holds: 
\begin{align}\notag
\phi_{\sS h} = \eta_{\Sh}^{\ast} \phi_{J}[d]
\in \mathrm{Perv}({\Sh}_{X, U}). 
\end{align}
Here $\eta_{\Sh}$ is the quotient map (\ref{quot:M}), 
and $d$ is the relative dimension of $\eta_{\Sh}$
given by (\ref{rel:dim}).  
\end{lem}
\begin{proof}
The lemma follows 
since the $d$-critical structure and the virtual 
canonical line bundle on ${\Sh}_{X, U}$ are pulled
back from $\mathrm{Obs}^{\ast}(\mathbb{L}_{{\Sh}_{S, U}})$.
\end{proof}  

Let us take a one-cycle $\gamma$ on $X$
 (see the diagram (\ref{HC:M}))
\begin{align}\label{take:gamma}
\gamma \in p_{\ast}^{-1}(c) \subset
\widehat{H}_U \subset
 \widehat{H}_{\beta}
\subset \Chow_{\beta}(X)
\end{align}
where $c=[C] \in H_{\beta}$ is 
the point (\ref{fix:c}). 
By Definition~\ref{def:locgv}, 
the local GV invariant
$n_{g, \gamma}^{\rm{loc}} \in \mathbb{Z}$ is given by 
\begin{align}\notag
\sum_{i\in \mathbb{Z}} \chi(\pH^i(\dR \pi_{{\Sh}\ast} 
\phi_{\sS h})|_{\gamma})
y^i =\sum_{g\ge 0} 
n_{g, \gamma}^{\rm{loc}}(y^{\frac{1}{2}}+y^{-\frac{1}{2}})^{2g}. 
\end{align}
The following lemma obviously follows from 
the diagram (\ref{Car:closed}) 
together with 
Lemma~\ref{lem:defphiM}. 
\begin{lem}\label{lem:relphi}
Let $\overline{\gamma} \in E^{\vee}$ be 
the image of $\gamma$ under the map 
$\widehat{H}_U \to E^{\vee}$ 
in the diagram (\ref{Car:closed}). 
Then 
we have the following identity (see the diagram (\ref{dia:MR2})): 
\begin{align}\notag
\sum_{i\in \mathbb{Z}}
\chi(\pH^i(\dR \pi_{J\ast} \phi_{J})|_{\overline{\gamma}})y^i 
=(-1)^{d}
\sum_{g\ge 0} n_{g, \gamma}^{\rm{loc}}
(y^{\frac{1}{2}}+y^{-\frac{1}{2}})^{2g}. 
\end{align}
\end{lem}

\subsection{Stable pairs on local surfaces}
Let $P_{n}(S, \beta)$ be the moduli space of 
stable pairs $(F, s)$ on $S$ with 
$[F]=\beta$, $\chi(F)=n$. 
For $n\ge 0$, we set
\begin{align*}
P_S \subset P_{n+1-g}(S, \beta)
\end{align*}
to be the open subscheme of 
stable pairs $(F, s)$ such that the 
fundamental cycle of $F$ is irreducible. 
Here $g$ is the arithmetic genus of 
curves in $H_{\beta}$
given by (\ref{genus}). 
Then we have the Hilbert-Chow map 
$\rho_P \colon P_S \to H_{\beta}$, 
and the Cartesian square by~\cite[Appendix]{PT3}
\begin{align}\label{dia:PS}
\xymatrix{
P_{S} \ar@<-0.3ex>@{^{(}->}[r] \ar[d]_{\rho_P}\ar@{}[dr]|\square & \cC_T^{[n]} 
\ar[d]^{\pi^{[n]}} \\
H_{\beta} \ar@<-0.3ex>@{^{(}->}[r]_{i} & T. 
}
\end{align}
Since $T$ is a base of a locally versal family, 
the relative Hilbert scheme 
$\cC_T^{[n]}$ is non-singular by Theorem~\ref{thm:versal}. 
As $\pi^{[n]}$ is flat (see~\cite[Lemma~2.6]{MY}), 
the diagram (\ref{dia:PS}) implies that 
$P_S$ has only complete 
intersection singularities. 
Let 
\begin{align*}
\mathbb{I}_{S} \cneq (\oO_{S\times P_S} \to \mathbb{F})
\in D^b(S \times P_S)
\end{align*}
be a universal pair on $S \times P_S$. 
By~\cite{KoTho}, we have the perfect obstruction theory
\begin{align*}
\uU_P^{\bullet} \cneq 
\dR \hH om_{p_P}(\mathbb{I}_S, \mathbb{F})^{\vee}
\to \mathbb{L}_{P_S}.
\end{align*}
Here $p_P \colon S \times P_S \to P_S$ is the projection. 
\begin{lem}\label{lem:Ptri}
There is a natural distinguished triangle
\begin{align}\label{tri:Pair}
\rho_P^{\ast}
R^1 \pi_{H\ast} \oO_{\cC_{\beta}}(\cC_{\beta})^{\vee}[1] \to 
\uU_P^{\bullet} \to \mathbb{L}_{P_S}. 
\end{align}
\end{lem}
\begin{proof}
The proof is similar to 
Proposition~\ref{prop:natU}, 
so we just give an outline. 
By~\cite[Appendix]{KoTho}, there is a natural morphism 
$\rho_P^{\ast}\uU_{H}^{\bullet} \to \uU_{P}^{\bullet}$
 such that its 
cone $\gG_P^{\bullet}$ fits into a 
distinguished triangle of perfect obstruction theories:
\begin{align}
\label{tri:perf2}
\xymatrix{
\rho_P^{\ast}\uU_H^{\bullet} \ar[r]\ar[d] & \uU_P^{\bullet} \ar[r] \ar[d] &
\gG_P^{\bullet} \ar[d] \\
\rho_P^{\ast}\mathbb{L}_{H_{\beta}} \ar[r] & \mathbb{L}_{P_S} \ar[r] &
\mathbb{L}_{P_S/H_{\beta}}. 
}
\end{align}
Similarly to the proof of Proposition~\ref{prop:natU}, 
the morphism  
$\gG_P^{\bullet} \to \mathbb{L}_{P_S/H_{\beta}}$ is 
shown to be a quasi-isomorphism. 
Therefore taking the 
cones of (\ref{tri:perf2})
gives the desired distinguished triangle (\ref{tri:Pair}).
\end{proof}
Let $P_n(X, \beta)$
be the moduli space of stable 
pairs $(E, s)$ on $X=\mathrm{Tot}(K_S)$
such that $[p_{\ast}E]=\beta$
and $\chi(E)=n$. Let
\begin{align*}
P_X \subset P_{n+1-g}(X, \beta)
\end{align*}
be the 
open subscheme of stable pairs $(E, s)$
such that the fundamental cycle of $p_{\ast}E$ is irreducible. 
We first observe the following
lemma, which is an analogue of Lemma~\ref{lem:Mobs}. 
\begin{lem}\label{lem:Pobs}
We have the following commutative diagram: 
\begin{align}\label{dia:PObs}
\xymatrix{
P_X \ar[rr]^{\cong} \ar[dr]_{p_{\ast}}
  &  & \mathrm{Obs}^{\ast}(\uU_P^{\bullet}) \ar[ld] \\
& P_S. &
}
\end{align}
Here the top morphism is a
canonical isomorphism, 
$p_{\ast}$ is induced by the projection 
$p \colon X \to S$ and 
the right morphism is the natural morphism 
defined from the cone structure.
\end{lem}
\begin{proof}
For a stable pair $(E, s)$ giving a closed point of $P_X$, 
we have the morphism 
$p_{\ast}s \colon \oO_S \to p_{\ast}E$
by the adjunction. 
Note that $p_{\ast}E$ is a pure one dimensional sheaf on $S$. 
Since the fundamental cycle of $p_{\ast}E$ is irreducible and 
$p_{\ast}s$ is non-zero,  
the pair $(p_{\ast}E, p_{\ast}s)$
is a stable pair on $S$. 
This argument can be generalized to families of stable pairs, 
so we obtain the morphism 
$p_{\ast} \colon P_X \to P_S$. 
For a closed point $(F, s)$ in ${\Sh}_S$, 
the fiber of $p_{\ast} \colon P_X \to P_S$ 
at $(F, s)$ consists of $\oO_X$-module structures of $F$, 
i.e. $\Hom(F, F\otimes K_S)$. 
Indeed if an $\oO_X$-module structure of $F$ is given, 
the morphism $s \colon \oO_S \to F$ 
is extended to an $\oO_X$-module homomorphism 
$\oO_X \to F$ by the adjunction. On the other hand, 
the fiber of the right morphism in (\ref{dia:PObs})
at $(F, s)$ consists of $\Ext_S^1(I, F)^{\vee}$, 
where $I=(\oO_S \stackrel{s}{\to} F)$
is the two term complex. 
We have the exact sequence
\begin{align*}
\Ext_S^1(F, F) \to H^1(F) \to \Ext_S^1(I, F) \to \Ext_S^2(F, F) \to 0.
\end{align*}
In~\cite[Proposition~C.2]{PT3},
the morphism $\Ext_S^1(F, F) \to H^1(F)$ is 
shown to be surjective. 
Hence we have the isomorphism 
$\Ext_S^1(I, F) \stackrel{\cong}{\to} \Ext_S^2(F, F)$, 
and the fiber of the right morphism 
in (\ref{dia:PObs}) is identified with 
$\Ext_S^2(F, F)^{\vee} \cong \Hom(F, F\otimes K_S)$. 
Then similarly to Lemma~\ref{lem:Mobs}, we obtain the 
desired isomorphism in the diagram (\ref{dia:PObs}).   
\end{proof}
Similarly to (\ref{HC:M}), we have the commutative 
diagram 
\begin{align}\notag
\xymatrix{
P_X^{\rm{red}} \ar[r]^{p_{\ast}} \ar[d]_{\pi_P} \ar[rd] 
& P_S \ar[d]^{\rho_P} \\
\wH_{\beta} \ar[r]_{p_{\ast}} & H_{\beta}
}
\end{align}
where the vertical morphisms are Hilbert-Chow maps. 
Let us take $\gamma \in \widehat{H}_{\beta}$ 
as in (\ref{take:gamma}). 
The local stable pair invariant is given by
\begin{align*}
P_{n+1-g, \gamma}^{\rm{loc}}=\int_{\pi_P^{-1}(\gamma)} \nu_P \ de
\end{align*}
where $\nu_P$ is the Behrend function on $P_X$. 

We describe the local stable pair invariant
in terms of data (\ref{etale:U}). 
By the diagram (\ref{dia:PS}), 
we have the isomorphism
\begin{align*}
(\pi^{[n]\ast} E^{\vee}|_{P_S} \stackrel{ds}{\to}
\Omega_{\cC_T^{[n]}}|_{P_S}) \stackrel{\cong}{\to} \mathbb{L}_{P_S}. 
\end{align*}
Similarly to (\ref{dia:MR2}),
we have the commutative diagram
\begin{align}\notag
\xymatrix{
\mathrm{Obs}^{\ast}(\mathbb{L}_{P_S}) \stackrel{\cong}{\to}
\{df^{[n]}=0\}
 \ar@<-0.3ex>@{^{(}->}[r]
& \cC_T^{[n]} \times_T E^{\vee}
 \ar[rd]^{f^{[n]}} \ar[d]_{\pi^{[n]}} & \\
 & E^{\vee} \ar[r]^{g} & \mathbb{A}^1. 
}
\end{align}
Let $\phi^{[n]}$ be the 
associated perverse sheaf of vanishing cycles
\begin{align*}
\phi^{[n]} \cneq \phi_{f^{[n]}}
(\IC(\cC_T^{[n]} \times_T E^{\vee}))
\in \Perv(\mathrm{Obs}^{\ast}(\mathbb{L}_{P_S})). 
\end{align*}

\begin{lem}\label{lem:relP}
By taking $d$ and 
$\overline{\gamma} \in E^{\vee}$ as in Lemma~\ref{lem:relphi},
we have the identity: 
\begin{align}\label{id:pairs}
P_{n+1-g, \gamma}^{\rm{loc}}=
(-1)^{d}
\chi(\dR \pi^{[n]}_{\ast} \phi^{[n]}|_{\overline{\gamma}}). 
\end{align}
\end{lem}
\begin{proof}
By taking the dual and the long exact sequence of 
cohomologies of (\ref{tri:Pair}), we obtain the exact 
sequence of sheaves
\begin{align*}
0 \to \hH^{-1}(\mathbb{L}_{P_S}^{\vee}) \to \hH^1(\uU_P^{\bullet \vee}) \to 
\rho_P^{\ast}R^1 \pi_{H\ast} \oO_{\cC_{\beta}}(\cC_{\beta}) \to 0. 
\end{align*}
By Lemma~\ref{lem:Pobs}, the 
 total space 
of the vector bundle (\ref{def:O})
acts on $P_X$ fiberwise over $H_{\beta}$ without fixed points. 
The quotient space is
\begin{align*}
P_X/O_H=\Obs^{\ast}(\mathbb{L}_{P_S}). 
\end{align*}
The canonical $d$-critical structure on 
the RHS as in Subsection~\ref{subsec:cone} is pulled-back to the $d$-critical 
structure on $P_X$ (also see Remark~\ref{rmk:compared}). 
Similarly to (\ref{Car:closed}), we have the commutative diagram
\begin{align}\label{Car:pair}
\xymatrix{
P_{X, U}^{\rm{red}} \ar[r]^(.4){\eta_P} \ar[d]_{\pi_P}
\ar@{}[dr]|\square & \Obs^{\ast}(\mathbb{L}_{P_{S, U}})^{\rm{red}} 
\ar@<-0.3ex>@{^{(}->}[r]
\ar[d]
 & \cC_{U}^{[n]} \times_U E^{\vee} \ar[d]^{\pi^{[n]}} 
\\
\wH_{U} \ar[r]^(.4){\eta}
 & \wH_{U}/O_{H, U} \ar[r] & E^{\vee}
} 
\end{align}
giving a $d$-critical chart of 
$\mathrm{Obs}^{\ast}(\mathbb{L}_{P_{S, U}})$. 
Here $-_{U}$ refers to the
pull-back by the open 
immersion (\ref{etale:U}), 
and the right horizontal morphisms 
are bijections onto their images. 
Let $\phi_{P}$ be the perverse sheaf on $P_{X, U}$ given by 
\begin{align*}
\phi_{P} \cneq \eta_P^{\ast} \phi^{[n]}[d]
\in \Perv(P_{X, U}). 
\end{align*}
By the definition of the Behrend function, 
we have 
\begin{align*}
\int_{\pi_P^{-1}(\gamma)} \nu_P \ de=
\chi(\dR \pi_{P\ast} \phi_{P} |_{\gamma}). 
\end{align*}
Therefore the identity (\ref{id:pairs}) follows from the 
commutative diagram (\ref{Car:pair}). 
\end{proof}

\subsection{Proof of Theorem~\ref{intro:thm1}}
By combining the arguments so far, 
Theorem~\ref{intro:thm1} now immediately follows:
\begin{thm}\label{thm:locsur}
Suppose that 
$C \subset S$ be an irreducible curve. 
Then for $\gamma \in \Chow_{\beta}(X)$
with $p_{\ast}\gamma=[C]$, 
we have the identity
\begin{align*}
\sum_{n \in \mathbb{Z}}
P_{n, \gamma}^{\rm{loc}} q^n=
\sum_{g\ge 0} n_{g, \gamma}^{\rm{loc}}
(q^{\frac{1}{2}}+q^{-\frac{1}{2}})^{2g-2}. 
\end{align*}
\end{thm}
\begin{proof}
By Lemma~\ref{lem:relphi}
and Lemma~\ref{lem:relP}, 
the result follows from the argument of Proposition~\ref{prop:irred2}. 
\end{proof}
By taking the integration over $\Chow_{\beta}(S)$, we 
also obtain the following: 
\begin{cor}
Conjecture~\ref{GV:conj1} holds for 
$X=\mathrm{Tot}(K_S)$ with irreducible 
curve class $\beta \in H_2(X, \mathbb{Z})=H_2(S, \mathbb{Z})$. 
\end{cor}

\subsection{Examples from rigid singular rational curves}\label{ex:enriques}
The result of Theorem~\ref{thm:locsur} includes many 
examples where the moduli space ${\Sh}_X$ is singular.
We describe what the above argument looks like in some examples arising from 
rigid singular rational curves on surfaces. 

Let $S$ be a smooth 
projective surface with $H^1(\oO_S)=0$ and 
\begin{align*}
C \subset S
\end{align*}
an irreducible curve whose normal bundle 
$\oO_C(C)$ is 
a non-trivial degree zero line bundle on $C$. 
We assume $C$ is either a nodal rational 
curve with one node, or a cuspidal rational 
curve. 
For example, we can construct such an example as follows: 
we first embed $C \subset \mathbb{P}^2$ 
as a cubic curve and blow-up $\mathbb{P}^2$ at 
9-points on $C$ in a general position. Then 
the resulting blow-up $S \to \mathbb{P}^2$ and 
the strict transform of $C$ to $S$ give such an example. 

Let $\beta \in H_2(S, \mathbb{Z})$ denote the homology class of $C$. 
Since $\oO_C(C)$ is non-trivial, 
we have 
\begin{align}\label{vanish:H}
H^0(\oO_C(C))=H^1(\oO_C(C))=0. 
\end{align}
In particular, the Chow variety
 $\Chow_{\beta}(S)$ on $S$ is one point.
Moreover we have the isomorphism
\begin{align*}
C \stackrel{\cong}{\to}{\Sh}_S, \ x \mapsto I_x^{\vee}. 
\end{align*}
Here $I_x \subset \oO_C$ is the ideal sheaf of $x$ in $C$ and $I_x^{\vee}$
is its (non-derived) dual on $C$, viewed as a sheaf on $S$.

Let $\nu \colon \mathbb{P}^1 \to C$ be the normalization. 
Since $K_S|_{C}=\oO_C(-C)$ is degree zero, 
we have 
\begin{align*}
H^0(\nu^{\ast}K_S)=H^0(\oO_{\mathbb{P}^1})=\mathbb{C},
\end{align*}
where this section does not descend to $C$.
By Lemma~\ref{lem:chow}, we see that 
$\Chow_{\beta}(X)=\mathbb{A}^1$, 
where $X=\mathrm{Tot}(K_S)$. 
The origin $0 \in \mathbb{A}^1$ corresponds to 
$C \subset S \subset X$, 
and $0\neq a \in \mathbb{A}^1$ corresponds to 
the embedding 
\begin{align*}
i_a \colon 
\mathbb{P}^1 \hookrightarrow X
\end{align*}
whose projection to $S$ is $C$. 
The moduli space ${\Sh}_X$ is
(set theoretically) identified with 
the union of ${\Sh}_S=C$ and 
$\mathbb{A}^1$. 
Here $a \in \mathbb{A}^1 \setminus \{0\}$ 
corresponds to the sheaf 
$i_{a\ast}\oO_{\mathbb{P}^1}$, 
and $\mathbb{A}^1$ intersects with $C$
at the singular point $p \in C$. 
We have the Hilbert-Chow map
\begin{align*}
\pi_{\Sh} \colon 
{\Sh}_X^{\rm{red}}=C \cup \mathbb{A}^1 \to \Chow_{\beta}(X)=\mathbb{A}^1
\end{align*}
which sends $C$ to the origin and 
restricts to $\id$ on $\mathbb{A}^1$. 
Below we describe ${\Sh}_X$ as critical locus, 
and compute the local GV invariants
\begin{align}\label{locgv:ex}
n_{g, -}^{\rm{loc}} \colon \Chow_{\beta}(X)=\mathbb{A}^1 \to \mathbb{Z}.
\end{align}

Let $\pi_T \colon \cC \to T$ be a versal deformation of 
$C$, 
with a point $0 \in T$ such that $\pi_T^{-1}(0)=C$. 
We first treat the case that $p \in C$ is a nodal singularity. 
 In this case, 
we can take $T$ to be a sufficiently small 
open neighborhood of $0 \in \mathbb{A}^1$, 
and 
\begin{align*}
\pi_T \colon 
\cC=\{zy^2=x^3+zx^2+tz^3\} \subset \mathbb{P}^2 \times T
\to T
\end{align*}
where $[x:y:z]$ is the homogeneous coordinate of
$\mathbb{P}^2$, $t$ is the coordinate of $T$
and the right arrow is the projection.  
The generic 
fiber of $\pi_T$ is a smooth elliptic curve, and 
the $\pi_T$-relative compactified Jacobian is 
isomorphic to 
$\cC$ itself. 
Since ${\Sh}_S$ is cut out by $t=0$ in $\cC$, 
by the diagram (\ref{dia:MR2})
and the vanishing (\ref{vanish:H}), we have the commutative diagram
\begin{align*}
\xymatrix{
{\Sh}_X \stackrel{\cong}{\to}
\{df=0\}
 \ar@<-0.3ex>@{^{(}->}[r]
& \cC \times \mathbb{A}^1
 \ar[rd]^{f} \ar[d]_{\pi_T \times \id} & \\
 & T \times \mathbb{A}^1 \ar[r]^{g} & \mathbb{A}^1. 
}
\end{align*}
Here $g$ is defined by
\begin{align*}
g(t, u)=tu.
\end{align*}
The moduli space ${\Sh}_X$ is singular at the node $p \in C$. 
We have the following affine open neighborhood 
of $p \in {\Sh}_X$:
\begin{align*}
\Spec 
\mathbb{C}[x, y, u]/(y^2-x^3-x^2, yu, (3x^2+2x)u). 
\end{align*}
This is the critical locus of the following function
\begin{align*}
f \colon \mathbb{A}^3 \to \mathbb{A}, \ 
(x, y, u) \mapsto u(y^2-x^3-x^2). 
\end{align*}
After taking the completion at
$0$ and coordinate change, 
the singularity at $p$ is simplified as 
\begin{align*}
\widehat{\oO}_{{\Sh}_X, p} \cong 
\mathbb{C}[[x, y, u]]/(xy, yu, ux). 
\end{align*}
This is the critical locus of the
super potential $xyu \in \mathbb{C}[[x, y, u]]$. 

As for the local GV invariants (\ref{locgv:ex}), 
the result is as follows: 
\begin{align}\label{gvcom:nodal}
n_{0, -}^{\rm{loc}} \equiv -1, \ 
n_{1, -}^{\rm{loc}} =\delta_0, \ 
n_{\ge 2, -}^{\rm{loc}} \equiv 0.
\end{align}
Here $\delta_0(t)=1$ for $t=0$ and 
$\delta_0(t)=0$ for $t\neq 0$. 
Indeed we have the perverse decomposition
\begin{align*}
\dR 
(\pi_T \times \id)_{\ast}
\IC(\cC \times \mathbb{A}^1)
=\IC(T \times \mathbb{A}^1)[1] \oplus V
 \oplus \IC(T \times \mathbb{A}^1)[-1]
\end{align*}
where $V=R^1(\pi_T \times \id)_{\ast}\mathbb{Q}[2]$ is a 
constructible sheaf on $T \times \mathbb{A}^1$. 
Applying $\phi_g$, we obtain
\begin{align*}
\dR \pi_{{\Sh}\ast}\phi_f
&=\phi_g(\dR 
(\pi_T \times \id)_{\ast}
\IC(\cC \times \mathbb{A}^1)) \\
&=\mathbb{Q}_0[1] \oplus \phi_g(V) \oplus 
\mathbb{Q}_0[-1]. 
\end{align*}
The above perverse decomposition immediately implies 
(\ref{gvcom:nodal}) for $n_{\ge 1, -}^{\rm{loc}}$. 
The computation of $n_{0, -}^{\rm{loc}}$
easily follows by the computation of the Behrend function. 
Note that we don't have to compute $\phi_g(V)$ in
the above computation. 

We next treat the case that $p\in C$ is a cusp. 
In this case, we can take 
$T$ to be a sufficiently small 
open neighborhood of $0 \in \mathbb{A}^2$, and 
\begin{align*}
\pi_T \colon \cC=\{zy^2=x^3+t_1 xz+t_2 z^3\} \subset
\mathbb{P}^2 \times T \to T.
\end{align*} 
Here $(t_1, t_2)$ is the coordinate of $T$. 
Similarly to the nodal case, we have the commutative diagram
\begin{align*}
\xymatrix{
{\Sh}_X \stackrel{\cong}{\to}
\{df=0\}
 \ar@<-0.3ex>@{^{(}->}[r]
& \cC \times \mathbb{A}^2
 \ar[rd]^{f} \ar[d]_{\pi_T \times \id} & \\
 & T \times \mathbb{A}^2 \ar[r]^{g} & \mathbb{A}^1. 
}
\end{align*}
Here $g$ is defined by
\begin{align*}
g(t_1, t_2, u_1, u_2)=t_1 u_1+t_2 u_2.
\end{align*}
We have the following affine
open neighborhood at $p\in {\Sh}_X$:
\begin{align*}
\Spec \mathbb{C}[x, y, u]/(x^2 u, yu, y^2-x^3). 
\end{align*}
In this case, 
$\mathbb{A}^1$ is a double line in ${\Sh}_X$. 
The above affine open neighborhood
 is the critical locus of the following function:
\begin{align*}
f \colon \mathbb{A}^5 \to \mathbb{A}, \ 
(x, y, t_1, u_1, u_2) \mapsto u_1 t_1+u_2(y^2-x^3-t_1x). 
\end{align*}
This is simplified as the critical locus of 
$u(y^2-x^3) \in \mathbb{C}[x, y, z]$. 
Similarly to the nodal case, the 
local GV invariants (\ref{locgv:ex}), 
are computed as follows: 
\begin{align}\notag
n_{0, -}^{\rm{loc}} \equiv -2, \ 
n_{1, -}^{\rm{loc}} =\delta_0, \ 
n_{\ge 2, -}^{\rm{loc}} \equiv 0.
\end{align}

\section{Smooth curve case}\label{sec:smooth}

Let $X$ be a smooth quasi-projective CY 3-fold 
and $C \subset X$ a smooth projective curve
with homology class $\beta$
and genus $g$. 
Note that we no longer are assuming that $X$ is a local surface.

In this section, we 
apply Proposition~\ref{prop:irred2} to 
prove Conjecture~\ref{GV:conj2}
at the point 
\begin{align*}
\gamma=[C] \in \Chow_{\beta}(X).
\end{align*}

\subsection{Moduli space at a smooth one-cycle}
We set
\begin{align*}
{\Sh}=\Sh_{\beta}(X), \ 
\eE \in \Coh(X \times {\Sh})
\end{align*}
where $\eE$ is a universal sheaf. 
We will use the following lemma: 
\begin{lem}\label{prop:isomJM}
Let ${\Sh}' \subset {\Sh}$ be the 
subset consisting of 
sheaves of the form $j_{\ast}L$
for 
a smooth curve 
$j \colon Z \hookrightarrow X$
and $L \in \Pic(Z)$. 
Let $\zZ \subset X \times {\Sh}$ be the 
closed subscheme defined by the following ideal 
sheaf $\iI_{\zZ}$ 
\begin{align}\label{nat:mor}
\iI_{\zZ} \cneq \Ker \left(
\oO_{X \times {\Sh}} \to \eE nd_{\oO_{X \times {\Sh}}}(\eE) \right). 
\end{align}
Then $\zZ$ is flat over ${\Sh}$ at any point in ${\Sh}'$. 
\end{lem}
\begin{proof}
Let us take a point $y \in {\Sh}'$
corresponding to $j_{\ast}L$
for a smooth curve $j \colon Z \hookrightarrow X$
and $L \in \Pic(Z)$. 
We also take a point $x \in X$ and set
\begin{align*}
z=(x, y) \in X \times {\Sh}'.
\end{align*}
We need to show that $\oO_{\zZ, z}$ is a flat 
$\oO_{{\Sh}, y}$-module. 
Since $\eE$ is a universal sheaf, 
we have an isomorphism 
$\eE|_{X \times \{y\}} \cong j_{\ast}L$. 
Since $(j_{\ast}L)_{x}$ is generated by 
one element as an $\oO_{X, x}$-module, 
for a sufficiently ample line 
bundle $\lL$ on $X$ there is a morphism
\begin{align*}
s \colon \oO_{X \times {\Sh}} \to \eE \otimes p_X^{\ast} \lL
\end{align*}
which is surjective at $z$. 
Here $p_X \colon X \times {\Sh} \to X$ is the projection. 
Therefore 
the morphism $s$ induces
an isomorphism $\eE nd_{\oO_{X \times {\Sh}, z}}(\eE_{z})\cong \eE_{z}$
as $\oO_{X \times {\Sh}, z}$-modules. 
Since we have the factorization
\begin{align*}
s_z \colon
\oO_{X \times {\Sh}, z} \twoheadrightarrow \oO_{\zZ, z}
\hookrightarrow \eE nd_{\oO_{X \times {\Sh}, z}}(\eE_{z}) \cong 
\eE_{z}
\end{align*}
and the above composition is surjective, 
we see that $\oO_{\zZ, z} \cong \eE_z$
as $\oO_{X \times {\Sh}, z}$-modules. 
Since $\eE$ is flat over ${\Sh}$, it follows that 
$\oO_{\zZ, z}$ is a flat $\oO_{{\Sh}, y}$-module.  
\end{proof}
Let $\Hilb(X)$ be the Hilbert scheme of 
compactly supported
closed subschemes in $X$. 
We  
take the open subscheme 
 \begin{align*}
[C] \in H \subset \Hilb(X)
\end{align*}
consisting
of smooth subschemes 
$Z \subset X$ with $\dim Z=1$, 
homology class $\beta$ and arithmetic genus $g$. 
Let
\begin{align}\label{dia:univC}
\xymatrix{
\cC_H 
\ar@<-0.3ex>@{^{(}->}[r] \ar[d]_-{\pi_{H}} & X \times H \ar[ld]^-{p_H} \\
H  &
}
\end{align}
be the universal curve. 
We denote by $J_H \to H$ the $\pi_{H}$-relative 
moduli space of line bundles with Euler characteristic one, 
which is a smooth abelian fibration with 
relative dimension $g$. We 
take the universal line bundle
\begin{align*}
\lL \in \Pic(\cC_H \times_H J_H).
\end{align*}
Let $i \colon \cC_H \times_H J_H \hookrightarrow X \times J_H$
be the closed
embedding induced by the diagram (\ref{dia:univC}).  
The object
\begin{align}\label{fam/J}
i_{\ast}\lL \in \Coh(X \times J_H)
\end{align}
is a $J_H$-flat family of one dimensional 
stable sheaves on $X$. 
The object (\ref{fam/J})
determines the morphism of schemes
\begin{align}\notag
h \colon J_H \to {\Sh}.
\end{align}

\begin{lem}\label{lem:isomJM2}
The subset ${\Sh}' \subset {\Sh}$ in Lemma~\ref{prop:isomJM} is open, and 
the morphism $h$ 
gives an isomorphism 
of schemes 
$h \colon J_H \stackrel{\cong}{\to} {\Sh}'$. 
\end{lem}
\begin{proof}
By Lemma~\ref{prop:isomJM}, 
the subset ${\Sh}' \subset {\Sh}$
coincides with 
the set of points $y \in {\Sh}$ 
such that $\zZ$ is flat over ${\Sh}$ at $y$
and $\zZ_y \cneq \zZ|_{X \times y}$
is smooth. 
Since the latter conditions are open conditions, 
the subset ${\Sh}' \subset {\Sh}$ is open. 

By Lemma~\ref{prop:isomJM}, 
the subscheme 
\begin{align}\notag
\zZ' \cneq \zZ|_{X \times {\Sh}'} \subset X \times {\Sh}'
\end{align}
is flat over ${\Sh}'$. 
Hence it 
defines the morphism 
\begin{align}\label{M'H}
\pi_{\Sh} \colon {\Sh}' \to H
\end{align}
such that $\zZ'=\cC_H \times_{H} {\Sh}'$. 
Then by the definition of $\zZ$ in (\ref{nat:mor}), 
we have 
\begin{align*}
\eE|_{X \times {\Sh}'}  \in \Coh(\zZ')=\Coh(\cC_H \times_H {\Sh}'). 
\end{align*}
The above object is a ${\Sh}'$-flat family of line bundles on 
the fibers of $\cC_H \to H$, 
hence 
determines the morphism 
$h' \colon {\Sh}' \to J_H$. 
The morphism $h'$ gives an inverse of $h$, 
hence $h$ is an isomorphism. 
\end{proof}

\subsection{CY condition for ${\Sh}$}
Here we prove the CY condition for ${\Sh}$: 
\begin{prop}\label{lem:M:CY}
The $d$-critical scheme 
\begin{align*}
({\Sh}=\Sh_{\beta}(X), s_{\Sh})
\end{align*}
in Theorem~\ref{thm:CYM}
is strictly CY at 
the point $\gamma=[C] \in \Chow_{\beta}(X)$ 
for a smooth projective curve $C \subset X$
(see Definition~\ref{defi:slcy}). 
\end{prop}
\begin{proof}
For the universal curve $\pi_H \colon \cC_H \to H$
in (\ref{dia:univC}),  
by Lemma~\ref{lem:versal}
we can take 
a locally versal family 
and a closed embedding
\begin{align*}
\pi_T \colon \cC \to T, \ 
j \colon H \hookrightarrow T
\end{align*}
such that 
$\cC_H=\cC \times_T H$.
Let $\pi_{J} \colon J \to T$ be the 
$\pi_T$-relative 
moduli space of line bundles with Euler characteristic one. 
By Lemma~\ref{lem:isomJM2}, 
we have the Cartesian squares: 
\begin{align}\label{Cart:M}
\xymatrix{
{\Sh}^{\rm{red}} \ar[d] \ar@{}[dr]|\square \ar@<-0.3ex>@{<-^{)}}[r]& 
{\Sh}^{' \rm{red}} \ar[r]\ar@{}[dr]|\square
 \ar[d] & {\Sh}' \ar@<-0.3ex>@{^{(}->}[r]^{i} \ar[d]_{\pi_{\Sh}}
\ar@{}[dr]|\square &
J \ar[d]^{\pi_{J}} \\
\Chow_{\beta}(X) \ar@<-0.3ex>@{<-^{)}}[r] &
H^{\rm{red}} \ar[r] & H  \ar@<-0.3ex>@{^{(}->}[r]^{j} &
T. 
}
\end{align}
Here $\pi_{\Sh}$ is given in (\ref{M'H}), 
 the left bottom arrow is
an injection induced by the cycle 
$\cC_H \times_H H^{\rm{red}}$, 
and the left top arrow is an open immersion. 
By the diagram (\ref{Cart:M}), 
it is enough to show the following: 
there is an open neighborhood $\gamma =[C] \in U \subset T$
and a regular function 
$g \colon U \to \mathbb{A}^1$
such that 
by setting the commutative diagram 
\begin{align*}
\xymatrix{
{\Sh}_U' \ar@{}[dr]|\square \ar@<-0.3ex>@{^{(}->}[r]^{i} \ar[d]_{\pi_{\Sh}}  & 
 J_U \ar[d]^{\pi_J} \ar[rd]^{f_J}  &   \\
H_U \ar@<-0.3ex>@{^{(}->}[r]^{j} &  U \ar[r]^{g}  & \mathbb{A}^1
}
\end{align*}
where $-_U$ refers to the pull-back by 
$U \subset T$, 
the data
\begin{align}\label{data:critical}
({\Sh}_U', J_U, f_{J}, i) 
\end{align}
is a $d$-critical chart of $({\Sh}_U', s_{\Sh}|_{{\Sh}_U'})$. 

Let $I \subset \oO_{T}$ be the ideal sheaf which 
defines the subscheme $H \subset T$. 
Since $\pi_{J}$ is smooth, the ideal sheaf
$\pi^{\ast}_J I \subset \oO_{J}$ defines 
the subscheme 
${\Sh}' \subset J$. 
By the property of the sheaves 
$\sS_{{\Sh}'}$ and $\sS_H$ (see (\ref{S:property})), 
we have the commutative diagram
\begin{align}\label{exact:S}
\xymatrix{
0 \ar[r] & H^0(\sS_H) \ar[r] \ar[d]^{\pi_{\Sh}^{\ast}} & H^0(\oO_T/I^2) \ar[r] 
\ar[d]^{\pi_J^{\ast}} & 
H^0(\Omega_{T}/I \cdot \Omega_T) \ar[d]^{\pi_J^{\ast}} \\
0 \ar[r] & H^0(\sS_{{\Sh}'}) \ar[r] & 
H^0(\oO_{J}/\pi_J^{\ast}I^2) \ar[r] & 
H^0(\Omega_{J}/\pi_J^{\ast} I \cdot \Omega_J). 
}
\end{align}
Here the horizontal arrows are exact sequences of 
vector spaces. 
By the derived base change, we have
\begin{align*}
\dR \pi_{J\ast} \oO_J \dotimes \oO_T/I^2 \cong
\dR \pi_{J\ast}(\oO_J/\pi_J^{\ast}I^2). 
\end{align*}
Since each $R^i \pi_{J\ast} \oO_{J}$ is locally free, 
it follows that $\pi_{J\ast} (\oO_{J}/\pi_{J}^{\ast}I^2) \cong \oO_T/I^2$, 
and the middle vertical arrow of (\ref{exact:S}) is an isomorphism. 
Since 
we have the exact sequence of locally free sheaves
\begin{align*}
0 \to \pi_J^{\ast} \Omega_T \to \Omega_{J} \to \Omega_{J/T} \to 0
\end{align*}
we have the injection 
\begin{align*}
\pi_J^{\ast}\Omega_T/\pi_J^{\ast} I \cdot \pi_J^{\ast} \Omega_T
 \hookrightarrow
\Omega_J/\pi_J^{\ast}I \cdot \Omega_J.
\end{align*} 
Therefore the right vertical arrow of (\ref{exact:S}) is an injection, 
hence the left vertical arrow of (\ref{exact:S}) is an isomorphism. 
It follows that 
there exists $s_H \in H^0(\sS_H^0)$ such that 
the identity
$s_{{\Sh}}|_{{\Sh}'}=\pi_{\Sh}^{\ast} s_H$ holds.  

Since $\pi_{\Sh}$ is a smooth surjective morphism, 
by~\cite[Proposition~2.8]{JoyceD},
the section $s_H$ is a $d$-critical structure of $H$. 
By~\cite[Proposition~2.7]{JoyceD}, 
we can find an open neighborhood 
$\gamma \in U \subset T$
and 
a regular function $g \colon U \to \mathbb{A}^1$
such that 
$
(H_U, U, g, j)
$
is a $d$-critical chart of $(H_U, s_H|_{H_{U}})$. 
Since 
$\pi_J$ is a smooth morphism
and $s_{\Sh}|_{{\Sh}'}=\pi_{\Sh}^{\ast}s_H$,  
the data (\ref{data:critical}) 
obviously gives a desired $d$-critical chart. 
\end{proof}

\subsection{Proof of Theorem~\ref{intro:thm2}}
The following is the main result of this section, 
which proves Theorem~\ref{intro:thm2}: 
\begin{thm}\label{thm:smooth}
For a smooth projective curve $C \subset X$ with genus $g$, 
Conjecture~\ref{GV:conj2} is true for $\gamma=[C] \in \Chow_{\beta}(X)$. 
In this case, we have
\begin{align}\label{smooth:GV}
n_{h, \gamma}^{\rm{loc}}=\left\{ \begin{array}{cc}
(-1)^g \nu_{{\Sh}}(\oO_C), & h=g, \\
0, & h \neq g.
\end{array} \right. 
\end{align}
Here $\nu_{\Sh}$ is the Behrend function on ${\Sh}=\Sh_{\beta}(X)$. 
\end{thm}
\begin{proof}
By the proof of Proposition~\ref{lem:M:CY}, 
the moduli space ${\Sh}$
satisfies
the assumption in Proposition~\ref{prop:irred2}
at $\gamma$. 
As for the moduli space of stable pairs 
let
\begin{align*}
P' \subset P=P_{n+1-g}(X, \beta)
\end{align*}
 be the 
subset consisting of stable pairs $(F, s)$
such that 
$F=j_{\ast}L$ for a smooth subscheme 
$j \colon Z \hookrightarrow X$ and $L \in \Pic(Z)$. 
Then the same argument of Lemma~\ref{prop:isomJM}
shows that $P'$ is an open subset of $P$, 
and isomorphic to the relative Hilbert scheme of $n$-points 
of $\pi_H \colon \cC_H \to H$. 
Then similarly to Lemma~\ref{lem:M:CY}, 
one can show 
show that $P$ also satisfies the assumption in Proposition~\ref{prop:irred2}
at $\gamma$. 
Therefore applying Proposition~\ref{prop:irred2}, 
Conjecture~\ref{GV:conj2}
is true in this case. 
The formula (\ref{smooth:GV})
for $n_{h, [C]}^{\rm{loc}}$ 
follows from the GV formula for stable pairs
given in~\cite[Proposition~3.6]{PT3}. 
\end{proof}

\section{Comparison with the former definitions}\label{sec:compare}
In this section, we review the previous 
definitions of GV invariants~\cite{HST, KL} using 
$sl_2 \times sl_2$-actions,
and compare them with our definition. 
We also give an example that the previous 
definitions do not match with the predicted answer.

\subsection{$sl_2 \times sl_2$-action}\label{subsec:KL}
Let $X$ be a smooth projective CY 3-fold. 
As in Subsection~\ref{subsec:defgv}, 
we take 
the moduli space of one-dimensional stable sheaves on $X$
with its canonical $d$-critical structure
\begin{align}\label{dcrit:sl}
(\Sh_{\beta}(X), s_{\Sh})
\end{align}
and the Hilbert-Chow map
\begin{align}\label{HC:sl} 
\pi \colon 
{\Sh}_{\beta}^{\rm{red}}(X) \to \Chow_{\beta}(X). 
\end{align}

Let 
$\mathrm{MHM}(\Sh_{\beta}(X))$
be the abelian category 
of polarized mixed Hodge modules on $\Sh_{\beta}(X)$, 
whose basics we refer to~\cite{MSaito, Schnell}. 
We take a perverse sheaf 
 $\phi$ on $\Sh_{\beta}(X)$ 
which underlies a polarized mixed Hodge module, 
i.e. there is an object $\phi^H \in \mathrm{MHM}(\Sh_{\beta}(X))$
 such that $\mathrm{rat}(\phi^H)=\phi$
for the forgetful functor
\begin{align}\notag
\mathrm{rat} \colon 
\mathrm{MHM}(\Sh_{\beta}(X)) \to \Perv(\Sh_{\beta}(X)). 
\end{align} 
We say that $\phi$ is \textit{pure} if $\phi^H$ is a pure 
Hodge module. 
If $\phi$ is pure, then 
we have the BBD decomposition theorem~\cite{BBD}
\begin{align*}
\dR \pi_{\ast}\phi \cong 
\bigoplus_{i\in \mathbb{Z}}
\pH^i(\dR \pi_{\ast} \phi)[-i]. 
\end{align*}
Then the hypercohomology of $\phi$ decomposes into
\begin{align*}
H^{\ast}(\Sh_{\beta}(X), \phi)=\bigoplus_{i, j}H^{i, j}, \ 
H^{i, j} \cneq 
H^j(\Chow_{\beta}(X), \pH^i(\dR \pi_{\ast}\phi)). 
\end{align*}
Let $\omega_L$ be a $\pi$-ample divisor on $\Sh_{\beta}^{\rm{red}}(X)$ and 
$\omega_R$ an ample divisor on $\Chow_{\beta}(X)$. 
Since each 
$\pH^{i}(\dR \pi_{\ast} \phi)$ is also 
pure (see~\cite[Theorem~16.1]{Schnell}), we have the 
Hard-Lefschetz isomorphisms
\begin{align*}
\omega_L^{i} \colon H^{-i, j} \stackrel{\cong}{\to}
H^{i, j}, \ 
\omega_R^{j} \colon H^{i, -j} \stackrel{\cong}{\to}
H^{i, j}. 
\end{align*}
The above isomorphisms 
define the $sl_2 \times sl_2$-action 
on $H^{\ast}(\Sh_{\beta}(X), \phi)$. 
The multiplication by $\omega_L$ defines the left $sl_2$-action, 
and the multiplication by $\omega_R$ defines the right 
$sl_2$-action. 

Let $I_g$ be the $sl_2$-representation given by 
\begin{align*}
I_g=IH^{\ast}(A, \mathbb{Q})
\end{align*}
where $A$ is a $g$-dimensional 
abelian variety with its $sl_2$-action
given by the Hard-Lefschetz theorem. 
For $2j \in \mathbb{Z}$, 
let $(j)$ be the unique 
irreducible $sl_2$-representation with dimension $2j+1$. 
The $sl_2$-representation $I_g$ is written as 
\begin{align*}
I_g=\left( \left(\frac{1}{2}\right) \oplus 2(0)  \right)^{\otimes g}. 
\end{align*}
By the Clebsch-Gordan rule, one can write
\begin{align}\label{HIR}
H^{\ast}(\Sh_{\beta}(X), \phi) 
=\bigoplus_{g\ge 0} (I_g)_L \otimes (R_{g})_R
\end{align}
for some virtual right $sl_2$-representation $R_g$.
Here $-_L$, $-_R$ refer to left, right $sl_2$-representations
respectively. 
We can write $R_g$ as 
\begin{align*}
R_{g}=\sum_{2j \in \mathbb{Z}} R_{g, j} \otimes (j), \ 
R_{g, j} \in K(\mathrm{Vect}(\mathbb{Q})). 
\end{align*}
Following the previous works~\cite{GV, HST}, 
we define
\begin{align}\label{phinv}
n_{g, \beta}(\phi) \cneq 
\sum_{2j\in \mathbb{Z}}
(-1)^{2j}(2j+1) \cdot \dim R_{g, j}. 
\end{align}
The invariants (\ref{phinv}) are 
characterized by the character formula, 
as in Lemma~\ref{lem:ng}: 
\begin{lem}\label{lem:KLid}
We have the identity
\begin{align}\label{KL:id}
\sum_{i\in \mathbb{Z}} \chi(\pH^i(\dR \pi_{\ast}\phi))
y^i
=\sum_{g\ge 0} n_{g, \beta}(\phi)
(y^{\frac{1}{2}}+y^{-\frac{1}{2}})^{2g}. 
\end{align}
\begin{proof}
By taking the characters of the 
identity (\ref{HIR}), 
we have the following identity in 
$K(\mathrm{Vect}(\mathbb{Q}))[x^{\pm 1}, y^{\pm 1}]$:
\begin{align*}
\sum_{i, j \in \mathbb{Z}} H^{i, j}
x^j y^i =
\sum_{g\ge 0, 2j \in \mathbb{Z}}
R_{g, j} (x^{-2j}+x^{-2j+2}+ \cdots + x^{2j})
(y^{\frac{1}{2}}+y^{-\frac{1}{2}})^{2g}. 
\end{align*} 
By substituting $x=-1$, we obtain the identity (\ref{KL:id}). 
\end{proof}
\end{lem}

\subsection{HST and KL definitions}
Both of HST~\cite{HST} and KL~\cite{KL} definitions
are 
given by (\ref{phinv}) for some pure perverse sheaf $\phi$. 
Let us take the normalization and
the 
intersection complex 
\begin{align*}
\nu \colon \widetilde{\Sh}_{\beta}(X) \to \Sh_{\beta}(X), \
\phi=\nu_{\ast}\IC(\widetilde{\Sh}_{\beta}(X)).
\end{align*}
The HST definition is given by 
\begin{align*}
n_{g, \beta}^{\rm{HST}} \cneq 
n_{g, \beta}(\phi=\nu_{\ast}\IC(\widetilde{\Sh}_{\beta}(X))). 
\end{align*}

The KL definition uses sheaves of vanishing cycles.  
Let us choose an orientation 
$(K_{\Sh}^{\rm{vir}})^{1/2}$
for the $d$-critical scheme (\ref{dcrit:sl}), 
and set 
\begin{align*}
\sS h=(\Sh_{\beta}(X), s_{\Sh}, (K_{\Sh}^{\rm{vir}})^{1/2}).
\end{align*}
Let $\phi_{\sS h}$ be the gluing of 
local sheaves of vanishing cycles as in Theorem~\ref{thm:BDJS}.
By~\cite[Theorem~6.9]{BDJS}, 
the perverse sheaf $\phi_{\sS h}$ underlies a polarized mixed Hodge 
module, but it is not pure in general. 
Let
\begin{align*}
\gr_W^{\bullet}(\phi_{\sS h})  \in \Perv(\Sh_{\beta}(X))
\end{align*} 
be the associated graded sheaf with respect to the weight 
filtration in $\mathrm{MHM}(\Sh_{\beta}(X))$, which is now a pure perverse sheaf. 
The KL definition is given by 
\begin{align}\label{KLinv}
n_{g, \beta}^{\rm{KL}} \cneq 
n_{g, \beta}(\phi=\gr_W^{\bullet}(\phi_{\sS h})).
\end{align} 

\begin{rmk}\label{rmk:HSTKL}
By Lemma~\ref{lem:KLid}, 
the HST and KL definitions 
are also given by 
substituting 
\begin{align*}
\phi=\nu_{\ast}\IC(\widetilde{\Sh}_{\beta}(X)), \ 
\phi=\gr_W^{\bullet}(\phi_{\sS h})
\end{align*}
to the formula (\ref{KL:id}) respectively. 
On the other hand, the character formula (\ref{KL:id}) 
makes sense even if $\phi$ is not pure. 
As in Lemma~\ref{lem:ng}, if $\phi$ is self-dual 
$\mathbb{D}(\phi)=\phi$, the invariant 
$n_{g, \beta}(\phi)$ is uniquely determined by the identity (\ref{KL:id}). 
This is our point of view defining GV invariants in Definition~\ref{intro:defi}. \end{rmk}

\begin{rmk}
In~\cite{KL}, Kiem-Li used the semi-normalization of 
(\ref{HC:sl}) as the definition of HC map, following the convention of 
the Chow variety in~\cite{Ko}. Since taking the semi-normalization is a 
homeomorphism, 
this step does not affect the definition of the GV invariants. 
\end{rmk}

\subsection{Dependence on orientation data of KL definition}\label{ex:depend}
In Kiem-Li's paper~\cite{KL}, they did not specify 
how to choose an orientation data. 
Indeed as the following example shows, the 
KL invariant (\ref{KLinv}) depends on a choice of an 
orientation data. 

Let $E$ be an elliptic curve, 
which is embedded into a CY 3-fold
\begin{align*}
i \colon E \hookrightarrow X
\end{align*}
whose normal bundle is written as 
\begin{align*}
N_{E/X}=L \oplus L^{-1}, \ L \in \Pic^0(E) \setminus \{\oO_E\}.
\end{align*}
Then $E$ is rigid inside $X$. 
Let us take the homology class
\begin{align*}
\beta=[E] \in H_2(X, \mathbb{Z}).
\end{align*} 
Suppose that $E$ is the unique curve in $X$ with homology class
$\beta$, i.e. the Chow variety is a point
\begin{align*}
\Chow_{\beta}(X)=\{[E]\}. 
\end{align*} 
Then we have the isomorphism (see~\cite[Proposition~4.4]{HST})
\begin{align*}
E \stackrel{\cong}{\to} \Sh_{\beta}(X), \ 
x \mapsto i_{\ast}\oO_E(x). 
\end{align*}
Since $E$ is smooth, we have 
$K_{{\Sh}}^{\rm{vir}}=\oO_E$. 
Then an orientation $(K_{{\Sh}}^{\rm{vir}})^{1/2}$
of $(\Sh_{\beta}(X), s_{\Sh})$
is 
a 2-torsion element 
of $\Pic^0(E)$. 
For the oriented $d$-critical scheme 
$\sS h=(\Sh_{\beta}(X), s_{\Sh}, (K_{{\Sh}}^{\rm{vir}})^{1/2})$, 
we have
\begin{align*}
\phi_{\sS h}=\lL[1]
\end{align*}
for 
a rank one local system $\lL$ on $E$
such that $\lL^{\otimes 2} \cong \mathbb{Q}_E$. 
The local system $\lL$ is trivial if and only if 
$(K_{{\Sh}}^{\rm{vir}})^{1/2}=\oO_E$. 
Since $H^{\ast}(E, \lL)=0$ 
if $\lL$ is non-trivial, we obtain
$n_{g, \beta}^{\rm{KL}}=0$ for $g\neq 1$ and 
\begin{align*}
n_{1, \beta}^{\rm{KL}}=\left\{
\begin{array}{cc}
1, & (K_{{\Sh}}^{\rm{vir}})^{1/2}=\oO_E, \\
0, & (K_{{\Sh}}^{\rm{vir}})^{1/2}\neq \oO_E. 
\end{array}
\right. 
\end{align*}
The expected answer is $n_{1, \beta}=1$, 
so should choose an 
orientation data to be CY, i.e. 
$(K_{{\Sh}}^{\rm{vir}})^{1/2}=\oO_{E}$. 

\begin{rmk}
Suppose that ${\Sh}_{\beta}(X)$ is non-singular 
and
the usual canonical line bundle $K_{\Sh}$ 
on ${\Sh}_{\beta}(X)$ is pulled back 
from the 
 Hilbert-Chow map (\ref{HC:sl}). 
Then as $K_{{\Sh}}^{\rm{vir}}=K_{\Sh}^{\otimes 2}$, 
we can take the oriented 
$d$-critical scheme 
\begin{align*}
\sS h=({\Sh}_{\beta}(X), s_{\Sh}=0, K_{\Sh})
\end{align*}
which is a CY fibration over $\Chow_{\beta}(X)$. 
In this case, we have 
$
\phi_{\sS h}=\mathrm{IC}({\Sh}_{\beta}(X))$
and 
all the definitions agree:  
\begin{align*}
n_{g, \beta}^{\rm{HST}}=
n_{g, \beta}^{\rm{KL}}=n_{g, \beta}. 
\end{align*}
\end{rmk}
Even if we choose a CY orientation data, 
our definition may not agree with KL definition. 
In general, there is a 
\textit{weight spectral sequence}
\begin{align}\label{weight:spectral}
E_1^{i, j}=\pH^{i+j} (\dR\pi_{\ast}
\mathrm{gr}_{W}^{-i}(\phi_{\sS h}))\Rightarrow
\pH^{i+j}(\dR\pi_{\ast}\phi_{\sS h})
\end{align}
which always degenerates at $E_2$ by considering weights
(see~\cite[Section~17]{Schnell} for the similar spectral sequence for 
nearby cycles). 
It is easy to see that, for a choice of CY orientation data, 
our definition agrees with KL definition 
if the spectral sequence (\ref{weight:spectral}) degenerates at $E_1$. 
In the next subsection, we will see an example where (\ref{weight:spectral})
does not degenerate, which also gives a counter-example to the conjectures
of Kiem-Li.

\subsection{Counter-example to Kiem-Li conjecture}\label{subsec:counter}
In this subsection, we prove Proposition~\ref{prop:counter}. 
Let $S$ be an Enriques surface. It always admits 
an elliptic fibration 
\begin{align}\label{h:fib}
h \colon S \to \mathbb{P}^1. 
\end{align}
Let $\sigma \colon \widetilde{S} \to S$ be 
a K3 cover and $E$ an elliptic curve. 
We set
\begin{align*}
X=(\widetilde{S} \times E)/\langle \tau \rangle. 
\end{align*}
Here $\tau$ 
is an involution on $\widetilde{S} \times E$ which acts 
on $\widetilde{S}$
as a covering involution of $\sigma$, 
and 
acts on $E$ by $x \mapsto -x$. 
The 3-fold $X$ is a smooth projective CY 3-fold, 
and its GW invariants were studied in~\cite{MPnew}. 

We first describe the geometry of $X$ via 
fibrations over $\mathbb{P}^1$ and $S$. 
The projections from 
$\widetilde{S} \times E$
onto each factors 
 induce the fibrations
\begin{align}\label{fibration}
p \colon X \to E/\langle \tau \rangle =\mathbb{P}^1, \ 
\widehat{p} \colon X \to S. 
\end{align}
We denote by
\begin{align*}
e_1, e_2, e_3, e_4 \in E, \ 
q_1, q_2, q_3, q_4 \in \mathbb{P}^1
\end{align*}
the 2-torsion points of $E$
and their images under the 
quotient map $E \to \mathbb{P}^1$
respectively. 
The fiber of $p$ at $x \in \mathbb{P}^1$ is $\widetilde{S}$
for $x \neq q_i$, and 
the fiber at $q_i$ is a double fiber $2S$. 
The map $\widehat{p}$ is a smooth fibration with fiber $E$. 
\begin{rmk}\label{rmk:TotKS}
The CY 3-fold $X$ is closely related to the
non-compact CY 3-fold 
\begin{align*}
X'=\mathrm{Tot}(K_S).
\end{align*}
Indeed 
$X'$ is given by the quotient of 
$\widetilde{S} \times \mathbb{A}^1$ by 
$\tau$, 
where $\tau$ acts on $\mathbb{A}^1$ by $x \mapsto -x$. 
Similarly to (\ref{fibration}), we have the fibration 
\begin{align*}
p' \colon 
X' \to \mathbb{A}^1/\tau=\mathbb{A}^1. 
\end{align*}
Then there
exist analytic open neighborhoods
$q_i \in U_i \subset \mathbb{P}^1$, 
$0\in U \subset \mathbb{A}^1$
such that $p^{-1}(U_i)$ is isomorphic to 
$p^{'-1}(U)$. 
\end{rmk}
Let $2C$ be a double fiber the elliptic fibration (\ref{h:fib}), 
and set 
$X_C \cneq \widehat{p}^{-1}(C)$, 
$\widetilde{C} \cneq \sigma^{-1}(C) \subset
\widetilde{S}$. 
Note that we have
\begin{align*}
X_C=(\widetilde{C} \times E)/\langle \tau \rangle.
\end{align*}
Here if $\widetilde{C}$ is smooth, $\tau$ acts on 
$\widetilde{C} \times E$ by 
$(x, y) \mapsto (x+t, -y)$
where $t$ is a two torsion point in 
$\widetilde{C}$. 
We will use the following diagram
\begin{align}\label{dia:F}
\xymatrix{
X_C \ar[r]^{p|_{X_C}} 
\ar[d]_{\widehat{p}|_{X_C}} & \mathbb{P}^1 \\
C.   &  
}
\end{align}
Let $C_i$
be the reduced fiber of $p|_{X_C}$ at $q_i$ 
for $1\le i\le 4$, giving 
four sections of $\widehat{p}|_{X_C}$. 
Note that
we have
\begin{align}\label{Ci=C}
C_i=C \subset 2S=p^{-1}(q_i).
\end{align}
Also we have  
\begin{align*}
p|_{X_C}^{-1}(x)=\widetilde{C} \subset p^{-1}(x)=
\widetilde{S}, \ x \neq q_i. 
\end{align*}
Using the K\"unneth formula, 
we see that
\begin{align}\label{H2}
H_2(X, \mathbb{Z})=H_2(S, \mathbb{Z}) \oplus \mathbb{Z}[E]
\end{align}
where $[E]$ is the fiber class of the 
projection 
$\widehat{p} \colon X \to S$. 
Let $\beta$ be the homology class given by 
\begin{align*}
\beta=([C], 0) \in H_2(X, \mathbb{Z})
\end{align*}
under the decomposition (\ref{H2}). 
By (\ref{Ci=C}), each $C_i$ 
has homology class $\beta$.
The computations in this subsection 
are summarized below
(which is also stated in Proposition~\ref{prop:counter}):  
\begin{prop}\label{prop:counter2}
Suppose that $C$ is of type $I_n$ for $n\ge 2$, i.e. 
$C$ is a circle of $\mathbb{P}^1$ with $n$-irreducible components.  
Then the HST, KL, our definitions, 
and the expected 
answers (from GW or PT theory)
are given in the following table: 
\begin{align*}
\begin{tabular}{|c|c|c|c|c|c|c|} \hline
 & \rm{HST} & \rm{KL} & \rm{ours} & \rm{expected}  \\ \hline
$n_{0, \beta}$  & $-8n$  &  $0$ &  $0$ & $0$   \\ \hline
$n_{1, \beta}$  & $4n$  &  $4n$ &  $4$ & $4$  \\ \hline
$n_{\ge 2, \beta}$  & $0$  &  $0$ & $0$ & $0$ \\ \hline
\end{tabular}
\end{align*}
\end{prop}

As before, we consider the 
moduli space $\Sh_{\beta}(X)$ and the 
Hilbert-Chow map 
\begin{align}\label{HC:KL}
\pi \colon {\Sh}_{\beta}^{\rm{red}}(X) \to \Chow_{\beta}(X).
\end{align}
We have the following lemma on the 
Chow variety $\Chow_{\beta}(X)$:
\begin{lem}\label{lem:chow2}
A closed point of $\Chow_{\beta}(X)$ 
corresponds to a one cycle $\gamma$ on 
$X_C$
in the diagram (\ref{dia:F}) 
satisfying 
$\widehat{p}_{\ast}\gamma=C$
and $p_{\ast}\gamma=0$. 
\end{lem}
\begin{proof}
For a one cycle $\gamma \in \Chow_{\beta}(X)$, the cycle 
$\widehat{p}_{\ast}\gamma$ is a one cycle 
on $S$ with homology class $[C]$.  
Then we have $\widehat{p}_{\ast}\gamma=C$ as $C$ is the 
unique effective one cycle on $S$
 with homology class 
$[C]$. 
In particular, $\gamma$ is supported on $X_C$. 
The last statement $p_{\ast}\gamma=0$ is obvious from 
$p_{\ast}\beta=0$. 
\end{proof}
We first assume that $C$ is of type $I_0$, i.e.  
$C$ is a smooth elliptic curve. 
In this case, we have the following: 
\begin{lem}\label{lem:I0}
If $C$ is of type $I_0$, then (\ref{HC:KL}) is
\begin{align*}
{\Sh}_{\beta}^{\rm{red}}(X) & \to \Chow_{\beta}(X) \\
\amalg_{i=1}^4 C_i & \to 
\{q_1, q_2, q_3, q_4\} \\
x \in C_i & \mapsto q_i.
\end{align*}
In particular, 
(\ref{HC:KL}) is a CY fibration with a CY orientation data. 
\end{lem}
\begin{proof}
Let $\gamma$ be a one cycle on 
$X_C$ satisfying the conditions in Lemma~\ref{lem:chow2}. 
Since $C$ is irreducible, the cycle $\gamma$ must be irreducible as well. 
Then $p(\gamma) \in \mathbb{P}^1$ is a one point, 
so $\gamma$ is 
supported on either on $C_i$ for $1\le i\le 4$ or 
$\widetilde{C} \subset p^{-1}(x)$ for $x \neq q_i$. 
Since $\widehat{p}_{\ast}\widetilde{C}=\sigma_{\ast}\widetilde{C}=2C$,
the latter possibility is excluded. It follows that  
\begin{align*}
\Chow_{\beta}(X)=\{[C_1], [C_2], [C_3], [C_4]\}. 
\end{align*}
Let $\Pic^1(C_i) \cong C_i$ be the 
moduli space of line bundles on $C_i$ of degree $1$.
We have the natural closed embedding
\begin{align}\label{Pic:emb}
\amalg_{i=1}^4 \Pic^1(C_i) \hookrightarrow {\Sh}_{\beta}(X)
\end{align}
which is bijective on closed points as $C_i$ is smooth. 
On the other hand, we have the exact sequence
\begin{align*}
0 \to H^1(\oO_{C_i}) \to \Ext_X^1(\oO_{C_i}, \oO_{C_i})
\to H^0(N_{C_i/X}). 
\end{align*}
Since $N_{C_i/X}$ is a rank two vector bundle 
given by an extension of non-trivial 2-torsion line bundles, 
we have $H^0(N_{C_i/X})=0$ and the second arrow 
of the above sequence is an isomorphism.
This shows that 
(\ref{Pic:emb}) induces isomorphisms on tangent spaces. 
As $\Pic^1(C_i)$ is smooth, the embedding (\ref{Pic:emb}) 
is an isomorphism. 
\end{proof}

By the above lemma, we can define
$n_{g, \beta} \in \mathbb{Z}$ 
as in Definition~\ref{def:GV}. 
Together with the argument
in Subsection~\ref{ex:depend}, we have the following: 
\begin{cor}\label{cor:gv1}
If $C$ is of type $I_0$, we have the identity:   
\begin{align}\label{gv1}
n_{g, \beta}=
\left\{  
\begin{array}{cc}
4, & g=1, \\
0, & g\neq 1. 
\end{array}\right. 
\end{align}
The same identity holds for $n_{g, \beta}^{\rm{KL}}$ if and only if 
we take a CY orientation data. 
\end{cor}
\begin{rmk}
As in the discussion of Subsection~\ref{cor:gv1}, 
the answer (\ref{gv1}) matches with the expected one
(see~\cite[Section~4]{HST}). 
\end{rmk}
We next consider the case 
that $C$ is of 
type $I_n$ for $n\ge 2$, i.e.
$C$ is a nodal curve 
of a circle of $n$ smooth rational curves. 
Such an Enriques surface exists
when $2\le n\le 9$
by~\cite[Theorem~5.7.5]{Dolbook}. 
We denote the irreducible components of $C$
as 
\begin{align*}
C=C^{(1)} \cup \cdots \cup C^{(n)}, \ 
C^{(j)}=\mathbb{P}^1.
\end{align*}
The nodal points are denoted as
\begin{align*}
p^{(j)}=C^{(j)} \cap C^{(j+1)}, \ 
j \in \mathbb{Z}/n\mathbb{Z}.  
\end{align*}
In this case, the Chow variety is 
described as follows: 
\begin{lem}
If $C$ is of type $I_n$
for $n\ge 2$,  
we have 
\begin{align}\label{Chow:I2}
\Chow_{\beta}(X)=E^{\times n}.
\end{align}
For $1\le i\le n$, let 
 $\Gamma_i \subset E^{\times n}$ be the closed
subvariety defined by 
\begin{align*}
&\Gamma_1=\{(x_1, \ldots, x_n) \in E^{\times n} : 
x_1=\cdots=x_n \}, \\
&\Gamma_i=\{(x_1, \ldots, x_n) \in E^{\times n} : 
x_i=\cdots=x_n=-x_1=\cdots=-x_{i-1}\}, \ 
i\ge 2. 
\end{align*}
Then the image of (\ref{HC:KL})
is identified with 
\begin{align}\notag
\Imm \pi =
\bigcup_{i=1}^{n} \Gamma_i \subset E^{\times n}.
\end{align}
The cycle $[C_i] \in \Chow_{\beta}(X)$ corresponds to 
the point $y_i=(e_i, \ldots, e_i)$. 
\end{lem}
\begin{proof}
By Lemma~\ref{lem:chow2}, 
giving a point of $\Chow_{\beta}(X)$ is 
equivalent to giving a 
one cycle $\gamma$ on $X_C$ written as 
\begin{align*}
\gamma=\gamma_1+\cdots +\gamma_n, \ 
\widehat{p}_{\ast}\gamma_j=C^{(j)}.
\end{align*}
Since $\mathbb{P}^1$ is simply connected, 
the fibration $\widehat{p}$ is trivial over 
each irreducible component $C^{(j)} \subset C$. 
It follows that a choice of $\gamma_j$
is equivalent to a choice of 
a section of the
 trivial bundle $E \times C^{(j)} \to C^{(j)}$. 
As there is no non-constant morphism $\mathbb{P}^1 \to E$, 
such a section is determined by a point in 
$E$. 
Therefore
the set of choices of $\gamma$ is identified 
with $E^{\times n}$, and (\ref{Chow:I2}) holds.   

For a stable sheaf $[F] \in {\Sh}_{\beta}(X)$, 
its support must be connected. 
Suppose that a one cycle $\gamma \in \Chow_{\beta}(X)$
corresponds to a point $(x_1, \ldots, x_n) \in E^{\times n}$. 
Since the monodromy of $\widehat{p}|_{X_C}$ around
the generator of 
$\pi_1(C)=\mathbb{Z}$ is given by $\tau \colon E \to E$
sending $x$ to $-x$,  
the cycle $\gamma$ is connected if and only if 
$(x_1, \ldots, x_n) \in \Gamma_i$ for some $i$. 
Conversely if $\gamma$ is connected, then its 
underlying curve is either 
$C$ or a partial normalization of $C$ at 
one of $p^{(j)}$. 
In each case, there exist 
line bundles on it
giving closed points of ${\Sh}_{\beta}(X)$, 
so $\gamma \in \Imm \pi$
holds.  
The case $\gamma=[C_i]$ is only possible 
when $x_1=x_2=\cdots=x_n=-x_1$, so it corresponds to $y_i$. 
\end{proof}
For $1\le i\le 4$ and $1\le j\le n$, we 
denote by $p^{(j)}_{i}$ the 
nodal point in $C_i$
corresponding to $p^{(j)} \in C$. 
We have the following lemma
describing (\ref{HC:KL})
(see Figure~2 for $n=2$ case):
\begin{lem}\label{moduli:pi}
For $y \in 
\cup_{i=1}^n \Gamma_i$, we have 
(set theoretically)
\begin{align*}
\pi^{-1}(y)=\left\{ \begin{array}{cc}
C_i, & y=y_i, \\
\mbox{one point}, & y \neq y_i. 
\end{array}  \right. 
\end{align*}
Moreover the moduli space $\Sh_{\beta}(X)$ is non-singular 
except points 
$p^{(j)}_i$
in $\pi^{-1}(y_i)=C_i$.  
At $p^{(j)}_i$, the singularity of ${\Sh}_{\beta}(X)$ 
is analytically isomorphic 
to the critical locus of
\begin{align}\label{crit:f2}
f \colon 
\mathbb{A}^3 \to \mathbb{A}, \ 
(x, y, z) \mapsto xyz
\end{align}
at the origin $0\in \mathbb{A}^3$. 
\end{lem}
\begin{proof}
It is well-known that the moduli of rank one stable 
sheaves on $C$ with Euler characteristic one 
is isomorphic to $C$ itself, 
by the map
$x \mapsto I_{x}^{\vee}$.  
Therefore $\pi^{-1}(y_i)=C_i$
follows. 
For $y \neq y_i$, 
$\pi^{-1}(y)$ consists of rank one 
stable sheaves $L$ on a 
partial normalization $C' \to C$
at one of $p^{(j)}$
with $\chi(L)=1$. So it
consists of a one point $\{\oO_{C'}\}$. 
For a point in ${\Sh}_{\beta}(X)$ except nodal
points in $\pi^{-1}(y_i)$, the corresponding sheaf is 
a line bundle on the underlying curve. 
Therefore the same argument of Lemma~\ref{lem:I0}
shows that ${\Sh}_{\beta}(X)$ is smooth except points $p_i^{(j)}$. 

By Remark~\ref{rmk:TotKS}, we can also 
describe the 
singularities 
of ${\Sh}_{\beta}(X)$ at $p_i^{(j)}$ by the same argument as 
in Subsection~\ref{ex:enriques}. 
Namely let $\pi_T \colon \cC \to T$ be a versal deformation 
of $C$, with $0\in T$ such that $C=\pi_T^{-1}(C)$. 
We can take $T$ to be a sufficiently small 
open neighborhood of $0 \in \mathbb{A}^n$. 
We define $B \subset \cC \times \mathbb{A}^n$ 
by the commutative diagram
\begin{align}\label{def:M'}
\xymatrix{
B =
\{df'=0\}
 \ar@<-0.3ex>@{^{(}->}[r]
\ar[rd]_{\pi'}
& \cC \times \mathbb{A}^n
 \ar[rd]^{f'} \ar[d]_{\pi_T \times \id} & \\
 & T \times \mathbb{A}^n \ar[r]^{g'} & \mathbb{A}^1. 
}
\end{align}
Here $g'$ is defined by
\begin{align*}
g'(t_1, \ldots, t_n, u_1, \ldots, u_n)=t_1 u_1+\cdots+t_n u_n.
\end{align*}
Then there exist analytic open neighborhoods 
$y_i \in V_i \subset \Chow_{\beta}(X)$, 
$0 \in V \subset T \times \mathbb{A}^n$ such that 
$\pi^{-1}(V_i)$ is isomorphic to $\pi^{'-1}(V)$\footnote{Indeed 
the image of $\pi'$ lies in ${0}\times \mathbb{A}^n$
so one can instead take an open neighborhood
$0 \in V \subset \mathbb{A}^n$}. 

It is easy to see that $B$ is only singular at 
$(p^{(j)}, 0) \in C \times \{0\} \subset \cC \times T$
for $1\le j\le n$, 
and near these points
 $f'$ is analytically isomorphic to
\begin{align}\label{write:f'}
f' \colon 
(x, t_1, \ldots, t_n, u_1, \ldots, u_n) \mapsto
xt_1u_1+t_2u_2+\cdots + t_n u_n
\end{align}
at the origin. 
The critical locus of the above function is isomorphic
to the critical locus of (\ref{crit:f2}). 
\end{proof}

\begin{lem}\label{lem:MCY}
In the above situation, the 
morphism (\ref{HC:KL}) is a CY fibration, and 
one can take a CY orientation data of ${\Sh}_{\beta}(X)$.  
\end{lem}
\begin{proof}
The lemma follows since 
the canonical line bundle of $\cC \times \mathbb{A}^n$
in the diagram (\ref{def:M'}) is trivial. 
\end{proof}

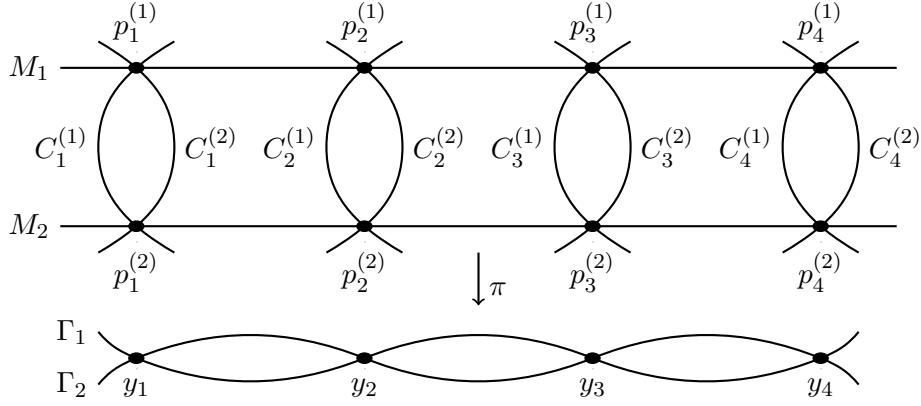
\begin{figure}
\begin{tikzpicture}[yscale=0.7]
\draw[thick] (0, 0) to [out=90, in=220] (1, 2);
\draw[thick] (0, 0) to [out=-90, in=-220] (1, -2);
\draw[thick] (1, 0) to [out=90, in=-40] (0, 2);
\draw[thick] (1, 0) to [out=-90, in=40] (0, -2);
\draw[fill] (0.5, 1.5) circle [radius=0.1] ;
\draw[fill] (0.5, 1.8) circle [radius=0] node[above]{$p_{1}^{(1)}$} ;
\draw[fill] (0.5, -1.5) circle [radius=0.1];
\draw[fill] (0.5, -1.8) circle [radius=0] node[below]{$p_{1}^{(2)}$} ;
\draw[thick] (3, 0) to [out=90, in=220] (4, 2);
\draw[thick] (3, 0) to [out=-90, in=-220] (4, -2);
\draw[thick] (4, 0) to [out=90, in=-40] (3, 2);
\draw[thick] (4, 0) to [out=-90, in=40] (3, -2);
\draw[fill] (3.5, 1.5) circle [radius=0.1];
\draw[fill] (3.5, 1.8) circle [radius=0] node[above]{$p_{2}^{(1)}$} ;
\draw[fill] (3.5, -1.5) circle [radius=0.1];
\draw[fill] (3.5, -1.8) circle [radius=0] node[below]{$p_{2}^{(2)}$} ;
\draw[thick] (6, 0) to [out=90, in=220] (7, 2);
\draw[thick] (6, 0) to [out=-90, in=-220] (7, -2);
\draw[thick] (7, 0) to [out=90, in=-40] (6, 2);
\draw[thick] (7, 0) to [out=-90, in=40] (6, -2);
\draw[fill] (6.5, 1.5) circle [radius=0.1];
\draw[fill] (6.5, 1.8) circle [radius=0] node[above]{$p_{3}^{(1)}$} ;
\draw[fill] (6.5, -1.5) circle [radius=0.1];
\draw[fill] (6.5, -1.8) circle [radius=0] node[below]{$p_{3}^{(2)}$} ;
\draw[thick] (9, 0) to [out=90, in=220] (10, 2);
\draw[thick] (9, 0) to [out=-90, in=-220] (10, -2);
\draw[thick] (10, 0) to [out=90, in=-40] (9, 2);
\draw[thick] (10, 0) to [out=-90, in=40] (9, -2);
\draw[fill] (9.5, 1.5) circle [radius=0.1];
\draw[fill] (9.5, 1.8) circle [radius=0] node[above]{$p_{4}^{(1)}$} ;
\draw[fill] (9.5, -1.5) circle [radius=0.1];
\draw[fill] (9.5, -1.8) circle [radius=0] node[below]{$p_{4}^{(2)}$} ;
\draw[thick] (-0.5, 1.5) node[left]{$M_1$} -- (10.5, 1.5);
\draw[thick] (-0.5, -1.5) node[left]{$M_2$}  -- (10.5, -1.5);
\draw[->][thick] (5, -2) -- (5, -3) node[above right]{$\pi$};

\draw[fill] (0, 0) circle [radius=0] node[left]{$C_{1}^{(1)}$} ;
\draw[fill] (1, 0) circle [radius=0] node[right]{$C_{1}^{(2)}$} ;
\draw[fill] (3, 0) circle [radius=0] node[left]{$C_{2}^{(1)}$} ;
\draw[fill] (4, 0) circle [radius=0] node[right]{$C_{2}^{(2)}$} ;
\draw[fill] (6, 0) circle [radius=0] node[left]{$C_{3}^{(1)}$} ;
\draw[fill] (7, 0) circle [radius=0] node[right]{$C_{3}^{(2)}$} ;
\draw[fill] (9, 0) circle [radius=0] node[left]{$C_{4}^{(1)}$} ;
\draw[fill] (10, 0) circle [radius=0] node[right]{$C_{4}^{(2)}$} ;

\draw[thick] (0.5, -4) to [out=30, in=150] (3.5, -4);
\draw[thick] (3.5, -4) to [out=30, in=150] (6.5, -4);
\draw[thick] (6.5, -4) to [out=30, in=150] (9.5, -4);

\draw[thick] (0.5, -4) to [out=-30, in=210] (3.5, -4);
\draw[thick] (3.5, -4) to [out=-30, in=210] (6.5, -4);
\draw[thick] (6.5, -4) to [out=-30, in=210] (9.5, -4);

\draw[fill] (0.5, -4) circle [radius=0.1];
\draw[fill] (0.5, -4.2) circle [radius=0] node[below]{$y_1$};
\draw[fill] (3.5, -4) circle [radius=0.1];
\draw[fill] (3.5, -4.2) circle [radius=0] node[below]{$y_2$};
\draw[fill] (6.5, -4) circle [radius=0.1];
\draw[fill] (6.5, -4.2) circle [radius=0] node[below]{$y_3$};
\draw[fill] (9.5, -4) circle [radius=0.1];
\draw[fill] (9.5, -4.2) circle [radius=0] node[below]{$y_4$};

\draw[thick] (0.5, -4) to [out=210, in=60] (0, -4.5) node[left]{$\Gamma_2$};
\draw[thick] (0.5, -4) to [out=150, in=-60] (0, -3.5) node[left]{$\Gamma_1$};
\draw[thick] (9.5, -4) to [out=30, in=-120] (10.0, -3.5);
\draw[thick] (9.5, -4) to [out=-30, in=120] (10.0, -4.5);
\end{tikzpicture}
\caption{Picture of HC map for $n=2$}
\end{figure}

Now we compute the KL definition 
for type $I_n$ case with $n\ge 2$, 
and see that it does not match with the 
predicted answer. 
Let us take a (not necessary CY)
orientation data on ${\Sh}_{\beta}(X)$, 
and
\begin{align*}
\phi_{\sS h} \in \Perv({\Sh}_{\beta}(X))
\end{align*}
the perverse sheaf as in Theorem~\ref{thm:BDJS}. 
Let 
\begin{align}\label{irred:M}
M_1, \ldots, M_n \subset {\Sh}_{\beta}(X)
\end{align}
be
the irreducible components 
of ${\Sh}_{\beta}(X)$ which are mapped to 
$\Gamma_i$ by the map (\ref{HC:KL}). 
We denote by 
\begin{align}\label{irred:C}
C_{i}^{(1)}, \ldots, C_i^{(n)} \subset 
\pi^{-1}(y_i)=C_i \subset {\Sh}_{\beta}(X)
\end{align}
the irreducible 
components of $C_i$. 
By Lemma~\ref{moduli:pi}, 
the subvarieties
(\ref{irred:M}), (\ref{irred:C})
form 
the set of irreducible components of ${\Sh}_{\beta}(X)$. 
We also denote by ${\Sh}^{\rm{sm}} \subset {\Sh}_{\beta}(X)$
the smooth locus of ${\Sh}_{\beta}(X)$, whose complement is 
the points $p_i^{(j)}$
by Lemma~\ref{moduli:pi}. 
We first compute the associated graded sheaf
of the weight filtration of $\phi_{\sS h}$. 
\begin{lem}\label{lem:gr}
There exist rank one local systems $\lL_k$
on $M_k \cap {\Sh}^{\rm{sm}}$ 
for $1\le k\le n$ such that $\gr_W^i(\phi_{\sS h})$ is written as
\begin{align}\label{grW}
&\mathrm{gr}_W^0(\phi_{\sS h})=
\bigoplus_{i=1}^n \mathrm{IC}(\lL_i) 
\oplus 
\bigoplus_{i=1}^{4} \bigoplus_{j=1}^n
\mathrm{IC}(C_{i}^{(j)}), \\ 
\notag
&\mathrm{gr}^{\pm 1}_W(\phi_{\sS h})=\bigoplus_{i=1}^{4} \bigoplus_{j=1}^n
\mathbb{Q}_{p_{i}^{(j)}}, \ 
\mathrm{gr}^{k}_W(\phi_{\sS h})=0, \ \lvert k \rvert \ge 2.  
\end{align}
\end{lem}
\begin{proof}
Note that $\phi_{\sS h}|_{{\Sh}^{\rm{sm}}}$ is  
written as $\lL[1]$ for a rank one local system 
$\lL$ on ${\Sh}^{\rm{sm}}$, which is of weight zero by 
our weight convention as in~\cite{BDJS}. 
Note that
any perverse sheaf
on ${\Sh}_{\beta}(X)$
which underlines
a pure Hodge module is a direct sum of 
intersection complexes of some local systems on
dense open subsets of closed irreducible
subvarieties
in ${\Sh}_{\beta}(X)$. 
Therefore, $\gr_W^0(\phi_{\sS h})$ is written as 
\begin{align*}
\mathrm{gr}_W^0(\phi_{\sS h})=
\bigoplus_{i=1}^n \mathrm{IC}(\lL|_{M_i \cap {\Sh}^{\rm{sm}}}) 
\oplus 
\bigoplus_{i=1}^{4} \bigoplus_{j=1}^n
\mathrm{IC}(\lL|_{C_{i}^{(j)} \cap {\Sh}^{\rm{sm}}})
\oplus Q.
\end{align*}
Here $Q$ 
and $\mathrm{gr}_W^k(\phi_{\sS h})$ for $k\neq 0$
are 
supported on points
$p_{i}^{(j)}$. 

Let us first take a CY orientation data of ${\Sh}_{\beta}(X)$
and show that 
$\lL|_{C_{i}^{(j)} \cap {\Sh}^{\rm{sm}}}$ is a trivial local system. 
In the diagram (\ref{def:M'}), 
we set 
\begin{align*}
\phi_{B}=\phi_{f'}(\IC(\cC \times \mathbb{A}^n)) \in \Perv(B). 
\end{align*}
By the description of $f'$ in (\ref{write:f'}) and 
the Thom-Sebastiani theorem, 
$\phi_{B}$ is locally near $(p^{(j)}, 0)$ 
calculated as the vanishing cycle sheaf of (\ref{crit:f2}). 
Let $N$ be the critical locus of (\ref{crit:f2})
and $N_i$ for $1\le i\le 3$ the irreducible 
components of $N$. 
For $\phi_N=\phi_{f}(\IC(\mathbb{A}^3))$, 
its weight filtration is easily computed 
(for example see
the last part of Section~6 in~\cite{Efi})
\begin{align}\label{local:gr}
\mathrm{gr}^0_{W}(\phi_N)=\bigoplus_{i=1}^3 \mathrm{IC}(N_i), 
\ \mathrm{gr}^{\pm 1}_W(\phi_N)=\mathbb{Q}_0, \ 
\mathrm{gr}^i_W(\phi_N)=0, \ \lvert i \rvert \ge 2. 
\end{align}
The above local 
calculation implies that 
the monodromy of 
$\lL|_{C_{i}^{(j)} \cap {\Sh}^{\rm{sm}}}$
around $p_i^{(j)}$ is trivial. 
Since $C_i^{(j)} \cap {\Sh}^{\rm{sm}}=\mathbb{C}^{\ast}$, 
the monodromy around $p_i^{(j)}$ determines the local 
system on it. Therefore 
$\lL|_{C_{i}^{(j)} \cap {\Sh}^{\rm{sm}}}$ is trivial. 
The computation (\ref{local:gr}) also 
shows that $Q=0$ and 
gives the result for 
$\mathrm{gr}_W^k(\phi_{\sS h})$ with $k\neq 0$.
Therefore 
we obtain the identities (\ref{grW}) 
for a CY orientation data of ${\Sh}_{\beta}(X)$. 

Next let us take an orientation data
of ${\Sh}_{\beta}(X)$ which is not necessary CY, and 
denote by $\phi_{\sS h'}$ the resulting 
perverse sheaf on ${\Sh}_{\beta}(X)$. 
Then by Theorem~\ref{thm:BDJS}, 
$\phi_{\sS h'}=\phi_{\sS h} \otimes \lL'$ for some
rank one local system $\lL'$ on ${\Sh}_{\beta}(X)$. 
Since $\otimes \lL'$ preserves the purity, and each 
$\lL'|_{C_i^{(j)}}$ is trivial 
as $C_i^{(j)}=\mathbb{P}^1$ is simply connected, 
the associated graded sheaf 
$\gr_W^{\bullet}(\phi_{\sS h'})$ is still of the form (\ref{grW}). 
\end{proof}

\begin{lem}\label{n1=8}
Suppose that
 $C$ is of type $I_n$ for $n\ge 2$.
Then for any choice of orientation data of ${\Sh}_{\beta}(X)$,
we 
have $n_{1, \beta}^{\rm{KL}}=4n$. 
\end{lem}
\begin{proof}
By Lemma~\ref{lem:gr}, we 
have 
\begin{align*}
\sum_{i\in \mathbb{Z}} 
\chi(\pH^i(\dR \pi_{\ast}\gr_{W}^{\bullet}(\phi_{\sS h})))
y^i
=4n(y^{\frac{1}{2}}+y^{-\frac{1}{2}})^{2}+(\mathrm{constant}). 
\end{align*}
Then the lemma follows by Remark~\ref{rmk:HSTKL}.  
\end{proof}
Since an Enriques surface with 
an $I_n$ type
double fiber can 
be deformed to a one with an $I_0$ type 
double fiber, by Lemma~\ref{lem:I0}, \ref{n1=8}, 
we have the following:
\begin{cor}
The Kiem-Li invariant $n_{g, \beta}^{\rm{KL}}$ is not deformation invariant. 
\end{cor}

Now we show that our invariants
agree with 
the result in Corollary~\ref{cor:gv1}. 
So we can observe that our invariants are deformation invariant in this example 
despite the fact that the Chow variety jumps in dimension.
Below we fix a CY orientation data of ${\Sh}_{\beta}(X)$.  

\begin{lem}\label{lem:match}
Suppose that $C$ is of type $I_n$
for $n\ge 2$.
Then 
we have $n_{1, \beta}=4$. 
\end{lem}
\begin{proof}
Let us consider the spectral sequence
(\ref{weight:spectral}).
The differential $E_1^{0, -1} \to E_1^{1, -1}$ is 
induced by taking 
the $\pH^0( \dR \pi_{\ast}(-))$ of the canonical 
morphism
\begin{align}\label{can:weight}
\mathrm{gr}^0_{W}(\phi_{\sS h})[-1] \to \mathrm{gr}_W^{-1}(\phi_{\sS h}).
\end{align}
Let $\nu_i$ be the normalization of $C_i$
\begin{align*}
\nu_i \colon \widetilde{C}_i \cneq \coprod_{j=1}^n C_i^{(j)}
\to C_i.
\end{align*}
By Lemma~\ref{lem:gr}, we have 
\begin{align}\notag
E_1^{0, -1}=
\bigoplus_{i=1}^4 
H^0(\widetilde{C}_i, \mathbb{Q})
\otimes \mathbb{Q}_{y_i}, \ 
E_1^{1, -1}= 
\bigoplus_{i=1}^4 \bigoplus_{j=1}^n
H^0(p_i^{(j)}, \mathbb{Q}_{p_{i}^{(j)}} )
\otimes \mathbb{Q}_{y_i}.
\end{align} 
The differential 
$E_1^{0, -1} \to E_1^{1, -1}$ at $y_i$ 
is induced by taking the global sections of
the exact sequences 
\begin{align*}
0 \to \mathbb{Q}_{C_i} \to 
\nu_{i\ast}\mathbb{Q}_{\widetilde{C}_i}
\to \bigoplus_{j=1}^{n}
\mathbb{Q}_{p_i^{(j)}} \to 0. 
\end{align*}
Indeed this follows from the local calculation 
of the function $(x, y, z) \mapsto xyz$
as in (\ref{local:gr}), where we can easily see that the canonical 
morphism
\begin{align*}
\bigoplus_{i=1}^3 \mathbb{Q}_{N_i}
=\mathrm{gr}_W^0(\phi_N)[-1]
\to 
\mathrm{gr}_W^{-1}(\phi_N)=\mathbb{Q}_0
\end{align*}
is the surjection of sheaves
which is non-zero on each component $N_i$. 

Therefore $E_1^{0, -1} \to E_1^{1, -1}$ is rank $n-1$ at $y_i$, 
and the $E_2$ term is
\begin{align*}
E_2^{0, -1}=\bigoplus_{i=1}^4 \mathbb{Q}_{y_i}. 
\end{align*}
Since $E_2^{0, -1}$ is the only term 
which contributes to $\pH^{-1}(\dR \pi_{\ast}\phi_{\sS h})$, 
and the spectral sequence does not contribute 
to $\pH^{i}(\dR \pi_{\ast}\phi_{\sS h})$
for $\lvert i \rvert \ge 2$, we have
\begin{align*}
\sum_{i\in \mathbb{Z}} \chi(\pH^i(\dR \pi_{\ast}\phi_{\sS h}))
y^i
=4(y^{\frac{1}{2}}+y^{-\frac{1}{2}})^{2}+(\mathrm{constant}). 
\end{align*}
Therefore the lemma follows. 
\end{proof}

\begin{rmk}
In type $I_n$ case with $n\ge 2$, the genus zero 
invariant $n_{0, \beta}$ can 
be directly checked to be zero as follows. 
As a singularity of ${\Sh}_{\beta}(X)$ 
is the critical locus of (\ref{crit:f2}), 
the Behrend function on ${\Sh}_{\beta}(X)$ is constant $-1$, 
and $n_{0, \beta}=-e({\Sh}_{\beta}(X))$. 
By Lemma~\ref{moduli:pi},
we have
\begin{align*}
e({\Sh}_{\beta}(X))=4 \cdot e(C) + n \cdot (e(E)-4)=0
\end{align*} 
as expected. 
\end{rmk}

\begin{rmk}
Contrary to Lemma~\ref{n1=8}, 
our GV invariant in Lemma~\ref{lem:match}
gives a different answer if we take a non-CY 
orientation data, as it affects
the map $E_1^{0, -1} \to E_1^{1, -1}$ in the proof of 
Lemma~\ref{lem:match}.  
\end{rmk}

The HST invariants are similarly computed, 
which we leave the readers for details. 
The computations in this subsection
are summarized in the table in Proposition~\ref{prop:counter2}.

\section{Non-reduced examples from 3-fold flops}\label{sec:non-red}
In this section, using the results of the previous sections and 
derived equivalences under 3-fold flops, we 
give some examples where Conjecture~\ref{GV:conj2} holds for 
non-reduced, non-planar one cycles. 
\subsection{3-fold flops}
Let $X$, $X^{\dag}$ be smooth quasi-projective CY 3-folds
which are connected by a flop
\begin{align}\label{flop:dia}
\xymatrix{
X \ar[dr]_{f}  \ar@{-->}[rr]^{\phi} &  & X^{\dag} \ar[dl]^{f^{\dag}} \\
&  Y.   &
}
\end{align}
This means that $f, f^{\dag}$ are birational morphisms 
which are isomorphic in codimension one
with relative Picard number one,  
$Y$ has only Gorenstein singularities, and 
$\phi$ is a non-isomorphic birational map.  
The exceptional loci of $f, f^{\dag}$ are 
chains of smooth rational curves.
By the result of Bridgeland~\cite{Br1}, 
there is an equivalence of derived categories
\begin{align}\label{D:equiv0}
\Phi \colon D^b(\Coh(X)) \stackrel{\sim}{\to}
 D^b(\Coh(X^{\dag}))
\end{align}
given by the Fourier-Mukai transform 
whose kernel is  
$\oO_{X \times_Y X^{\dag}}$. 
The above equivalence 
restricts to the equivalence of 
triangulated subcategories 
(see~\cite[Proposition~5.2]{ToBPS})
\begin{align}\label{D:equiv}
\Phi \colon D^b(\Coh_{\le 1}(X)) \stackrel{\sim}{\to}
 D^b(\Coh_{\le 1}(X^{\dag}))
\end{align}
The equivalence (\ref{D:equiv}) preserves the
hearts of perverse t-structures. 
Namely there exist hearts of bounded t-structures
defined in~\cite[Section~3]{Br1}
\begin{align}\label{p:heart}
&\pPPer_{\le 1}(X/Y) \\ \notag
&\cneq \left\{ E \in D^b (\Coh_{\le 1}(X)) \colon 
\begin{array}{c}
\dR f_{\ast} E \in \Coh(Y) \\
\Hom^{<-p}(E, \cC_X)=\Hom^{<p}(\cC_X, E)=0 
\end{array} \right\}
\end{align}
where $\cC_X$ is defined by
\begin{align*}
\cC_X \cneq \{ F \in \Coh(X) : \dR f_{\ast} F=0\}.
\end{align*}
Below we always take $p\in \{-1, 0\}$. 
The equivalence (\ref{D:equiv}) 
restricts to the equivalence
\begin{align}\label{equiv:Per}
\Phi \colon \oPPer_{\le 1}(X/Y) \stackrel{\sim}{\to}
 \iPPer_{\le 1}(X^{\dag}/Y). 
\end{align}
We can describe the hearts (\ref{p:heart})
in terms of tilting. 
Let $\Coh(X/Y)$ be the subcategory of 
$\Coh_{\le 1}(X)$ consisting 
of sheaves supported on the exceptional locus of $f$. 
For an ample divisor $\omega$ on $X$, we set 
\begin{align*}
\oF& \cneq 
\langle E \in \Coh(X/Y) : E \mbox{ is }\omega
\mbox{-semistable with }
\mu_{\omega}(E) < 0 \rangle \\
\iF& \cneq 
\langle E \in \Coh(X/Y) : E \mbox{ is } \omega 
\mbox{-semistable with }
\mu_{\omega}(E) \le 0 \rangle. 
\end{align*}
Here $\langle \ast \rangle$ is the 
smallest extension closed subcategory which contains $\ast$. 
Let $\pT$ be the orthogonal complement of $\pF$
\begin{align*}
\pT \cneq \{ E \in \Coh_{\le 1}(X) \colon 
\Hom(E, \pF)=0\}. 
\end{align*}
\begin{lem}\label{lem:per/tilt}
We have the identity
\begin{align*}
\pPPer(X/Y)=\langle \pF[1], \pT \rangle. 
\end{align*}
\end{lem}
\begin{proof}
See~\cite[Lemma~2.5]{TodS}. 
\end{proof}

\begin{rmk}\label{rmk:fact}
Below, we will use the 
fact that  
the equivalence $\Phi$ commutes with 
$\dR f_{\ast}$ and $\dR f^{\dag}_{\ast}$
(see~\cite{Br1}). 
By the definition of $\pPPer_{\le 1}(X/Y)$, we have
\begin{align*}
\cC_X[-p]=\{E \in \pPPer_{\le 1}(X/Y) : \dR f_{\ast}E=0\}. 
\end{align*}
In particular, 
$\Phi$ restricts to the equivalence 
of $\cC_{X}$ and $\cC_{X^{\dag}}[1]$. 
\end{rmk}

\subsection{Isomorphism of moduli spaces}
Let us consider the moduli space of 
one dimensional stable sheaves 
$\Sh_{\beta}(X)$ as in (\ref{moduli:M}). 
We have the following lemma: 
\begin{lem}\label{lem:Eper}
For $[E] \in \Sh_{\beta}(X)$, we have
\begin{align}\label{EPCoh}
E \in \pT =\pPPer_{\le 1}(X/Y) \cap \Coh_{\le 1}(X).
\end{align}
\end{lem}
\begin{proof}
By the $h$-stability of $E$, 
we have $\Hom(E, \pF)=0$, 
hence $E \in \pT$ follows. 
The right identity of (\ref{EPCoh})
 is due to Lemma~\ref{lem:per/tilt}. 
\end{proof}
We define the open subscheme
\begin{align*}
\Sh_{\beta}^{\circ}(X)
\subset \Sh_{\beta}(X)
\end{align*}
to 
be 
consisting of sheaves $E$ 
whose supports are irreducible 
and not contained in $\Ex(f)$.  
For the birational map $\phi$ in (\ref{flop:dia}),
let
\begin{align*}
\phi_{\ast} \colon H_2(X, \mathbb{Z}) \stackrel{\cong}{\to}
H_2(X^{\dag}, \mathbb{Z})
\end{align*} 
be the induced map on homology groups. 
\begin{lem}\label{lem:Phi(E)}
For $[E] \in \Sh_{\beta}^{\circ}(X)$, we have 
$[\Phi(E)] \in 
\Sh_{\phi_{\ast}\beta}(X^{\dag})$. 
\end{lem}
\begin{proof}
We have $\Phi(E) \in \iPPer_{\le 1}(X^{\dag}/Y)$
by Lemma~\ref{lem:Eper}
and the equivalence (\ref{equiv:Per}). 
Let us set $A=\hH^{-1}(\Phi(E))$ and $B=\hH^0(\Phi(E))$. 
We have the exact sequence
\begin{align*}
0 \to A[1] \to \Phi(E) \to B \to 0
\end{align*}
in $\iPPer_{\le 1}(X^{\dag}/Y)$. 
By applying $\dR f^{\dag}_{\ast}$ and noting 
Remark~\ref{rmk:fact}, we 
obtain the exact sequence of sheaves
\begin{align*}
0 \to \dR f_{\ast}^{\dag}(A[1]) \to f_{\ast}E \to 
\dR f_{\ast}^{\dag}B \to 0. 
\end{align*}
Since $f_{\ast}E$ is a pure one dimensional 
sheaf on $Y$, we have $\dR f_{\ast}^{\dag}(A[1])=0$, i.e. 
$A \in \cC_{X^{\dag}}$. 
Suppose that $A \neq 0$. 
Then 
$\Phi^{-1}(A[1])$ is a non-zero object in 
$\cC_{X}$ (see Remark~\ref{rmk:fact}), which admits a non-zero morphism to 
$E$. This contradicts to the assumption that 
the support of $E$ does not contain irreducible components 
in $\Ex(f)$. Hence $A=0$, and 
$\Phi(E) \in \Coh_{\le 1}(X^{\dag})$ follows. 

It remains to show $\Phi(E)$ is 
a stable sheaf.  
Let 
\begin{align*}
0 \to P \to \Phi(E) \to Q \to 0
\end{align*}
be an exact sequence in $\Coh_{\le 1}(X^{\dag})$
with $P, Q \neq 0$. 
By pushing forward to $Y$, we obtain the exact 
sequence of sheaves
\begin{align*}
0 \to f_{\ast}^{\dag} P \to f_{\ast}E \to 
f_{\ast}^{\dag}Q \to R^1 f_{\ast}^{\dag}P. 
\end{align*}
Since $\dR f_{\ast}E=f_{\ast}E$ is a stable sheaf
with Euler characteristic one,  
we 
have 
$\chi(f_{\ast}^{\dag}P) \le 1$. 
Since $R^1 f_{\ast}^{\dag}P$ is a zero dimensional sheaf, we have
\begin{align*}
\chi(P)=\chi(f_{\ast}^{\dag}P) -\chi(R^1 f_{\ast}^{\dag}P) \le 1. 
\end{align*}
In order to conclude 
the stability of 
 $\Phi(E)$, 
we need to exclude the case of $\chi(P)=1$. In this case, 
$\dR f^{\dag}_{\ast}Q=0$, i.e. 
$Q \in \cC_{X^{\dag}}$. 
Since $\Phi^{-1}(Q) \in \cC_{X}[-1]$, 
and $\Hom(E, \cC_X[-1])=0$, we obtain the contradiction. 
\end{proof}
We define the open subscheme 
\begin{align*}
\Sh_{\phi_{\ast}\beta}^{\circ}(X^{\dag}) \subset
\Sh_{\phi_{\ast}\beta}(X^{\dag})
\end{align*}
to be consisting of sheaves $E$
such that the support of $\Phi^{-1}(E)$
is irreducible and not contained in $\Ex(f)$.  
\begin{prop}\label{prop:Phisom}
The equivalence $\Phi$ induces the isomorphism
\begin{align}\label{isom:Phi}
\Phi_{\ast} \colon 
\Sh_{\beta}(X) \stackrel{\cong}{\to}
\Sh_{\phi_{\ast}\beta}(X^{\dag}). 
\end{align}
\end{prop}
\begin{proof}
By the definition of 
$\Sh_{\phi_{\ast}\beta}^{\circ}(X^{\dag})$
and Lemma~\ref{lem:Phi(E)},
we have the well-defined morphism (\ref{isom:Phi}). 
Let us take an object
 $[E] \in \Sh_{\phi_{\ast}\beta}^{\circ}(X^{\dag})$. 
It is enough to show that $\Phi^{-1}(E)$ is an object in 
$\Sh_{\beta}^{\circ}(X)$. 
We first note that $\dR f^{\dag}_{\ast}E=f^{\dag}_{\ast}E$
is a pure sheaf on $Y$. Indeed otherwise, 
there is $y\in Y$ and a non-zero morphism 
$\oO_y \to f^{\dag}_{\ast}E$. 
By the adjunction, 
we have the non-zero morphism 
$f^{\dag \ast}\oO_y \to E$. 
Since $f^{\dag \ast}\oO_y$ is a stable sheaf 
(see the proof of~\cite[Lemma~3.2]{Katz})
with 
Euler characteristic one
supported on $\Ex(f^{\dag})$, this contradicts to the stability of $E$. 
Therefore $f_{\ast}^{\dag}E$ is a pure sheaf. 

Next we show that $\Phi^{-1}(E)$ is a coherent sheaf. 
By Lemma~\ref{lem:Eper},
we have $\Phi^{-1}(E) \in \oPPer_{\le 1}(X/Y)$. 
We set $A=\hH^{-1}(\Phi^{-1}(E))$
and 
$B=\hH^0(\Phi^{-1}(E))$. We have the exact sequence in 
$\oPPer_{\le 1}(X/Y)$
\begin{align}\label{exact:AEB}
0 \to A[1] \to \Phi^{-1}(E) \to B \to 0. 
\end{align}
Suppose that $A \neq 0$. 
Then $\dR f_{\ast}(A[1])$ is 
a subsheaf of $f^{\dag}_{\ast}E$ which is at most 
zero dimensional. 
By the purity of $f^{\dag}_{\ast}E$, it follows that 
$\dR f_{\ast}(A[1])=0$, i.e. 
$A \in \cC_{X}$. 
This implies $A \in \oPPer_{\le 1}(X/Y)$, 
which contradicts to 
that (\ref{exact:AEB}) is an exact sequence in 
$\oPPer_{\le 1}(X/Y)$. 
Hence
$A=0$
and $\Phi^{-1}(E) \in \Coh_{\le 1}(X)$
follows. 

It remains to show that $\Phi^{-1}(E)$ is
a stable sheaf, or equivalently 
a 
pure sheaf as the support of $\Phi^{-1}(E)$ is irreducible. 
If otherwise, there is $x \in X$ and
a non-zero morphism $\oO_x \to \Phi^{-1}(E)$. 
Let $P$ be its image in $\oPPer_{\le 1}(X/Y)$. 
Then $\dR f_{\ast}P$ is a subsheaf of $f_{\ast}^{\dag}E$
which is at most zero dimensional, hence $\dR f_{\ast}P=0$
by the purity of $f_{\ast}^{\dag}E$. 
Then $P \in \cC_{X}$, and 
$\Phi(P) \in \cC_{X^{\dag}}[1]$
is a subobject of $E$ in $\iPPer_{\le 1}(X^{\dag}/Y)$. 
Since $\Hom(\cC_{X^{\dag}}[1], E)=0$
as $E$ is a sheaf, this is a contradiction. 
\end{proof}

\begin{rmk}\label{rmk:compare:dcrit}
The equivalence of derived categories (\ref{D:equiv0})
induces the isomorphism
\begin{align}\label{isom:HK}
H^0(X, K_X) \stackrel{\cong}{\to} H^0(X^{\dag}, K_{X^{\dag}})
\end{align}
by taking the induced isomorphism on Hochschild homologies. 
Therefore 
the trivialization (\ref{isom:omega})
induces the trivialization 
$\oO_{X^{\dag}} \stackrel{\cong}{\to}
K_{X^{\dag}}$, hence
a $d$-critical stricture on $\Sh_{\phi_{\ast}\beta}(X^{\dag})$
by Theorem~\ref{thm:CYM}. 

The isomorphism (\ref{isom:Phi})
can be proved to preserve the $d$-critical 
structures of both sides. 
Indeed the 
Fourier-Mukai equivalence (\ref{D:equiv0})
lifts to a 
dg quasi-functor 
between the enhancements of both sides of
(\ref{D:equiv0}), therefore 
the derived moduli 
schemes 
$\widehat{\Sh}_{\beta}(X)$ and 
$\widehat{\Sh}_{\phi_{\ast}\beta}(X^{\dag})$ in 
Remark~\ref{rmk:shifted}
are equivalent. 
On the other hand 
by the announced work~\cite[Theorem~1.2]{CBTD}, 
the $(-1)$-shifted 
symplectic structures on the above 
derived schemes are canonically constructed
from the CY dg enhancements 
(see~\cite[Theorem~1.2]{CBTD}), 
where the CY structures
are given by non-zero elements in (\ref{isom:HK})
by~\cite[Lemma~5.11]{CBTD}.
By taking the truncations of $\widehat{\Sh}_{\beta}(X)$ and 
$\widehat{\Sh}_{\phi_{\ast}\beta}(X^{\dag})$, we 
have the matching of $d$-critical structures of both sides of 
(\ref{isom:Phi}).
\end{rmk}


\subsection{Comparison of GV invariants}
We define the open subset
\begin{align*}
U_{\beta} \subset \Chow_{\beta}(X)
\end{align*}
to be consisting of irreducible one 
cycles which do not contain 
irreducible components of $\Ex(f)$. 
For $\gamma \in U_{\beta}$, 
by taking pull-back and push-forward along with 
the morphisms
\begin{align*}
X \leftarrow X \times_Y X^{\dag} \rightarrow X^{\dag}
\end{align*}
we have the map
\begin{align*}
\phi_{\ast} \colon 
U_{\beta} \to \Chow_{\phi_{\ast}\beta}(X^{\dag}). 
\end{align*}
The above map is obviously injective. 
Let $U_{\phi_{\ast}\beta}$ be the image of the above map. 
We have the commutative diagram
\begin{align*}
\xymatrix{
\Sh_{\beta}^{\rm{red}}(X)
 \ar[r]^{\Phi_{\ast}} \ar[d] &  
\Sh_{\phi_{\ast}\beta}^{\rm{red}}(X^{\dag})
\ar[d] \\
U_{\beta} \ar[r]_{\phi_{\ast}} & U_{\phi_{\ast}\beta}.
}
\end{align*}
Here the vertical arrows are 
Hilbert-Chow maps. 
Then from Proposition~\ref{prop:Phisom}, we 
obtain the following: 
\begin{cor}\label{cor:ng}
For $\gamma \in U_{\beta}$, the 
$d$-critical scheme 
$(\Sh_{\beta}(X), s_{\Sh})$ is CY at $\gamma$ if and only if 
$(\Sh_{\phi_{\ast}\beta}(X^{\dag}), s_{\Sh})$ is 
CY at $\phi_{\ast}\gamma$. In this case, we have the identity
$n_{g, \gamma}^{\rm{loc}}=n_{g, \phi_{\ast}\gamma}^{\rm{loc}}$. 
\end{cor}
As for stable pairs, 
let us set
\begin{align*}
&\PT(X)\cneq 1+\sum_{n\in \mathbb{Z}, \beta>0}P_{n, \beta}q^n t^{\beta}, \\ 
&\PT(X/Y) \cneq 1+\sum_{n\in \mathbb{Z}, f_{\ast}\beta=0}
P_{n, \beta}q^n t^{\beta}. 
\end{align*}
The following result is proved using 
wall-crossing formulas in the derived category:
\begin{thm}\emph{(\cite{Tcurve2, Cala})}
We have the following identity:
\begin{align}\label{PT:flop}
\phi_{\ast}
\frac{\PT(X)}{\PT(X/Y)}
=\frac{\PT(X^{\dag})}{\PT(X^{\dag}/Y)}. 
\end{align}
Here $\phi_{\ast}$ is the variable change 
$t^{\beta} \mapsto t^{\phi_{\ast}\beta}$. 
\end{thm}
By taking the logarithm of 
both sides of (\ref{PT:flop}) and comparing the coefficient at 
$t^{\beta}$, we have
\begin{align*}
n_{g, \beta}^P=n_{g, \phi_{\ast}\beta}^P, 
\ \beta \in H_2(X, \mathbb{Z}), 
 \ f_{\ast}\beta >0. 
\end{align*}
The arguments of~\cite{Tcurve2, Cala} 
also apply to the local version, 
which give 
\begin{align*}
n_{g, \gamma}^{P, \rm{loc}}
=n_{g, \phi_{\ast}\gamma}^{P, \rm{loc}}, \ \gamma \in \Chow_{\beta}(X), \ 
f_{\ast}\gamma>0. 
\end{align*}
By combining with Corollary~\ref{cor:ng}, we obtain the following:

\begin{cor}\label{cor:ng2}
Under the situation of Corollary~\ref{cor:ng}, 
for $\gamma \in U_{\beta}$ 
we have $n_{g, \gamma}^{P, \rm{loc}}=n_{g, \gamma}^{\rm{loc}}$
if and only if $n_{g, \phi_{\ast}\gamma}^{P, \rm{loc}}
=n_{g, \phi_{\ast}\gamma}^{\rm{loc}}$. 
\end{cor}

\subsection{Examples via flops}
For $\gamma \in U_{\beta}$, the 
one cycle $\phi_{\ast}\gamma$ 
is not a reduced cycle if it intersects with 
$\Ex(f)$ with multiplicity bigger than or equal to two. 
So if we know 
$n_{g, \gamma}^{P, \rm{loc}}=n_{g, \gamma}^{\rm{loc}}$, 
by Corollary~\ref{cor:ng2}
we obtain examples 
of non-reduced one cycles $\gamma'$
where $n_{g, \gamma'}^{P, \rm{loc}}=n_{g, \gamma'}^{\rm{loc}}$ holds. 
We give two examples 
where such an argument applies. 

First, the following corollary obviously 
follows from Theorem~\ref{thm:smooth}
and Corollary~\ref{cor:ng2}:  
\begin{cor}
Let $\phi \colon X \dashrightarrow X^{\dag}$ 
be a flop as in (\ref{flop:dia}), and 
$C \subset X$ a smooth curve 
which is not contained in (but may intersect with)
the exceptional locus of $\phi$. 
Then Conjecture~\ref{GV:conj2}
holds for the one cycle $\phi_{\ast}[C]$ on $X^{\dag}$. 
\end{cor}

We state the next example. 
Let 
$S$ be a smooth projective surface with $H^1(\oO_S)=0$, and 
take a blow-up 
\begin{align*}
h \colon S^{\dag} \to S
\end{align*}
at a point $p\in S$. 
Then there exist smooth projective 3-folds $X$, $X^{\dag}$
and a flop diagram (\ref{flop:dia})
satisfying the following conditions
(see~\cite[Lemma~4.2]{TodS}) 
\begin{itemize}
\item Both of the
 exceptional locus $Z=\Ex(f)$, $Z^{\dag} =\Ex(f^{\dag})$
are 
irreducible 
$(-1, -1)$-curves. 
\item There are closed embeddings 
\begin{align}\label{emb:S}
i \colon S \hookrightarrow X, \quad 
i^{\dag} \colon S^{\dag} \hookrightarrow X^{\dag}
\end{align}
such that 
$S \cap Z$ consists of one point, 
the strict transform of $S$
in $X^{\dag}$ 
coincides with $S^{\dag}$, 
and $Z^{\dag} \subset S^{\dag}$ coincides with 
the exceptional locus of $h \colon S^{\dag} \to S$.   
\item There are open neighborhoods 
$S \subset X_{0}$, $S^{\dag} \subset X_0^{\dag}$
and isomorphisms
\begin{align}\label{isom:KS}
X_0 \cong \mathrm{Tot}(K_S), \ 
X_0^{\dag} \cong \mathrm{Tot}(K_{S^{\dag}})
\end{align}
such that the embeddings (\ref{emb:S})
are identified with the zero sections. 
\end{itemize}
Below we regard 
one cycles on $S$, $S^{\dag}$
as one cycles on $X$, $X^{\dag}$
by isomorphisms (\ref{isom:KS}) and 
zero sections. 
Applying the 
result of Theorem~\ref{thm:locsur} and the 
argument of Corollary~\ref{cor:ng}, we have the following: 
\begin{cor}\label{cor:blowup}
(i) 
For any irreducible 
curve $C \subset S$, 
Conjecture~\ref{GV:conj2} holds for the one cycle 
$\phi_{\ast}C=h^{\ast}C$ on $X^{\dag}$. 

(ii) 
For any irreducible curve $C^{\dag} \subset S^{\dag}$, 
Conjecture~\ref{GV:conj2} holds for the one cycle 
$\phi_{\ast}^{-1}C^{\dag}$. 
\end{cor}

\begin{rmk}
Although $X$, $X^{\dag}$ may 
not be CY, they are CY at
$X_0 \cup C$ and $X_0^{\dag}$, so the 
statements of Corollary~\ref{cor:blowup} make sense. 
\end{rmk}

\begin{rmk}
In Corollary~\ref{cor:blowup} (i), suppose that the multiplicity of $C$ at $p$ is $m$. 
Then $\phi_{\ast}C=\overline{C}+mZ$, 
where $\overline{C} \subset S^{\dag}$ is the strict transform of $C$. 
In particular, $\phi_{\ast}C$ is not reduced, but it is planar. 
\end{rmk}
\begin{rmk}
In Corollary~\ref{cor:blowup} (ii), suppose that 
the curves $C^{\dag}$, $Z^{\dag}$ in 
$S^{\dag}$ intersect
 with 
multiplicity $m$. Then $\phi_{\ast}^{-1}C^{\dag}=h(C)+mZ$. 
In this case, the cycle 
$\phi_{\ast}^{-1}C^{\dag}$ is not reduced, not planar. 
\end{rmk}

\section{Non-primitive examples}\label{sec:non-prim}
We give some examples where Conjecture~\ref{GV:conj2} holds for 
non-primitive one cycles
and $g\ge 1$. 
\subsection{Super rigid elliptic curves}
Let 
$C$ be an elliptic curve 
and 
\begin{align*}
X=\mathrm{Tot}(L \oplus L^{-1})
\end{align*} for 
a for generic line bundle $L$ on $C$ with 
degree zero.
Note that $X$ is a non-compact CY 3-fold, with 
unique projective curve $C \subset X$
given by the zero section.   
For $\beta=m[C]$, the 
Chow variety $\Chow_{\beta}(X)$ consists of 
one point $m[C]$, and the moduli 
space 
$\Sh_{m[C]}(X)$ is isomorphic to 
$C$ itself (see~\cite[Proposition~4.4]{HST}). 
Therefore we have 
\begin{align*}
n_{1, m[C]}=0, \ n_{g, m[C]}=0, \ g\neq 1. 
\end{align*}
As $C$ is smooth, our GV invariants 
agree with the invariants defined in~\cite{HST}. 
In this case, 
Conjecture~\ref{conj:GV/GW} 
is checked in~\cite{HST} using the result of~\cite{PHodge}. 
Combined with DT/GW correspondence for 
local curves~\cite{BrPa}, 
Conjecture~\ref{GV:conj2} holds for the one cycle $\gamma=m[C]$. 

\subsection{Elliptic fibrations}
Let $X$ be a smooth projective CY 3-fold with 
an elliptic fibration
\begin{align*}
\pi \colon X \to S
\end{align*}
such that every scheme theoretic fiber is an integral curve. 
Let $F \in H_2(X, \mathbb{Z})$ be a fiber class 
of $\pi$, and set $\beta=n[F]$. 
Then 
$\Chow_{\beta}(X)=\mathrm{Sym}^n(S)$ and 
we have the commutative diagram
\begin{align}\label{dia:HCM}
\xymatrix{
X \ar[d]_{\pi} \ar[r]^(.4){\cong} & \Sh_{\beta}(X) \ar[d] \\
S \ar@<-0.3ex>@{^{(}->}[r] &  \mathrm{Sym}^n(S). 
}
\end{align}
Here the right arrow is the Hilbert-Chow map, 
and the bottom arrow is the diagonal map. 
We have the decomposition
\begin{align*}
\dR \pi_{\ast} \mathrm{IC}(X)=
\mathrm{IC}(S)[1] \oplus V \oplus \mathrm{IC}(S)[-1]
\end{align*}
where $V=R^1 \pi_{\ast}\mathbb{Q}_X[2]$ is a
perverse sheaf on $S$. 
For $s \in S$, let $X_s$ be the fiber of $\pi$ at $s$
which is either an elliptic curve, 
rational curve with one node or a cusp. 
In any case, we have
\begin{align*}
R^1 \pi_{\ast}\mathbb{Q}_X|_{s}=H^1(X_s, \mathbb{Q})=\mathbb{Q}^{2-e(X_s)}. 
\end{align*}
Then an easy calculation shows 
\begin{align*}
\chi(\mathrm{IC}(S))y^{-1}+\chi(V)+\chi(\mathrm{IC}(S))y
=-e(X)+e(S)(y^{\frac{1}{2}}+y^{-\frac{1}{2}})^2. 
\end{align*}
By the diagram (\ref{dia:HCM}), we obtain 
\begin{align*}
n_{0, \beta}=-e(X), \ n_{1, \beta}=e(S), \ 
n_{g, \beta}=0, \ g \ge 2.
\end{align*}
The invariants $n_{g, \beta}^P$ from stable pairs
are computed in~\cite[Theorem~6.9]{Tsurvey} via wall-crossing method,
and completely agree with the above $n_{g, \beta}$. 
The argument here is easily applied to the local 
version, proving Conjecture~\ref{GV:conj2}
for the one cycle $\gamma=n[X_s]$ for some $s \in S$.

\subsection{Hitchin moduli spaces}\label{subsec:hitchin}
Let $C$ be a smooth projective curve and 
\begin{align*}
X=\mathrm{Tot}(\oO_C \oplus K_C)
\end{align*}
a non-compact CY 3-fold. 
Let $C \subset X$ be the zero section, and 
take the curve class $\beta=r[C]$. 
Then the moduli space $\Sh_{\beta}(X)$ is isomorphic to 
the product of $\mathbb{A}^1$ with 
Hitchin moduli space of rank $r$
and Euler characteristic one stable Higgs bundles on $C$. 
In this case, 
by the work of Chuang-Diaconescu-Pan~\cite{CDP2}, 
Conjecture~\ref{GV:conj1} (or rather its refined version)
is reduced to the following conjectures: 
\begin{itemize}
\item Cataldo-Hausel-Migliorini's $P=W$ conjecture~\cite{CHM}:
it claims the perverse filtration on the
Hitchin moduli space coincides with the weight filtration
of the character variety
under their natural diffeomorphism. 
Here the character variety 
is the moduli space of representations of 
$\pi_1(C)$. 

\item Hausel-Rodriguez-Villegas's conjecture~\cite{HRV}:
it 
describes the generating series of mixed Poincar\'e
polynomials of character varieties in terms of 
explicit sum of rational functions associated to 
Young diagrams. 
\end{itemize}
The first conjecture is proven in~\cite{CHM} when $r=2$; for character varieties, the weight
polynomial (ignoring cohomological degree) is calculated in \cite{HRV}. 
By combining these results, Conjecture~\ref{GV:conj1} holds for 
$\beta=2[C]$. 
By applying the fiberwise $(\mathbb{C}^{\ast})^{\times 2}$-action 
on $X$ and localizing, we also obtain 
Conjecture~\ref{GV:conj2} for the one cycle $\gamma=2[C]$.

\appendix
\addcontentsline{toc}{section}{}
\section{Calabi-Yau orientation data}\label{sec:append}
In this Appendix, we discuss the 
role of orientation data 
on our definition of GV type 
invariants in Section~\ref{subsec:GVtype}. 
In general, they depend on 
a choice of an orientation data. 
Our idea to solve this issue is that we impose 
an additional condition (called CY condition)
of an orientation data,
and show that the resulting GV type invariants 
are independent of 
an orientation data as long as it satisfies the 
CY condition. 

Roughly speaking, a CY condition of an orientation 
data is that it is trivial along the fibers of 
the map $\pi \colon M^{\rm{red}} \to T$
as a line bundle. 
This is a quite strong restriction, and
such an orientation data 
does not always exist for arbitrary $d$-critical schemes. 
However for the moduli space of one dimensional 
sheaves and its HC map, we expect that such an orientation data 
exists. In other words, we expect that the
HC map for the moduli space of one dimensional sheaves
is a kind of Calabi-Yau fibration in a somewhat virtual sense.

\subsection{Dependence on orientation data}
Let us consider the situation in Section~\ref{subsec:GVtype},
i.e. $\mM=(M, s, K_{M, s}^{1/2})$ be an oriented $d$-critical 
scheme and $\pi \colon M^{\rm{red}} \to T$
be a projective morphism 
for a finite type complex scheme $T$. 
For $g \ge 0$ and $t \in T$, let  
\begin{align}\label{append:gvtype}
\GV_{g, \mM/T} \in \mathbb{Z}, \ 
\GVloc_{g, \mM/T, t} \in \mathbb{Z}
\end{align}
be the GV type invariants
given in Lemma~\ref{lem:ng}, Lemma~\ref{lem:ngloc}
respectively. 
For $g=0$, they do not depend on a choice of 
an orientation data $K_{M, s}^{1/2}$
(see Lemma~\ref{lem:g=0}). 
For $g\ge 1$, they may depend on a choice of an 
orientation data. 
However we have the following lemma: 
\begin{lem}\label{lem:orient}
Let $K_{M, s}^{'1/2}$
be another orientation of $(M, s)$. 
Suppose that for a Stein factorization 
\begin{align}\label{stein}
\pi \colon M^{\rm{red}} \stackrel{\pi_1}{\to} 
\overline{T}
\stackrel{\pi_2}{\to} T
\end{align}
i.e. $\overline{T}=\Spec_T(\pi_{\ast}\oO_{M^{\rm{red}}})$, 
there exists a line bundle $L_{\overline{T}}$ on $\overline{T}$
such that 
\begin{align*}
K_{M, s}^{'1/2} \cong 
K_{M, s}^{1/2} \otimes \pi_1^{\ast}L_{\overline{T}}
\end{align*}
as line bundles. 
Then for $\mM'=(M, s, K_{M, s}^{'1/2})$
and $t \in T$, 
we have 
\begin{align*}
\GV_{g, \mM/T}=\GV_{g, \mM'/T}, \ 
\GVloc_{g, \mM/T, t}=\GVloc_{g, \mM'/T, t}. 
\end{align*}
\end{lem}
\begin{proof}
The isomorphism (\ref{isom:orient})
and the similar isomorphism for $K_{M, s}^{'1/2}$
gives an isomorphism 
$s \colon 
\pi_1^{\ast}L_{\overline{T}}^{\otimes 2} 
\stackrel{\cong}{\to} \oO_{M^{\rm{red}}}$.
Since $\pi_1$
satisfies $\pi_{1\ast}\oO_{M^{\rm{red}}}=\oO_{\overline{T}}$, 
the isomorphism $s$ is pulled 
back from an isomorphism 
$s' \colon L_{\overline{T}}^{\otimes 2} \stackrel{\cong}{\to}
\oO_{\overline{T}}$
via $\pi_1^{\ast}$. 
Let $\tau_T \colon \widetilde{T} \to \overline{T}$ be the   
$\mathbb{Z}/2\mathbb{Z}$-principal bundle 
which parametrizes local square roots of $s'$, and 
$\lL_{\overline{T}}$ the rank one local system on $\overline{T}$
given by $\tau_{T\ast}\mathbb{Q}_{\widetilde{T}}=
\mathbb{Q}_{\overline{T}} \oplus \lL_{\overline{T}}$. 
Let $\phi_{\mM'}$ be the vanishing cycle sheaf
on $M$
in Theorem~\ref{thm:BDJS}
defined from the 
oriented $d$-critical scheme $\mM'$. 
Then
the property (\ref{isom:IC})
shows that $\phi_{\mM'}=\phi_{\mM} \otimes \pi_1^{\ast}\lL_{\overline{T}}$. 
Also since $\pi_2$ is finite, $\dR \pi_{2\ast}=\pi_{2\ast}$ 
takes perverse sheaves on $\overline{T}$
to perverse sheaves on $T$.  
Therefore we obtain the isomorphisms
\begin{align*}
\pH^{i}(\dR \pi_{\ast}\phi_{\mM'})
\cong \pi_{2\ast}\left( \pH^{i}(\dR \pi_{1\ast}\phi_{\mM}) \otimes 
\lL_{\overline{T}} \right). 
\end{align*}
The lemma follows from the above
isomorphisms. 
\end{proof}

\subsection{Calabi-Yau $d$-critical schemes}
We introduce the following notion of CY $d$-critical schemes:
\begin{defi}\label{loc:cy}
(i)  
A $d$-critical scheme $(M, s)$ is called 
Calabi-Yau (CY for short)
 if $K_{M, s} \cong \oO_{M^{\rm{red}}}$. 

(ii)
An oriented $d$-critical scheme 
$(M, s, K_{M, s}^{1/2})$ is called 
CY if $K_{M, s}^{1/2} \cong \oO_{M^{\rm{red}}}$. 

(iii) 
 A $d$-critical scheme $(M, s)$
with a projective morphism $\pi \colon M^{\rm{red}} \to T$
is called CY at
$t \in T$ if 
there is an open neighborhood $t \in U \subset T$
such that, by setting $M_U^{\rm{red}} \cneq \pi^{-1}(U)$,
 the $d$-critical scheme
\begin{align}\label{notation:MU}
(M_U, s_U), \ 
M_U \cneq \iota(M_U^{\rm{red}}), \ s_U \cneq s|_{M_U}
\end{align}
is CY. Here $\iota \colon M^{\rm{red}} \hookrightarrow M$ is the 
closed immersion.

(iv) A $d$-critical 
scheme $(M, s)$ with a projective morphism 
$\pi \colon M^{\rm{red}} \to T$
is called a CY fibration over $T$
 if and only if 
it is CY at all of $t \in T$. 
\end{defi}

The following lemma is obvious. 
\begin{lem}
For a $d$-critical 
scheme $(M, s)$ with a projective morphism
$\pi \colon M^{\rm{red}} \to T$, 
it is a CY fibration over $T$ if and only if 
there is a line bundle $L$ on $\overline{T}$
for the Stein factorization (\ref{stein}) 
such that $K_{M, s} \cong \pi_1^{\ast}L$
\end{lem}

We introduce the following notion of 
CY orientation data: 
\begin{defi}\label{defi:CYorient}
For a $d$-critical 
scheme $(M, s)$ with a projective morphism
$\pi \colon M^{\rm{red}} \to T$, 
suppose that it is a CY fibration over $T$. 
A CY orientation data of $(M, s)$ is an orientation 
data 
$K_{M, s}^{1/2}$
satisfying 
$K_{M, s}^{1/2} \cong \pi_1^{\ast} L^{1/2}$
for a line bundle $L^{1/2}$ on $\overline{T}$.
Here 
$M^{\rm{red}} \stackrel{\pi_1}{\to} \overline{T} \to T$ is the 
Stein factorization. 
\end{defi}

By Lemma~\ref{lem:orient}, we immediately have the following lemma: 
\begin{lem}\label{lem:gvcompare}
The GV type invariants 
(\ref{append:gvtype}) are independent of 
a choice of an orientation data as long as it is CY 
orientation data. 
\end{lem}

\begin{rmk}\label{rmk:loc/global}
If $\pi \colon M^{\rm{red}} \to T$ is a CY fibration, it is not 
a priori true that it always has a CY orientation data 
as in Definition~\ref{defi:CYorient}. 
However of course such an orientation data always exists
locally on $T$. Therefore using
a local CY orientation data, 
we can define the local GV type invariant
$\GVloc_{g, \mM/T, t} \in \mathbb{Z}$. 
Then following the relation (\ref{int:id}), 
we can define the global GV type invariant 
by the integration
\begin{align*}
\GV_{g, \mM/T} \cneq \int_{T} \GVloc_{g, \mM/T, t} \ de. 
\end{align*}
\end{rmk}

\subsection{Conjecture on Calabi-Yau conditions}
We keep the situation and notation from the previous subsection. 
We conjecture
that the GV invariants are 
always well-defined: 
\begin{conj}\label{conj:GV1}
The $d$-critical scheme $(\Sh_{\beta}(X), s)$ in 
Theorem~\ref{thm:CYM}
is a CY fibration over $\Chow_{\beta}(X)$
so that the local/global GV invariants $n_{g, \gamma}^{\rm{loc}}$, 
$n_{g, \beta}$ are defined
by Definition~\ref{def:locgv}, Remark~\ref{def:global}
respectively. 
\end{conj}
We have the following evidence 
of the above conjecture: 
\begin{prop}\label{prop:Ktriv}
Let $T$ be a normal quasi-projective 
variety and $\fF \in \Coh(X \times T)$
be a $T$-flat family of 
one dimensional sheaves on $X$ 
such that the fundamental 
cycle $[\fF_t] \in \Chow(X)$ 
for $t \in T$ is constant. 
Then we have 
\begin{align*}
\det \left(\dR p_{T\ast} \dR \hH om_{X \times T}(\fF, \fF)
\right) \cong \oO_T. 
\end{align*}
\end{prop}
\begin{proof}
For a smooth quasi-projective variety $Y$, 
let
\begin{align*}
K^{\ge i}(Y) \subset K(Y)
\end{align*}
be
the subgroup
generated by 
sheaves whose supports have codimensions bigger than or
equal to $i$.
By a classical result of
Grothendieck (see~\cite[Theorem~3.10]{HG}), the 
smoothness of $Y$ implies that the
tensor product on the K-theory restricts to the map
\begin{align}\label{Ktheory}
\otimes \colon K^{\ge i}(Y) \times K^{\ge j}(Y)
\to K^{\ge i+j}(Y). 
\end{align}
For a normal variety $T$, 
any line bundle on it is determined by 
its smooth part, so we may assume that $T$ is smooth. 
Let $C \subset X$ be a subscheme whose 
fundamental cycle coincides with $[\fF_t]$. 
By the property (\ref{Ktheory}), 
we have 
\begin{align}\label{OcOc}
[\oO_C] \otimes [\oO_C] \in K^{\ge 4}(X)=0. 
\end{align}
Also by the assumption,
we have $[\fF], [\fF]^{\vee} \in K^{\ge 2}(X \times T)$ 
and 
\begin{align}\label{FOCT}
[\fF], [\fF]^{\vee}
 \in [\oO_{C \times T}]+ K^{\ge 3}(X \times T). 
\end{align}
By (\ref{Ktheory}), (\ref{OcOc}) and (\ref{FOCT}), 
we have 
\begin{align*}
[\dR \hH om_{X \times T}(\fF, \fF)]=
[\fF]^{\vee} \otimes [\fF] \in K^{\ge 5}(X \times T). 
\end{align*}
Since $X$ is three-dimensional, it follows that 
\begin{align*}
[\dR p_{T\ast} \dR \hH om_{X \times T}(\fF, \fF)]
\in K^{\ge 2}(T). 
\end{align*}
By taking the determinant, we obtain the proposition.
\end{proof}
\begin{rmk}
The above proposition 
in particular implies that the virtual 
canonical line bundle of $\Sh_{\beta}(X)$
is numerically 
trivial 
on any fiber of
the HC map 
\begin{align*}
\pi \colon \Sh_{\beta}^{\rm{red}}(X) \to \Chow_{\beta}(X)
\end{align*}
and 
trivial
on any fiber of $\pi$ with at worst normal singularities. 
These are necessary conditions for Conjecture~\ref{conj:GV1}. 
When $\Chow_{\beta}(X)$ is a one point,
the above proposition implies Conjecture~\ref{conj:GV1}
if $\Sh_{\beta}^{\rm{red}}(X)$ is normal. 
On the other hand if $\Sh_{\beta}^{\rm{red}}(X)$ is not normal, 
the argument of the above proposition does not imply 
Conjecture~\ref{conj:GV1}. 
\end{rmk}

\begin{rmk}
We may ask a question whether $(\Sh_{\beta}(X), s)$
is strictly CY at any point in $\Chow_{\beta}(X)$ 
(see Definition~\ref{defi:slcy}), which is stronger 
than Conjecture~\ref{conj:GV1}. 
In Section~\ref{sec:local}, we more or less proved 
such a statement for the local surface case. 
As we have no other evidence, we just leave it 
as a question (rather than a conjecture) in this paper. 
\end{rmk}

\newcommand{\etalchar}[1]{$^{#1}$}
\providecommand{\bysame}{\leavevmode\hbox to3em{\hrulefill}\thinspace}
\providecommand{\MR}{\relax\ifhmode\unskip\space\fi MR }
\providecommand{\MRhref}[2]{%
  \href{http://www.ams.org/mathscinet-getitem?mr=#1}{#2}
}
\providecommand{\href}[2]{#2}


\vspace{5mm}

Davesh Maulik \\
Massachusetts Institute of Technology, 
 Departement of Mathematics, 
 77 Massachusetts Avenue
 Cambridge, MA 02139, US.

\textit{E-mail address}: maulik@mit.edu

\vspace{5mm}

Yukinobu Toda \\
Kavli Institute for the Physics and 
Mathematics of the Universe, University of Tokyo,
5-1-5 Kashiwanoha, Kashiwa, 277-8583, Japan.

\textit{E-mail address}: yukinobu.toda@ipmu.jp

\end{document}